\documentclass[10pt]{article}

\usepackage[papersize={21cm,29.7cm}, top = 3cm, bottom = 3cm]{geometry}
\usepackage{titlesec}
\titleformat*{\section}{\large\bfseries}
\titleformat*{\subsection}{\bfseries}
\titleformat*{\subsubsection}{\bfseries}
\titleformat*{\paragraph}{\bfseries}
\titleformat*{\subparagraph}{\bfseries}
\usepackage{amsthm}
\usepackage{amssymb}
\usepackage{amsmath}
\numberwithin{equation}{subsection}
\usepackage[backend = biber]{biblatex}
\usepackage{bbm}
\usepackage[english]{babel}
\usepackage[latin1]{inputenc}
\usepackage[T1]{fontenc}
\usepackage[dvipsnames]{xcolor}
\usepackage{hyperref}
\hypersetup{
    colorlinks = true,
    linkcolor = Blue,
    citecolor = Blue,
    urlcolor = Blue
    }
\usepackage{ifpdf}
\ifpdf
  \usepackage[pdftex]{graphicx}
  \usepackage{epstopdf}
\else
  \usepackage[dvips]{graphicx}
\fi
\usepackage{graphicx}
\usepackage{cleveref}
\usepackage[all]{xy}
\usepackage{csquotes}
\usepackage{dynkin-diagrams}
\usepackage{etoolbox}
\usepackage{mathrsfs}
\usepackage{braket}

\makeatother

\theoremstyle{plain}
\newtheorem{theorem}{Theorem}[section]

\newtheorem{lemma}[theorem]{Lemma}
\newtheorem{proposition}[theorem]{Proposition}
\newtheorem{corollary}[theorem]{Corollary}
\newtheorem{question}[theorem]{Question}

\newtheorem{conjecture}[theorem]{Conjecture}

\theoremstyle{definition}
\newtheorem{definition}[theorem]{Definition}
\newtheorem{setting}[theorem]{Setting}

\newtheorem{convention}[theorem]{Convention}
\newtheorem{notation}[theorem]{Notation}
\newtheorem{remark}[theorem]{Remark}

\theoremstyle{remark}
\newtheorem{example}[theorem]{Example}

\newcommand{\lie}{\mathrm{Lie}}
\newcommand{\I}{\mathscr{I}}
\newcommand{\F}{\mathscr{F}}
\newcommand{\alg}{\mathbb{R}_\mathrm{alg}}
\newcommand{\anexp}{\mathbb{R}_\mathrm{an,exp}}

\newcommand{\N}{\mathbb N}
\newcommand{\Z}{\mathbb Z}
\newcommand{\uZ}{\underline{\mathbb{Z}}}
\newcommand{\Q}{\mathbb Q}
\newcommand{\uQ}{\underline{\mathbb{Q}}}
\newcommand{\C}{\mathbb C}
\newcommand{\R}{\mathbb R}

\newcommand{\V}{\mathbb V}
\newcommand{\W}{\mathbb W}

\newcommand{\s}{\mathbb S}
\newcommand{\sm}{\mathrm{sm}}

\newcommand{\mult}{\mathbb G}
\newcommand{\G}{\mathbf G}
\newcommand{\pp}{\mathbf P}
\newcommand{\uu}{\mathbf U}

\newcommand{\mon}{\mathbf H}
\newcommand{\M}{\mathbf M}
\newcommand{\LL}{\mathbf L}
\newcommand{\NN}{\mathbf N}

\newcommand{\fix}{\mathbf{Fix}}
\newcommand{\stab}{\mathbf{Stab}}

\newcommand{\HL}{\mathrm{HL}(S, \V^\otimes)}
\newcommand{\HLM}{\mathrm{HL}(S, \V^\otimes, \M)}

\newcommand{\bul}{\bullet}

\newcommand{\im}{\mathrm{im} \hspace{0.05cm}}
\newcommand{\id}{\mathrm{id}}

\newcommand{\an}{\mathrm{an}}
\newcommand{\gr}{\mathrm{Gr}}
\newcommand{\hod}{\mathrm{Hdg}}

\newcommand{\Hom}{\mathrm{Hom}}

\newcommand{\quo}{\backslash}

\newcommand{\gl}{\mathbf{GL}}

\newcommand{\res}{\mathrm{Res}}

\newcommand{\para}{\mathfrak{p}}
\newcommand{\lieu}{\mathfrak{u}}

\newcommand{\mt}{\mathfrak{m}}
\newcommand{\ort}{\mathfrak{h}}

\newcommand{\Ad}{\mathrm{Ad}}
\newcommand{\ad}{\mathrm{ad}}

\newcommand{\period}{\mathscr{D}}

\newcommand{\typ}{\mathrm{typ}}
\newcommand{\atyp}{\mathrm{atyp}}
\newcommand{\pos}{\mathrm{pos}}
\newcommand{\der}{\mathrm{der}}
\newcommand{\trans}{\mathrm{trans}}

\newcommand{\End}{\mathrm{End}}
\newcommand{\Gr}{\mathrm{Gr}}

\newcommand{\ee}{e_{-2}}
\newcommand{\zz}{\underline{z}}
\newcommand{\xxi}{\underline{\xi}}
\newcommand{\Np}{\mathcal{N}(\pp)}

\newcommand{\Del}{\mathrm{Del}}
\newcommand{\GL}{\mathrm{GL}}
\newcommand{\ppol}{\pp_{W_\bul, (q_k)}}
\newcommand{\gpol}{\G_{W_\bul, (q_k)}}
\newcommand{\ex}{\mathrm{ex}}
\newcommand{\gen}{\mathrm{gen}}

\newcommand{\zetaa}{\underline{\zeta}}
\newcommand{\lambdaa}{\underline{\lambda}}

\newcommand{\xx}{\underline{x}}
\newcommand{\Zcal}{\mathscr{Z}}
\newcommand{\deef}{\mathrm{def}}

\newcommand{\intt}{\mathrm{int}}

\newcommand{\ortl}{\mathfrak{l}}

\title{On the distribution of mixed Hodge loci}
\author{Nazim Khelifa}
\date{\today}

\begin{document}
\maketitle
\begin{abstract}
Let $\V$ be an admissible and graded-polarized integral variation of mixed Hodge structures over a smooth and irreducible complex algebraic variety $S$. We show that if the typical Hodge locus $\HL_\typ$ of $\V$ is non-empty, the full Hodge locus $\HL$ is dense in $S$ for the Zariski topology. In an other direction, we show that if the associated graded variation $\gr(\V)$ for the weight filtration has large monodromy and level at least $3$ in the sense of Baldi-Klingler-Ullmo, the typical Hodge locus of $\V$ is empty, and the full Hodge locus of $\V$ is a strict Zariski-closed subset of $S$, at least if one restricts to its factorwise positive dimensional part, improving a classical result of Brosnan-Pearlstein-Schnell in this situation.

These results follow from a detailed study of the transverse part $\HL_\trans$ of the Hodge locus of $S$ for $\V$, a subset which contains $\HL_\typ$ and whose Zariski-density in $S$ is equivalent, under the Zilber-Pink conjecture for $\V$, to the Zariski-density of $\HL_\typ$. We show that non-emptiness of $\HL_\trans$ is equivalent to its Zariski-density in $S$, we completely classify variations whose transverse Hodge locus $\HL_\trans$ is Zariski-dense, and we prove an independent criterion ensuring that $\HL_\trans$ is empty.
\end{abstract}
\begin{center}
\tableofcontents
\end{center}
\section{Introduction}
Let $\V$ be an admissible and graded-polarized integral variation of mixed Hodge structures over a smooth and irreducible quasi-projective complex algebraic variety $S$. Our main object of study will be the Hodge locus $\HL$ of $S$ for $\V$. It is the set of points $s \in S^\an$ such that some rational tensor $\zeta_s \in T^{a,b}(\V_{s,\Q}) = \bigoplus_{a,b \geqslant 0} (\V_{s,\Q})^{\otimes a} \otimes (\V_{s,\Q}^\vee)^{\otimes b}$ satisfies:
\begin{itemize}
\item $\zeta_s$ is a Hodge class for the natural $\Q$-mixed Hodge structure on $T^{a,b}(\V_{s,\Q})$;
\item there exists a continuous path $\tau : [0,1] \rightarrow S^\an$ such that $\tau(0) = s$ and the parallel transport of $\zeta_s$ to $T^{a,b}(\V_{\tau(1),\Q})$ along $\tau$ is not a Hodge class for the natural mixed Hodge structure on $T^{a,b}(\V_{\tau(1),\Q})$.
\end{itemize}
A strong motivation for the study of the Hodge locus is Beilinson's Hodge conjecture for higher cycles (\cite{beil}) which, for a variation arising as the cohomology of a family of smooth quasi-projective varieties, predicts that $\HL$ is the locus where some self-product of the corresponding fiber carries a higher cycle that does not spread over the generic fiber (see for instance \cite{as2, as1} and references therein).

\textit{A priori} the Hodge locus is merely a countable union of strict complex-analytic subvarieties of $S$. A fundamental fact due to Brosnan-Pearlstein-Schnell in this generality (building on and generalizing a classical result of Cattani-Deligne-Kaplan \cite{cdk} and a series of work \cite{bpan,bpduk,bpcomp} of the first two authors) is:
\begin{theorem}[\cite{bps}]
The Hodge locus $\HL$ is a countable union of strict algebraic subvarieties of $S$.
\end{theorem}
The present paper belongs to a line of works aiming to go one step further in the study of the distribution of Hodge locus in $S$, building on the intersection theoretic description of $\HL$ proposed in \cite{klingler_atyp}.
\subsection*{The Zilber-Pink program and first results}
In \cite{klingler_atyp}, Klingler described $\HL$ as a countable union of intersection loci, using the period map associated to $\V$ (see. Sect. \ref{sect4} for an overview). Building on this, he exhibited a decomposition of the Hodge locus
\[
\HL = \HL_\typ \cup \HL_\atyp
\]
into two parts, the typical and atypical part whose distribution is predicted by the:
\begin{conjecture}[Zilber-Pink conjecture]\label{zpintr}
Let $\V$ be an admissible and graded-polarized integral variation of mixed Hodge structures over a smooth and irreducible quasi-projective complex algebraic variety $S$. Then
\begin{itemize}
\item[(a)] the atypical Hodge locus $\HL_\atyp$ is a strict Zariski-closed subset of $S$;
\item[(b)] if the typical Hodge locus $\HL_\typ$ is non-empty it is dense in $S$ for the Zariski topology.
\end{itemize}
\end{conjecture}
\begin{remark}
We will actually work with a refined definition of $\HL_\typ$ and $\HL_\atyp$ due to \cite[Def. 2.4, Conj. 2.5]{bku} for pure variations and \cite[Def. 4.8, Conj. 4.10]{bu} in general. Part $(a)$ of the conjecture was first stated in \cite[Conj. 1.10]{klingler_atyp} and Part $(b)$ is an analog of \cite[Conj. 2.7]{bku} in the setting of variations of mixed Hodge structures. As we shall show later, it (trivially) cannot hold for variations of mixed Hodge structures with density for the Zariski topology replaced by density for the euclidean topology as in the case of pure variations. 
\end{remark}
This work is dedicated to studying part $(b)$ of the Zilber-Pink conjecture, partially generalizing results about pure variations in \cite{bku, TT, es, ku, ketay} to mixed variations. Our first main result is an analog of the all-or-nothing theorem \cite[Thm. 3.9]{bku} for pure variations:
\begin{theorem}\label{aon}
Assume that $\HL_\typ$ is non-empty. Then the full Hodge locus $\HL$ is Zariski-dense in $S$.
\end{theorem}
In particular, this yields:
\begin{corollary}\label{zpimp}
Conjecture \ref{zpintr}$(a)$ implies Conjecture \ref{zpintr}$(b)$.
\end{corollary}
Building on Baldi-Urbanik's Geometric Zilber-Pink theorem (\cite[Thm. 7.1]{bu}), we deduce from Theorem \ref{aon} the statement of Conjecture \ref{zpintr}$(ii)$ for the typical Hodge locus of factorwise positive dimension (as defined below in Definition \ref{posperdim}):
\begin{corollary}\label{aontypos}
The typical Hodge locus of factorwise positive dimension $\HL_{\typ, \mathrm{pos}}$ is non-empty if and only if it is Zariski-dense in $S$.
\end{corollary}
\subsection*{The transverse Hodge locus and the all-or-nothing property}
The above described all-or-nothing property follows from an analogous and more precise one for the transverse Hodge locus. To define the latter let us introduce some language for which we refer the non-expert reader to Sect. \ref{sect2} and \ref{sect3}. 

Let $(\pp, D_\pp)$ be the generic mixed Hodge datum\footnote{Strictly speaking there is no canonically defined generic mixed Hodge datum, as we will explain in the body of the manuscript. We will introduce an equivalence relation between mixed Hodge data and work with equivalence classes, called mixed Hodge classes. For simplicity, we will ignore this technicality in this introduction.} of $\V$ and $\Gamma \subset \pp(\Q)^+$ be an arithmetic lattice containing the image of the monodromy representation of $\V$. For simplicity, we will assume in this introduction that $\Gamma$ can be chosen neat. Let $\period = \Gamma \quo D_\pp$ be the associated Hodge variety and $\Phi : S^\an \rightarrow \period$ the period map of $\V$. For any strict mixed Hodge subdatum $(\M, D_\M) \subsetneq (\pp, D_\pp)$ we denote by $\period_{(\M, D_\M)}$ the corresponding Hodge subvariety which is a locally closed complex-analytic subset of $\period$.
\begin{definition}[Definition \ref{deftrans}]
Let $(\M, D_\M) \subsetneq (\pp, D_\pp)$ be a strict mixed Hodge subdatum. 
\begin{itemize}
\item A \textit{$(\M, D_\M)$-transverse special subvariety of $S$ for $\V$} is a complex-analytic irreducible component $Z$ of the preimage $\Phi^{-1}(\period_{(\M, D_\M)})$ such that:
\begin{itemize}
\item[(a)] $\Phi(Z^\an)$ is not contained in the singular locus of $\Phi(S^\an)$;
\item[(b)] $\dim \Phi(Z^\an) = \dim \Phi(S^\an) + \dim \period_{(\M, D_\M)} - \dim \period$.
\end{itemize}
\item The \textit{transverse Hodge locus of type $\M$ of $\V$} is the union $\HLM_\trans$ of $(p\M p^{-1}, p \cdot D_\M)$-transverse special subvarieties of $S$ for $\V$ as $p$ ranges over $\pp(\Q)^+$. 
\item \textit{The transverse Hodge locus of $\V$} is the union $\HL_\trans$ of all $(\LL,D_\LL)$-transverse special subvarieties of $S$ for $\V$ over the set of strict mixed Hodge subdata $(\LL, D_\LL) \subsetneq (\pp, D_\pp)$.
\end{itemize}
\end{definition}
The relevance of these definitions for the study of the typical Hodge locus lies in the following two facts (see Remark \ref{typtrans}):
\begin{itemize}
\item $\HL_\typ \subset \HL_\trans$
\item if Conjecture \ref{zpintr}$(a)$ holds, then $\HL_\trans$ is Zariski-dense in $S$ if and only if $\HL_\typ$ is.
\end{itemize}
Theorem \ref{aon} is then an immediate consequence of the following property of transverse Hodge loci which is our first main result:
\begin{theorem}[Theorem \ref{allornothing}]\label{aontrans}
Let $(\M, D_\M) \subsetneq (\pp, D_\pp)$ be a strict mixed Hodge subdatum. The transverse Hodge locus $\HLM_\trans$ of type $\M$ is non-empty if and only if it is Zariski-dense in $S$.
\end{theorem}
With this result in hand, a natural goal to pursue is that of understanding when each possibility of the above alternative occurs. The remaining part of this work is dedicated to this problem.
\subsection*{Likeliness and density of the Hodge locus}
Let $(\M, D_\M) \subsetneq (\pp, D_\pp)$ be a strict mixed Hodge subdatum and assume that $\HLM_\trans$ is non-empty. By definition of transverse special subvarieties, this forces $(\M,D_\M)$ to satisfy the dimension inequality
\[
\dim \Phi(S^\an) + \dim \period_{(\M,D_\M)} - \dim \period \geqslant 0.
\]
As we shall show, there are actually much more constraints in general: one such inequality for every factor. More precisely, for any normal subgroup $\NN$ of $\pp$ whose radical is unipotent, the datum $(\M,D_\M)$ must satisfy
\begin{equation}
\tag{$\ast_\NN$}
\dim \Phi_{/\NN}(S^\an) + \dim p_\NN(\period_{(\M, D_\M)}) - \dim \period/\NN \geqslant 0.
\label{vlikN}
\end{equation}
where $p_\NN : \period \rightarrow \period/\NN$ is the quotient-by-$\NN$ (holomorphic surjective) map and $\Phi_{/\NN} = p_\NN \circ \Phi$.
\begin{definition}
A strict mixed Hodge subdatum $(\M, D_\M) \subsetneq (\pp, D_\pp)$ is called $\V$-likely if it satisfies the inequality (\ref{vlikN}) for every normal subgroup $\NN$ of $\pp$ whose radical is unipotent.
\end{definition}
The distribution of the transverse Hodge locus will depend on the weights that occur in $\V$, hence we will need the following notation:
\begin{notation}
Any $h \in D_\pp$ defines a mixed $\Q$-Hodge structure on $\lie(\pp)$ with non-positive weights only. We will denote by $W_\bul \lie(\pp)$ the corresponding weight filtration. For integers $p,q \in \Z$ with $p+q \leqslant 0$, let $h_{\lie(\pp)}^{p,q}$ the Hodge numbers of the weight $p+q$ pure Hodge structure on $\gr^W_{p+q}(\lie(\pp))$ defined by $h$. The filtration and Hodge numbers do not depend on $h$ hence the omission of $h$ in the notation. 
\end{notation}
Our second main results states that if $\gr^W_{-2}(\lie(\pp)) = 0$, the $\V$-likeliness of a strict mixed Hodge subdatum $(\M, D_\M) \subsetneq (\pp, D_\pp)$ ensures that $\HLM_\trans$ is dense in $S$ \textit{for the euclidean topology}:
\begin{theorem}[Theorem \ref{mainprescribed}]\label{thm1}
Let $(\M, D_\M) \subsetneq (\pp, D_\pp)$ be a strict mixed Hodge subdatum. Assume that $\gr^W_{-2}(\lie(\pp)) = 0$. The following assertions are equivalent:
\begin{itemize}
\item[(i)] $\HLM_\trans$ is non-empty;
\item[(ii)]$\HLM_\trans$ is dense in $S^\an$ for the euclidean topology;
\item[(iii)]$(\M, D_\M)$ is $\V$-likely.
\end{itemize}
\end{theorem}
This in particular includes the analog statement in \cite{ku, es} for pure variations and density statements in the euclidean topology for the torsion locus of some normal functions as \cite[Thm. 1.5]{ketay}. In general however density for the euclidean topology may fail for the transverse Hodge locus of some prescribed $\V$-likely type as we demonstrate in an example:
\begin{proposition}[Proposition \ref{contrexc}]\label{pr1}
There exists a smooth and irreducible quasi-projective complex curve $C$ and an admissible integral graded-polarized variation of mixed Hodge structures $\V$ on $C$ with generic mixed Hodge datum $(\pp,D_\pp)$ such that:
\begin{itemize}
 \item[(i)] There exists a $\V$-likely strict mixed Hodge subdatum $(\M,D_M)$ of $(\pp, D_\pp)$ and for every such mixed Hodge subdatum:
\begin{itemize}
\item[($i_1)$] the transverse Hodge locus $\mathrm{HL}(C, \V^\otimes, \M)_\trans$ of type $\M$ is dense in $C(\C)$ for the Zariski topology;
\item[$(i_2)$] the transverse Hodge locus $\mathrm{HL}(C, \V^\otimes, \M)_\trans$ of type $\M$ is not dense in $C(\C)$ for the euclidean topology;
\end{itemize}
\item[(ii)] The full transverse Hodge locus $\mathrm{HL}(C, \V^\otimes)_\trans$ is dense in $C(\C)$ for the euclidean topology.
\end{itemize}
\end{proposition}
This result leaves open the question of knowing wether the existence of a $\V$-likely strict mixed Hodge subdatum of $(\pp, D_\pp)$ implies in general that the full transverse Hodge locus $\HL_\trans$ is dense in $S^\an$ for the euclidean topology. There is one obvious counterexample (see Proposition \ref{contrexc1} for more details):
\begin{example}[Proposition \ref{contrexc1}]\label{ex1}
Let $S = \mathbb{A}^1 - \{0\}$. There is a natural extension $\V$ of the constant variation $\underline{\Z}_S(0)$ by the constant variation $\underline{\Z}_S(1)$ in the category of integral variations of mixed Hodge structures which is admissible. We refer to Sect. \ref{sect6} for a general description of mixed Hodge-Tate variations, and in particular of their Hodge loci. An easy computation shows that $\HL_\trans = \HL = \mu_{\infty}$ where $\mu_{\infty} \subset S^\an = \C^\times$ is the set of roots of unity. In particular it is dense in $S(\C)$ for the Zariski topology but not for the euclidean topology.
\end{example}
In general, we can prove:
\begin{theorem}[Theorem \ref{mainfull}]\label{thm2}
The following assertions are equivalent:
\begin{itemize}
\item[(i)] $\HL_\trans$ is non-empty;
\item[(ii)]$\HL_\trans$ is dense in $S$ for the Zariski topology;
\item[(iii)]There exists a $\V$-likely strict mixed Hodge subdatum of $(\pp, D_\pp)$.
\end{itemize}
\end{theorem}
As we do not know any example other than Example \ref{ex1}  where the full transverse Hodge locus fails to be dense in $S^\an$ for the euclidean topology, the following question remains unsolved: 
\begin{question}\label{conjdensan}
Assume that $h_{\lie(\pp)}^{-1,-1} \neq 1$ and that there exists a $\V$-likely strict mixed Hodge subdatum of $(\pp, D_\pp)$. Is $\HL_\trans$ dense in $S^\an$ for the euclidean topology?
\end{question}
To finish this section, we highlight the following result obtained along the way and which might be of independent interest:
\begin{proposition}\label{triv-3}
Assume that $W_{-1} \lie(\pp) = W_{-3} \lie(\pp)$. Then the transverse Hodge locus $\HL_\trans$ of $S$ for $\V$ and the transverse Hodge locus $\mathrm{HL}(S, (\gr(\V)^\otimes))_\trans$ coincide.
\end{proposition}
\subsection*{Higher level and algebraicity of the Hodge locus}
In \cite{bku}, the authors showed that, when the variation is pure, Griffiths' transversality can force non-trivially the transverse Hodge locus to be empty and could deduce from their geometric Zilber-Pink theorem a finiteness statement for the full Hodge locus in many situations. More precisely, they introduce the following notion:
\begin{definition}[{\cite[Def. 4.9, Def. 4.13]{bku}}]
Let $\ort$ be a \textit{$\Q$-Hodge-Lie algebra} i.e a semi-simple $\Q$-Lie algebra endowed with a $\Q$-Hodge structure of weight $0$ polarized by the Killing form and such that the Lie bracket is a morphism of Hodge structures $[\cdot,\cdot] : \bigwedge^2 \ort \rightarrow \ort$. The \textit{level of $\ort$} is the largest non-negative integer $l \geqslant 0$ such that for each $\Q$-simple factor $\ort'$ of $\ort$, there exists a $\R$-simple factor $\ort'_1$ of the semi-simple $\R$-Lie algebra $\ort'_\R = \ort' \otimes_\Q \R$ whose induced Hodge structure has Hodge decomposition:
\[
(\ort'_1)_\C = \bigoplus_{-w \leqslant k \leqslant w} (\ort'_1)^{k, -k}
\]
with $(\ort'_1)^{-w, w} \neq 0$ and $w \geqslant l$.
\end{definition}
They show:
\begin{theorem}[{\cite[Thm. 3.3]{bku}}]
Assume that $\V$ is an integral and polarized variation of Hodge structures on $S$ and denote by $\mon$ its algebraic monodromy group. If the $\Q$-Hodge-Lie algebra $\lie(\mon)$ has level at least $3$ the transverse Hodge locus $\HL_\trans$ is empty.
\end{theorem}
Remarkably, this is proven independently of the criteria for non-emptiness such as Theorem \ref{thm1} for pure variations that can be found in \cite{es,ku}, and the confrontation of the two can yield non-trivial results (see for example \cite{kbounds}). Motivated by these results, the last part of this work is dedicated to investigating the extent to which the level of the (pure) graded $\gr(\V)$ of $\V$ for the weight filtration can constrain the transverse Hodge locus of the mixed variation $\V$. For instance, we show (see Theorems \ref{alternativempty} and \ref{emptymain} for a more refined statement):
\begin{theorem}[Corollary \ref{critnum}, Corollary \ref{critalg}]\label{thm3}
Let $\mon_0$ (resp. $\pp_0$) be the algebraic monodromy group (resp. the generic Mumford-Tate group) of $\gr(\V)$. Assume that the $\Q$-Hodge-Lie algebra $\lie(\mon_0)$ has level at least $3$, that $\pp_0^\der = \mon_0$ and that one of the following assumptions hold:
\begin{itemize}
\item[(1)] $h^{-1,-1}_{\lie(\pp)} = 0$
\item[(2)] $S$ has a smooth projective compactification $\overline{S}$ such that $\overline{S}-S$ is either empty, of codimension at least $2$ or a smooth irreducible divisor in $\overline{S}$;
\end{itemize}
Then $\HL_\trans$ is empty.
\end{theorem}
Let us comment on the assumptions (see Sect. \ref{sect10} for a more thorough discussion). As illustrated in \cite{bku} through examples, the transverse Hodge locus of $\gr(\V)$ (hence of $\V$) might be non-empty in levels $1$ and $2$, hence the assumption on the level. The other assumptions are meant to exclude the possibility of adding direct factors to $\V$ whose monodromy is unipotent, as this doesn't influence the level of $\gr(\V)$ while adding factors that are decorrelated to $\gr(\V)$. The assumption $\pp_0^\der = \mon_0$ is often satisfied (for example when $\pp_0^\der$ is $\Q$-simple and $\V$ does not have unipotent monodromy) and the content of the above result is that this assumption is enough to exclude the aforementionned pathological cases, except for some very specific unipotent factors in weight $-2$, studied in length in Sect. \ref{sect6} and which are excluded by assumptions $(1)$ or $(2)$. We do not know if assumptions $(1)$ and $(2)$ are really needed. Note however that they are of very different natures and make it reasonable to hope that such a result might be applicable in many concrete situations.

The above emptiness criterion yields a finiteness statement for special subvarieties of the following type:
\begin{definition}\label{posperdim}
An irreducible algebraic subvariety $Z \subset S$ has \textit{factorwise positive dimension} if for every normal subgroup $\NN$ of $\pp$ whose radical is unipotent,
\[
\dim \Phi_{/\NN}(Z^\an) > 0.
\]
 The union of special subvarieties of $S$ for $\V$ which have factorwise positive dimension will be denoted $\HL_{\mathrm{pos}}$.
\end{definition} 
Then, using Baldi-Urbanik's Geometric Zilber-Pink theorem (\cite{bu}) we prove the following complement to \cite[Thm. 3.2]{ketay}:
\begin{corollary}[Theorem \ref{alglvl3}]
Under the assumptions of Theorem \ref{thm3}, the Hodge locus of factorwise positive dimension $\HL_{\mathrm{pos}}$ is a Zariski-closed subset of $S$.
\end{corollary}
\begin{remark}
Having factorwise positive dimension consists in principle in infinitely many conditions. As we shall explain in Remark \ref{finposdim}, Baldi-Urbanik's theorem implies that all but (non-explicit) finitely many (depending only on $\V$) of these conditions are superfluous for atypical special subvarieties.
\end{remark}
\subsection*{Organization of the paper}
The cornerstone of the methods used in this paper is Klingler's realization in \cite{klingler_atyp} of special subvarieties as intersection loci, and we did not find any detailed account with proofs in the litterature of this realization, in particular no proof of the fact that mixed Mumford-Tate domains parametrize mixed Hodge structures with bounded Mumford-Tate group. In order to fill this gap in litterature and to be as self-contained as possible we included rather long preliminaries on Mumford-Tate groups of mixed Hodge structures and the associated Mumford-Tate domains, on variations of mixed Hodge structures and on their Hodge loci (Sect. \ref{sect2}-\ref{sect4}). Let us mention however that the reader who is only interested in mixed period domains can refer to the very good exposition in \cite[Sect. 2-4]{pearls}. We then proceed in Sect. \ref{sect5} with a quick review of weakly special subvarieties, and two central results in our proofs, namely the Ax-Schanuel and Geometric Zilber-Pink theorems for variations of mixed Hodge structures. These four sections can be safely skipped by experts.

In Sect. \ref{sect6}, we give an elementary description of families of extensions of $\Z(0)$ by $\Z(1)^n$ (for $n \geqslant 1$) and their Hodge loci. We then exhibit a concrete example realizing Proposition \ref{pr1}. In Sect. \ref{sect7} we study in depth the $\V$-likeliness condition and prove the core technical results on which our main results rely. In Sect. \ref{sect8} we prove the all-or-nothing principle of Theorem \ref{aontrans} and a more general form of the characterization of density of $\HL_\trans$ for the euclidean topology in Theorem \ref{thm1}. In Sect. \ref{sect9} we prove Theorem \ref{thm2}. Finally in Sect. \ref{sect10} we investigate emptiness of the transverse Hodge locus, algebraicity of the full Hodge locus, and prove generalizations of Theorem \ref{thm3}.
\subsection*{Conventions and notations}
\begin{itemize}
\item An algebraic variety is a (non-necessarily irreducible) reduced separated scheme of finite type over $\C$.
\item For a characteristic $0$ field $K$ and a $\Q$-vector space $V$ (resp. a $\Q$-algebraic group $\pp$), we denote by $V_K$ (resp. $\pp_K$) the object obtained by extension of scalars to $K$. If $V$ is a $K$-vector space, we denote by $V^\vee$ its $K$-dual.
\item If $k \in \Z$ and $A$ is an object of an abelian category, endowed with an ascending finite filtration $W_\bul$, we denote by $\gr^W_k(A)$ its $k$-th graded piece for $W_\bul$ and by $\gr^W(A)$ the coproduct of the non-trivial graded-pieces of $A$ for $W_\bul$. For mixed Hodge structures and variations of mixed Hodge structures, when the superscript indicating the filtration is omitted, it should be understood that we are considering the graded for the weight filtration.
\item $\Q$-algebraic groups will be denoted in bold letters ($\pp$, $\mon$, $\G$ etc.). For a $\Q$-algebraic group $\pp$, we denote by $\pp^u$ its unipotent radical, by $Z(\pp)$ its center, by $\pp^\ad$ its adjoint and by $\pp^\der$ its derived subgroup. We also denote by $\pp(\R)^+$ the connected component of the identity of the real Lie group $\pp(\R)$ and by $\pp(\Q)^+$ the intersection $\pp(\R)^+ \cap \pp(\Q)$. A ($\Q$-)representation of $\pp$ will mean a rational, algebraic and finite-dimensional representation of $\pp$. If $(V,q)$ is a finite-dimensional $\Q$-vector space endowed with a non-degenerate bilinear form, we denote by $\mathbf{GAut}(V,q)$ the $\Q$-algebraic group of linear automorphisms of $V$ preserving $q$ up to scaling.
\end{itemize}
\subsection*{Aknowledgements}
I thank my PhD advisor Emmanuel Ullmo for suggesting this project and for many enlightening discussions and suggestions which greatly helped to improve this work. I also thank Ziyang Gao for a useful discussion that led to Lemma \ref{subgroup}. I also thank Gregorio Baldi, David Urbanik and Ofer Gabber for pointing out a mistake in a previous version of the text and for helping resolve it. Finally, I would like to thank Matt Kerr for his careful reading of the text and for many valuable mathematical and bibliographical remarks.
\section{Mixed Hodge structures and their parameter spaces}\label{sect2}
\subsection{Mixed Hodge structures}
\subsubsection*{First definitions}
\begin{definition}
Let $n \in \Z$. A \textit{pure $\Q$-Hodge structure of weight $n$} is the datum of a finite dimensional $\Q$-vector space $V$ and of a decreasing finite filtration by $\C$-vector subspaces $F^\bul$ on $V_\C$, called the \textit{Hodge filtration}, such that $V_\C = F^p \oplus \overline{F^{n+1-p}}$ for all $p\in \Z$. A \textit{pure $\Z$-Hodge structure of weight $n$} is the datum of a torsion-free abelian group of finite rank $V$ and of a pure $\Q$-Hodge structure of weight $n$ on $V_\Q = V \otimes_\Z \Q$.
\end{definition}
Let $V$ be a finite dimensional $\Q$-vector space, let $\s = \res_{\C/\R}(\mult_{m,\C})$ be the Deligne torus which is a real algebraic torus split over $\C$ and $w : \mult_{m,\R} \rightarrow \s$ be the weight homomorphism defined on $\R$-points as the inclusion $\R^\times \subset \C^\times$. The datum of a pure $\Q$-Hodge structure of weight $n$ on $V$ is equivalent to that of
\begin{enumerate}
\item a bigraduation $V_\C = \bigoplus_{p+q = n} V^{p,q}$ called the \textit{Hodge decomposition} subject to the \textit{Hodge symmetry} condition $V^{p,q} = \overline{V^{q,p}}$ for the complex conjugation on $V_\C$ defined by the real form $V_\R$. The set $\{(p,q) \hspace{0.1cm} : \hspace{0.1cm} V^{p,q} \neq 0\} \subset \Z \times \Z$ is called the \textit{type} of the Hodge structure.
\item a morphism of real algebraic groups $x : \s \rightarrow \gl(V)_\R$, called the \textit{Hodge homomorphism} such that $x \circ w$ is given on real points by $\R^\times \ni r \mapsto r^n \cdot \mathrm{id}\in \mathrm{GL}(V_\R)$.
\end{enumerate}
\begin{example}
Let $r \in \Z$. There is a unique pure $\Z$ (resp. $\Q$)-Hodge structure of weight $-2r$ on $V = (2i\pi)^r\Z$ (resp $(2i\pi)^r\Q$) called the \textit{Tate Hodge structure} and it is denoted by $\Z(r)$ (resp. $\Q(r)$). It is of type $\{(-r,-r)\}$.
\end{example}
\begin{example}
Let $X$ be a smooth compact Kähler variety. For any $0 \leqslant n \leqslant 2\dim X$, the $n$-th cohomology group modulo torsion $H^n(X,\Z)$ carries a canonical $\Z$-Hodge structure of weight $k$.
\end{example}
Among compact Kähler varieties those which are actually smooth projective complex algebraic varieties are privileged, and characterized by the property that they admit a Kähler class lying in $H^2(X,\Z)$. By Hodge index theorem, this additional piece of linear algebra data is the notion of polarization which we recall now.
\begin{definition}\label{pol}
Let $V = (V,F^\bul)$ be a pure $\Q$ (resp. $\Z$)-Hodge structure of weight $n$, let $x : \s \rightarrow \mathrm{GL}(V_\R)$ be the associated Hodge homomorphism and $V_\C = \bigoplus_{p+q = n}V^{p,q}$ the corresponding Hodge decomposition. Let $C = x(i)$. A \textit{polarization of $V$} is a non-degenerate bilinear $(-1)^n$-symmetric form $\psi : V \times V \rightarrow \Q$ (resp. $\psi : V \times V \rightarrow \Z$) such that:
\begin{itemize}
\item for every $p+q = n$ and $r+s = n$ such that $(p,q) \neq (r,s)$, 
\begin{equation}\label{HR1}
\psi_\C(V^{p,q}, \overline{V^{r,s}}) = 0;
\end{equation}
\item denoting by $h = i^{n}\psi_\C(\cdot, \overline{\cdot})$ the associated hermitian form on $V_\C$, for every $p+q = n$ and $v \in V^{p,q}$,
\begin{equation}\label{HR2}
i^{p-q}h(v,v) > 0.
\end{equation}
\end{itemize}
In other words $h(\cdot, C(\cdot))$ is a hermitian form which makes the Hodge decomposition orthogonal and is positive definite on the $V^{p,q}$'s.
\end{definition}
\begin{example}
Let $X$ be a smooth complex projective algebraic variety. For any $0 \leqslant n \leqslant 2\dim X$, the $n$-th primitive cohomology group modulo torsion $H^n(X,\Z)_{\mathrm{prim}}$ carries a canonical pure \textit{polarized} $\Z$-Hodge structure of weight $n$
\end{example}
By the celebrated work of Deligne \cite{hodge2, hodge3}, the existence of a remarkable linear algebraic structure on the cohomology of smooth projective varieties extends to possibly open and singular varieties, at the cost of allowing more general structures, which roughly arise as successive extensions of pure Hodge structures of various weights.
\begin{definition}
\begin{itemize}\item[(a)]A \textit{mixed $\Q$-Hodge structure} is the datum of a finite dimensional $\Q$-vector space $V$, of an increasing finite filtration by $\Q$-vector subspaces $W_\bul$ on $V$, called the \textit{weight filtration}, and of a decreasing finite filtration by $\C$-vector subspaces $F^\bul$ on $V_\C$, called \textit{the Hodge filtration}, such that for each $n \in \Z$, the induced filtration $\gr^W_n(F^\bul)$ defines a pure $\Q$-Hodge structure of weight $n$ on $\gr^W_n(V)$.  \item[(b)] A \textit{graded-polarization of a mixed $\Q$-Hodge structure} $(V, W_\bul, F^\bul)$ is the datum for each $n \in \Z$ of a form $q_n$ on $\gr^W_n(V)$ which is a polarization of the pure $\Q$-Hodge structure of weight $n$ on $\gr^W_n(V)$ induced by $F^\bul$. \item[(c)] A \textit{mixed $\Z$-Hodge structure} is the datum of torsion-free abelian group of finite rank $V$ and of a mixed $\Q$-Hodge structure on $V_\Q = V \otimes_\Z \Q$. \item[(d)] A \textit{graded polarization of a mixed $\Z$-Hodge structure} is a graded-polarization of the corresponding mixed $\Q$-Hodge structure.\end{itemize}
\end{definition}
\begin{example}
Let $X$ be a smooth complex algebraic variety. For any $0 \leqslant n \leqslant 2\dim X$ the $n$-th cohomology space $H^n(X,\Q)$ of $X$ is endowed in \cite{hodge2} with a canonical mixed $\Q$-Hodge structure with weights between $n$ and $2n$.
\end{example}
\begin{example}
Let $X$ be a (possibly singular) complex proper algebraic variety. For any $0 \leqslant n \leqslant 2\dim X$ the $n$-th cohomology space $H^n(X,\Q)$ of $X$ is endowed in \cite{hodge3} with a canonical mixed $\Q$-Hodge structure with weights between $0$ and $n$.
\end{example}
Let $(V,W_\bul,F^\bul)$ be a mixed $\Q$-Hodge structure. For $p,q \in \Z$, the $(p,q)$-\textit{Hodge number} of $V$ is the integer
\[
h^{p,q} = \dim_\C F^p (\Gr^W_{p+q}V)_\C/F^{p+1}(\Gr^W_{p+q}V)_\C.
\]
\begin{lemma}\label{hsymnum}
For every $p,q \in \Z$, one has $h^{p,q} = h^{q,p}$.
\end{lemma}
\begin{proof}
Let $p,q \in \Z$. Then $F^\bul$ induces a pure $\Q$-Hodge structure of weight $p+q$ on $\Gr^W_{p+q}V$ and it is easily seen that $h^{p,q}$ is the dimension of the $(p,q)$-piece of the Hodge decomposition of $(\Gr^W_{p+q}V)_\C$, which satisfies Hodge symmetry by definition. The result follows.
\end{proof}
\subsubsection*{Bigraduations}
By analogy with the pure case, it is natural to try to encode the datum of a mixed Hodge structure in that of a bigraduation as follows:
\begin{definition}
Let $(V, W_\bul, F^\bul)$ be a mixed $\Q$-Hodge structure. A \textit{bigraduation of $(V,W_\bul, F^\bul)$} is a direct sum decomposition $V_\C = \bigoplus_{p,q \in \Z} V^{p,q}$ such that:
\begin{itemize}
\item for each $k \in \Z$, 
\begin{equation}\label{weight}(W_k V)_\C = \bigoplus_{p+q \leqslant k} V^{p,q}; \end{equation}
\item for each $k \in \Z$, 
\begin{equation}\label{hodge} F^r V_\C = \bigoplus_{p \geqslant r} V^{p,q}; \end{equation}
\end{itemize}
Conversely, a decomposition $V_\C = \bigoplus_{p,q \in \Z} V^{p,q}$ is said to \textit{define a mixed $\Q$-Hodge structure} if the filtrations defined by (\ref{hodge}) and (\ref{weight}) define a mixed $\Q$-Hodge structure.
\end{definition}
\begin{remark}
Let $V_\C = \bigoplus_{p,q \in \Z} V^{p,q}$ be a bigraduation of some mixed $\Q$-Hodge structure with Hodge numbers $h^{p,q}$. From (\ref{weight}) and (\ref{hodge}), there exists a $\C$-linear isomorphism $V^{p,q} \rightarrow F^p (\Gr^W_{p+q}V)_\C/F^{p+1}(\Gr^W_{p+q}V)_\C$. In particular $\dim V^{p,q} = h^{p,q}$.
\end{remark}
In general, there is no unique way to define a bigraduation of a mixed $\Q$-Hodge structure. However, one still has:
\begin{proposition}[{\cite[Prop. 1.2.8]{hodge2}, \cite[Prop. 2.2]{klingao}}]\label{delbigrad}
Let $(V, W_\bul, F^\bul)$ be a mixed $\Q$-Hodge structure. There exists a unique bigraduation $V_\C = \bigoplus_{p,q \in \Z} V^{p,q}$ of $(V, W_\bul, F^\bul)$ such that for each $p,q \in \Z$,
\begin{equation}\label{conj}
V^{p,q} = \overline{V^{q,p}} \mod \bigoplus_{\substack{r < p \\ s < q}} V^{r,s}.
\end{equation}
This bigraduation will be referred to as \textit{the canonical bigraduation} of $(V,W_\bul,F^\bul)$.
\end{proposition}
As the datum of a bigraduation of the $\C$-vector space $V_\C$  is equivalent to that of a morphism of $\C$-groups $\s_\C \cong \mult_{m,\C} \times \mult_{m,\C} \rightarrow \gl(V)_\C$, the canonical bigraduation associates such a morphism $x : \s_\C \rightarrow \gl(V)_\C$ to any mixed $\Q$-Hodge structure on $V$. Conversely, Pink proved the following which elucidates under which conditions such a morphism actually defines a polarizable mixed $\Q$-Hodge structure on $V$.
\begin{lemma}[{\cite[(1.4) and (1.5)]{pinkthese}}]\label{pinkcarac}
Let $V$ be a finite dimensional $\Q$-vector space. A morphism $x : \s_\C \rightarrow \gl(V)_\C$ defines a mixed $\Q$-Hodge structure if and only if there exists a connected $\Q$-algebraic group $\pp$ and a $\Q$-representation $\rho : \pp \rightarrow \gl(V)$ such that $x$ factors through $\rho_\C$ and the following hold:
\begin{itemize}
\item[(i)] the composite $\bar{x} : \s_\C \rightarrow \pp_\C \rightarrow (\pp/\pp^u)_\C$ is defined over $\R$;
\item[(ii)]the composite $\bar{x} \circ w : \mult_{m,\R} \rightarrow \s \rightarrow (\pp/\pp^u)_\R$ is defined over $\Q$ and factors through the center of $\pp/\pp^u$;
\item[(iii)] the weight filtration on $\para_\C$ defined by $\ad_{\pp_\C} \circ x$ through formula (\ref{weight}) satisfies $W_0(\para_\C) = \para_\C$ and $W_{-1}(\para_\C) = \lieu_\C$.
\end{itemize}
Under this assumption, the mixed $\Q$-Hodge structure on $V$ defined by $x$ is polarizable if and only if for a group $\pp$ as above the following holds:
\begin{itemize}
\item[(iv)] the involution $\ad (\bar{x}_\R(i))$ is a Cartan involution of $(\pp/\pp^u)_\R^\ad$
\end{itemize}
Under both assumptions, denoting by $W_\bul$ the weight filtration on $V$ defined by $x$, by $(q_k)$ some graded-polarization of the latter mixed $\Q$-Hodge structure and by $\ppol$ the $\Q$-subgroup of $\gl(V)$ whose $R$-points for any $\Q$-algebra $R$ are
{\small\[
\ppol(R) = \{ g \in \gl(V_R) \hspace{0.1cm} : \hspace{0.1cm} g(W_k) \subset W_k \hspace{0.1cm} \mathrm{and} \hspace{0.1cm} \gr_k^W(g) \in \mathbf{GAut}(\gr_k^W(V), q_k)(R) \hspace{0.1cm} \forall k \in \Z\},
\]}
the morphism $x$ satisfies properties $(i)-(iv)$ for $\pp = \ppol$, with $\rho$ being the natural inclusion in $\gl(V)$. Conversely, any pair $(\pp, \rho)$ for which $x$ satisfies properties $(i)-(iv)$ satisfies that $\rho(\pp)$ is contained in $\ppol$ for some graded-polarization $(q_k)$ of $V$.
\end{lemma}
\begin{proof}
The first assertion is \cite[(1.5)]{pinkthese}. The second is \cite[(1.12)]{pinkthese} with the "$G_1$" in \textit{op. cit.} equal to $(\pp/\pp^u)^\ad$. The third assertion is a byproduct of the proof of \cite[(1.5)]{pinkthese} combined with the definition of graded-polarizations. The fourth assertion is \cite[(1.4)(b)]{pinkthese} combined with the definition of graded-polarizations.
\end{proof}
\begin{definition}
Let $V$ be a finite dimensional $\Q$-vector space and $\pp$ a connected $\Q$-algebraic subgroup of $\gl(V)$. A morphism $x : \s_\C \rightarrow \gl(V_\C)$ is said to \textit{satisfy Pink's axioms for $\pp$} if $x(\s_\C) \subset \pp_\C$ and the resulting morphism $\s_\C \rightarrow \pp_\C$ satisfies properties $(i)-(iv)$ of Lemma \ref{pinkcarac}.
\end{definition}
We will need the following slight refinement of Lemma \ref{pinkcarac}, which gives more informations on the groups for which a bigraduation of a graded-polarized mixed $\Q$-Hodge structure satisfies Pink's axioms. Its statement and proof were communicated to me by Z. Gao, but any mistake is of course mine.
\begin{lemma}[Gao]\label{subgroup}
Let $V$ be a mixed $\Q$-Hodge structure with weight filtration $W_\bul$ and graded-polarized by $(q_k)$. Let $x : \s_\C \rightarrow \gl(V)_\C$ be some bigraduation of $V$. Let $\pp$ be a $\Q$-subgroup of $\ppol$ such that $x(\s_\C) \subset \pp_\C$. Then $x$ satisfies Pink's axioms for $\pp$. Furthermore $\pp^u = (\ppol)^u \cap \pp$.
\end{lemma}
\begin{proof}
Let $\pp$ be a $\Q$-subgroup of $\ppol$ such that $x(\s_\C) \subset \pp_\C$. As the mixed $\Q$-Hodge structure on $V$ is polarizable, the third assertion of Lemma \ref{pinkcarac} implies that the morphism $x$ satisfies Pink's axioms for $\ppol$. We first prove that $\pp^u = (\ppol)^u \cap \pp$. The inclusion of the right-hand side in the left-hand side is immediate. For the other inclusion, we need to prove that $\G := \pp/((\ppol)^u \cap \pp)$ is reductive. Denote by $\gpol$ the quotient of $\ppol$ by its unipotent radical, so that the inclusion of $\pp$ in $\ppol$ descends to an inclusion of $\G$ in $\gpol$. Since $x$ satisfies Pink's axioms for $\ppol$, the composite morphism $\s_\C \rightarrow (\ppol)_\C \rightarrow (\gpol)_\C$ is defined over $\R$, and we denote it by $\bar{x} : \s \rightarrow (\gpol)_\R$. Again by Pink's axioms for $\ppol$, the involution $\ad(\bar{x}(i))$ is a Cartan involution of $(\gpol)_\R^\ad$. Because $x(\s_\C) \subset \pp_\C$, the morphism $\bar{x}$ factors through the inclusion of $\G$ in $\gpol$, hence the involution $\ad(\bar{x}(i))$ on $(\gpol)_\R^\ad$ restricts to an involution of $(\G/(Z(\gpol) \cap \G))_\R$, which is a Cartan involution because the real form it defines is a closed subgroup of the real form of $(\gpol)_\R^\ad$ defined by $\ad(\bar{x}(i))$, which is compact since $\ad(\bar{x}(i))$ is a Cartan involution of $(\gpol)_\R^\ad$. Therefore $(\G/(Z(\gpol) \cap \G)$ is reductive hence so is $\G$ as $Z(\gpol) \cap \G \subset Z(\G)$. Thus we have shown that $\pp^u = (\ppol)^u \cap \pp$. In particular, letting $\G_\pp = \pp/\pp^u$ the reductive quotient of $\pp$, this implies that the inclusion of $\pp$ in $\ppol$ descends to an inclusion of $\G_\pp$ in $\gpol$. It is easy and left to the reader to see that the latter fact combined to assertions $(i)-(iii)$ of Pink's axioms for $x$ and $\ppol$ implies that $x$ satisfies assertions $(i)-(iii)$ of Pink's axioms for $\pp$. For $(iv)$, once we know that $\pp^u = (\ppol)^u \cap \pp$, repeating the first part of the proof shows that $\ad(\bar{x}(i))$ is a Cartan involution of $(\G_\pp/(Z(\gpol) \cap \G))_\R$. As $Z(\gpol) \cap \G_\pp)_\R \subset Z(\G_\pp)_\R$, it follows that $\ad(\bar{x}(i))$ is a Cartan involution of $(\G_\pp)^\ad_\R$. The result is proven.

\end{proof}
\subsubsection*{Canonical bigraduations of tensor products and duals}
The end of this subsection is dedicated to describing the canonical bigraduations associated to the natural mixed $\Q$-Hodge structures on tensor products and duals of $\Q$-vector spaces endowed with mixed $\Q$-Hodge structures. This will be crucial for the characterization of Mumford-Tate groups given in Proposition \ref{mtmor}. 

Let $V$ be a $\Q$-mixed Hodge structure and $x : \s_\C \rightarrow \gl(V_\C)$ be the canonical bigraduation of $V$. Consider the morphism of $\Q$-algebraic groups $\phi : \gl(V) \rightarrow \gl(V^\vee)$ defined on $R$-points for $R$ a $\Q$-algebra by sending $g \in \gl(V_R)$ to the $R$-linear automorphism of $V_R^\vee$ mapping a $R$-linear form $f$ to $f \circ g^{-1}$. Set $x^\vee = \phi_\C \circ x$.
\begin{lemma}\label{dual}
The morphism $x^\vee$ defines a mixed $\Q$-Hodge structure on $V^\vee$. Furthermore, the canonical bigraduation of $V^\vee$ endowed with this mixed Hodge structure is $x^\vee$.
\end{lemma} 
\begin{proof}
By Lemma \ref{pinkcarac}, the morphism $x$ factors through the complexification of a $\Q$-representation $\rho : \pp \rightarrow \gl(V)$ of a connected $\Q$-algebraic group $\pp$ such that properties $(i)-(iii)$ of the lemma are satisfied. By construction, $x^\vee$ then factors through the complexification of the $\Q$-representation $\phi \circ \rho$ of $\pp$ and conditions $(i)-(iii)$ are automatically satisfied as they are relative to the morphism $\s_\C \rightarrow \pp_\C$ which is the same as for $x$. By Lemma \ref{pinkcarac} this shows that $x^\vee$ defines a mixed $\Q$-Hodge structure on $V^\vee$.

It remains to see that $x^\vee$ is the canonical bigraduation of $V^\vee$ for the above mixed $\Q$-Hodge structure, i.e. to check condition (\ref{conj}) of Proposition \ref{delbigrad}. We start by expliciting the bigraduation defined on $V^\vee_\C$ by $x^\vee$. To lighten the notation, we set for $p,q \in \Z$:
\[
(V^{p,q})^\vee := \Big(V /{\bigoplus_{(r,s) \neq (p,q)} V^{r,s}}\Big)^\vee
\]
which is canonically a sub-$\C$-vector space of $V_\C^\vee$. Then, one easily checks that the bigraduation on $V_\C^\vee$ defined by $x^\vee$ is:
\[
(V^\vee)^{p,q} = (V^{-p,-q})^\vee,
\]
for $p,q \in \Z$, and that this bigraduation indeed gives the weight and Hodge filtrations on the dual that appear in the litterature (see e.g. \cite[Sect. 1.1]{hodge2}:
\[
W_lV^\vee = \Big( V/W_{-l-1}\Big)^\vee
\]
and 
\[
F^rV_\C^\vee = \Big( V_\C/F^{1-r}V_\C\Big)^\vee.
\]
Let $p,q \in \Z$, set $R^{p,q} = \bigoplus_{r<p,s<q}(V^\vee)^{r,s}$ and $\pi_{p,q} : V^\vee_\C \rightarrow V^\vee_\C/R^{p,q}$ be the quotient map. We wish to prove that $\pi_{p,q}((V^\vee)^{p,q}) = \pi_{p,q}(\overline{(V^\vee)^{q,p}})$. We start by giving an explicit description of $\pi_{p,q}$. First, one has
{\small\[
R^{p,q} = \bigoplus_{r<p,s<q}\Big(\frac{V_\C}{\bigoplus_{(k,l) \neq (-r,-s)}V^{k,l}}\Big)^\vee =  \Big(\frac{V_\C}{\bigcap_{r<p,s<q} \bigoplus_{(k,l) \neq (-r,-s)} V^{k,l}}\Big)^\vee =  \Big(\frac{V_\C}{\bigoplus_{(r,s) \in I_{p,q}} V^{r,s}}\Big)^\vee
\]}
where $I_{p,q} = \{(r,s) \hspace{0.1cm} : \hspace{0.1cm} r \leqslant -p\} \cup \{(r,s) \hspace{0.1cm} : \hspace{0.1cm} s \leqslant -q\}$.
Then, the quotient vector space $V_\C^\vee / R^{p,q}$ is canonically isomorphic to $\Big(\bigoplus_{(r,s) \in I_{p,q}} V^{r,s}\Big)^\vee$, and under this identification, $\pi_{p,q}$ identifies to the map which sends a $\C$-linear form on $V_\C$ to its restriction to the $\C$-vector subspace $\bigoplus_{(r,s) \in I_{p,q}} V^{r,s}$.

Besides, in view of the definition of the $(V^\vee)^{p,q}$'s, a $\C$-linear form is in $(V^\vee)^{p,q}$ (resp. $\overline{(V^\vee)^{q,p}}$) if and only if $\bigoplus_{(r,s) \neq (-p,-q)}V^{r,s} \subset \ker f$ (resp. $\bigoplus_{(r,s) \neq (-q,-p)}\overline{V^{r,s}} \subset \ker f$). It follows that proving that $\pi_{p,q}((V^\vee)^{p,q}) = \pi_{p,q}(\overline{(V^\vee)^{q,p}})$ reduces to the equality of $\C$-vector subspaces
\[
\Big(\bigoplus_{(r,s) \neq (-p,-q)}V^{r,s} \Big) \cap \Big(\bigoplus_{(r,s)\in I_{p,q}} V^{r,s} \Big) = \Big(\bigoplus_{(r,s) \neq (-p,-q)}\overline{V^{s,r}} \Big) \cap \Big(\bigoplus_{(r,s)\in I_{p,q}} V^{r,s} \Big).
\]
Let's first prove that both have the same dimension. First note that for $r,s \in \Z$, we have $\dim_\C V^{r,s} = h^{r,s} = h^{s,r} = \dim_\C V^{s,r} = \dim_\C \overline{V^{s,r}}$. Therefore,
\[
\dim_\C \bigoplus_{(r,s) \neq (-p,-q)}V^{r,s} = \dim_\C \bigoplus_{(r,s) \neq (-p,-q)}\overline{V^{s,r}}.
\]
Besides, since $(-p,-q) \in I^{p,q}$, one has
\[
\Big(\bigoplus_{(r,s) \neq (-p,-q)}V^{r,s} \Big) + \Big(\bigoplus_{(r,s)\in I_{p,q}} V^{r,s} \Big) = V_\C.
\]
Finally (\ref{conj}) of Proposition \ref{delbigrad} implies that $\overline{V^{-q,-p}} \subset V^{-p,-q} \oplus \bigoplus_{r<-p, s<-q} V^{r,s} \subset \bigoplus_{(r,s) \in I_{p,q}} V^{r,s}$ which shows that
\[
\Big(\bigoplus_{(r,s) \neq (-p,-q)}\overline{V^{s,r}} \Big) + \Big(\bigoplus_{(r,s)\in I_{p,q}} V^{r,s} \Big) = V_\C.
\]
By the Grassman formula, this implies that the $\C$-dimensions of the above two intersections are equal. Therefore, we are reduced to proving one inclusion. Again by (\ref{conj}) of Proposition \ref{delbigrad}, we have
\[
\bigoplus_{(r,s) \in I_{p,q} -\{(-p,-q)\}}\overline{V^{s,r}} \subset \sum_{(r,s) \in I_{p,q} -\{(-p,-q)\}} \Big[V^{r,s} \oplus \bigoplus_{k<r,l<s} V^{k,l}\Big] \subset \bigoplus_{(r,s) \in I_{p,q}-\{(-p,-q)\}} V^{r,s}.
\]
Since the first and last member of this chain of inclusions of $\C$-vector subspaces of $V_\C$ have the same dimension, all inclusions are equalities. Therefore, one finds that
\[
\bigoplus_{(r,s) \in I_{p,q}-\{(-p,-q)\}} V^{r,s} = \bigoplus_{(r,s) \in I_{p,q} -\{(-p,-q)\}}\overline{V^{s,r}} \subset \Big(\bigoplus_{(r,s) \neq (-p,-q)}\overline{V^{s,r}} \Big) \cap \Big(\bigoplus_{(r,s)\in I_{p,q}} V^{r,s} \Big)
\]
and the desired inclusion follows from the remark that:
\[
\Big(\bigoplus_{(r,s) \neq (-p,-q)}V^{r,s} \Big) \cap \Big(\bigoplus_{(r,s)\in I_{p,q}} V^{r,s} \Big) = \bigoplus_{(r,s) \in I^{p,q}-\{(-p,-q)\}} V^{r,s}.
\]
This finishes the proof of the proposition.
\end{proof}
Let $V_1$ and $V_2$ be two mixed $\Q$-Hodge structures and $x_i : \s_\C \rightarrow \gl(V_{i,\C})$ $(i=1,2)$ their respective canonical bigraduations. Consider the map of $\Q$-algebraic groups $\psi : \gl(V_1) \times \gl(V_2) \rightarrow \gl(V_1 \otimes V_2)$ which for any $\Q$-algebra sends a pair of $R$-linear automorphisms $(g_1,g_2)$ to the $R$-linear automorphism of $V_{1,R} \otimes_R V_{2,R}$ defined on pure tensors by $v_1 \otimes v_2 \mapsto g_1(v_1) \otimes g_2(v_2)$. Denote by $x_1 \otimes x_2$ the morphism $\s_\C \rightarrow \gl(V_1 \otimes V_2)$ defined as $\psi_\C \circ (x_1,x_2)$ where $(x_1,x_2)$ is the product morphism $\s_\C \rightarrow \gl(V_{1,\C}) \times \gl(V_{2,\C})$ defined in the usual categorical way.
\begin{lemma}\label{tens}
The morphism $x_1 \otimes x_2$ defines a mixed $\Q$-Hodge structure on $V_1 \otimes V_2$. Furthermore, the canonical bigraduation of $V_1 \otimes V_2$ endowed with this mixed $\Q$-Hodge structure is $x_1 \otimes x_2$.
\end{lemma}
\begin{proof}
Let $i \in \{1,2\}$. By Lemma \ref{pinkcarac}, there is a connected $\Q$-algebraic group $\pp_i$ and a representation $\rho_i: \pp_i \rightarrow \gl(V_i)$ such that $x_i$ factors by $\rho_{i,\C}$ and satisfies axioms $(i) - (iii)$ of the lemma. Then, $(x_1,x_2)$ factors through the complexification of the $\Q$-representation $\psi \circ (\rho_1,\rho_2)$ of $\pp_1 \times \pp_2$ in $\gl(V_1 \otimes V_2)$. It is easily checked that $(\pp_1 \times \pp_2)^u = (\pp_1)^u \times (\pp_2)^u$ and that under the canonical identification $\mathrm{Lie}(\pp_1 \times \pp_2) \cong \mathrm{Lie}(\pp_1) \oplus \mathrm{Lie}(\pp_2)$, the weight filtration on the former, induced by post-composition with the adjoint representation of $\pp_1 \times \pp_2$  is given by
\[
W_k \mathrm{Lie}(\pp_1 \times \pp_2) \cong W_k \mathrm{Lie}(\pp_1) \oplus W_k \mathrm{Lie}(\pp_2)
\]
for each $k \in \Z$. It follows that $(x_1,x_2) = \s_\C \rightarrow \pp_1 \times \pp_2$ satisfies axioms $(i) - (iii)$ of Lemma \ref{pinkcarac} which proves the first part of the statement.

It remains to see that $\psi_\C \circ (x_1, x_2)$ is indeed the canonical bigraduation of the mixed $\Q$-Hodge structure on $V_1 \otimes V_2$ that it defines. Explicitely, the bigraduation $\psi_\C \circ (x_1,x_2)$ of $V_{1,\C} \otimes V_{2,\C}$ is
\[
(V_1 \otimes V_2)^{p,q} := \bigoplus_{\substack{p_1 + p_2 = p \\ q_1 + q_2 = q}} V_1^{p_1,q_1} \otimes V_2^{p_2,q_2}
\]
for $p,q \in \Z$, where we denoted by $V_i^{p,q}$ the pieces of the bigraduation$x_i$ of $V_{i,\C}$ $(i=0,1)$. We wish to prove condition (\ref{conj}) of Proposition \ref{delbigrad}, that is for any $p,q \in \Z$,
\[
(V_1 \otimes V_2)^{p,q} = \overline{(V_1 \otimes V_2)^{q,p}} \mod \bigoplus_{\substack{r<p \\ s<q}} (V_1 \otimes V_2)^{r,s}.
\]
To see this, remark that
\[
\bigoplus_{\substack{r<p \\ s<q}} (V_1 \otimes V_2)^{r,s}  \supseteq \bigoplus_{\substack{p_1+p_2 =p\\q_1+q_2=q}} \bigoplus_{\substack{r_1< p_1, r_2<p_2 \\ s_1 < q_1, s_2<q_2}} V_1^{r_1,s_1} \otimes V_2^{r_2,s_2} 
\]
so that it suffices to prove the equality of $(V_1 \otimes V_2)^{p,q}$ and  $\overline{(V_1 \otimes V_2)^{q,p}}$ modulo the right hand side of the last equation. But this equality is now an immediate consequence of condition (\ref{conj}) of Proposition \ref{delbigrad} for $x_1$ and $x_2$. This finishes the proof.
\end{proof}
An easy induction combining Lemmata \ref{dual} and \ref{tens} easily implies the following:
\begin{corollary}\label{arbtens}
Let $V$ be a mixed $\Q$-Hodge structure, and $x : \s_\C \rightarrow \gl(V)_\C$ be the canonical bigraduation of $V$. Let $(a,b) \in \Z_{\geqslant 0}^2 - \{(0,0)\}$ and $x_{a,b} = x\otimes \cdots \otimes x \otimes x^\vee \otimes \cdots \otimes x^\vee$ where $x$ appears $a$ times and $x^\vee$ appears $b$ times. The morphism $x_{a,b}$ defines a mixed $\Q$-Hodge structure on $T^{a,b}(V) := V^{\otimes a} \otimes (V^{\vee})^{\otimes b}$ and the canonical bigraduation of the latter is $x_{a,b}$.
\end{corollary}
\subsection{Mumford-Tate groups}
Let $V$ be a mixed $\Q$-Hodge structure.
\begin{definition}
A \textit{Hodge tensor for $V$} is an element of $F^0(T^{a,b}(V_\C)) \cap W_0(T^{a,b}(V))$ for some $(a,b) \in \Z_{\geqslant 0}^2 - \{(0,0)\}$. We denote by $\hod^{\bul,\bul}(V^\otimes)$ the bigraded (by the bidegree) vector space of Hodge tensors for $V$.
\end{definition}
We then take the following as a definition of Mumford-Tate group, which is equivalent to the usual tannakian definition by classical arguments (see for example \cite{andre}).
\begin{definition}
The Mumford-Tate group $\pp(V)$ of $V$ is the subgroup of $\gl(V)$ fixing all Hodge tensors for $V$.
\end{definition}
We have the following characterisation, which will be crucial for the interpretation of Hodge loci as intersection loci and which makes the link with the definition below \cite[Prop. 2.3]{klingao}. Note that assertion $(ii)$ below is contained in \cite[(I.C.6)]{ggk}.
\begin{proposition}\label{mtmor}
Let $V$ be a polarizable mixed $\Q$-Hodge structure. For any bigraduation $x : \s_\C \rightarrow \gl(V_\C)$ of $V$, denote by $\pp(x)$ the intersection of $\Q$-algebraic subgroups of $\gl(V)$ whose $\C$-points contain $x(\s_\C)$.
\begin{itemize}
\item[(i)] If $x : \s_\C \rightarrow \gl(V_\C)$ is any bigraduation of $V$, one has $\pp(V) \subset \pp(x)$;
\item[(ii)] If $x : \s_\C \rightarrow \gl(V_\C)$ is the canonical bigraduation of $V$ one has $\pp(V) = \pp(x)$.
\end{itemize}
\end{proposition}
\begin{proof}
We will use the following two lemmas to be proven below.
\begin{lemma}\label{tenszero}
Let $M$ be a mixed $\Q$-Hodge structure. Let
\[
M_\C = \bigoplus_{p,q} M^{p,q}
\]
be a bigraduation of $M$. Then $M^{0,0} \cap M \subset W_0(M) \cap F^0(M_\C)$. If furthermore the bigraduation is the canonical bigraduation, we have equality $W_0(M) \cap F^0(M_\C) = M^{0,0} \cap M$.
\end{lemma}
\begin{lemma}\label{extchar}
Let $x : \s_\C \rightarrow \gl(V_\C)$ be any bigraduation of $V$, and let $\pp(x)'$ be the largest $\Q$-subgroup of $\gl(V)$ which fixes pointwise all rational tensors of $V$ which are fixed by $\pp(x)$. Then $\pp(x) = \pp(x)'$.
\end{lemma}
$(i)$ Let $x : \s_\C \rightarrow \gl(V_\C)$ be any bigraduation of $V$. Let $a,b \in \Z^2_{\geqslant 0} - \{(0,0)\}$. Let $\zeta \in T^{a,b}(V)$ be a rational tensor which is fixed by $\pp(x)$. It is henceforth fixed by $x_{a,b}(\s_\C)$, or equivalently is an eigentensor for the trivial character of $\s_\C$. It follows that $\zeta$ belongs to $T^{a,b}(V)^{0,0} \cap T^{a,b}(V)$. Applying the first assertion of Lemma \ref{tenszero} to the mixed $\Q$-Hodge structure on $T^{a,b}(V)$ and the bigraduation $x_{a,b}$, we find that any $\zeta$ is a Hodge tensor, hence is fixed by $\pp(V)$. Therefore $\pp(V) \subset \pp(x)'$. Now by Lemma \ref{extchar}, we have $\pp(x) = \pp(x)'$, and the desired inclusion $\pp(V) \subset \pp(x)$ follows.

$(ii)$ Assume futhermore that $x$ is the canonical bigraduation of $V$. Let $a,b \in \Z^2_{\geqslant 0} - \{(0,0)\}$ and $\zeta \in T^{a,b}(V)$ be a Hodge tensor. By Corollary \ref{arbtens} the morphism $x_{a,b}$ is the canonical bigraduation of $T^{a,b}$ and applying Lemma \ref{tenszero}(ii) to the mixed $\Q$-Hodge structure on $T^{a,b}(V)$ and its canonical bigraduation $x_{a,b}$ we find that $\zeta \in T^{a,b}(V)^{0,0} \cap T^{a,b}(V)$. Therefore $\zeta$ is fixed by $x_{a,b}(\s_\C)$. Since this is true for an arbitrary Hodge tensor, one finds that $x(\s_\C) \subset \pp(V)_\C$ hence $\pp(x) \subset \pp(V)$ by minimality of $\pp(x)$, as desired.
\end{proof}
\begin{proof}[Proof of Lemma \ref{tenszero}]
The inclusion $M^{0,0} \cap M \subset W_0(M) \cap F^0(M_\C)$ is an immediate consequence of the definition of a bigraduation of a mixed $\Q$-Hodge structure, since it forces $M^{0,0}$ to lie simultaneously in $W_0(M)_\C$ and $F^0(M_\C)$.

We now assume that the given bigraduation is the canonical bigraduation and seek to prove the inverse inclusion. We need to prove that for any $(p,q) \neq (0,0)$, we have that $\Big( W_0(M) \cap F^0(M_\C)\Big) \cap M^{p,q}  = \{0\}$. We assume this does not hold, and pick some pair of indexes $(p,q) \neq (0,0)$ such that 
\[
\Big( W_0(M) \cap F^0(M_\C)\Big) \cap M^{p,q}  \neq \{0\}
\]
First remark that by equations (\ref{weight}) and (\ref{hodge}) of Proposition \ref{delbigrad}, the indexes have to satisfy $p+q \leqslant 0$ and $p \geqslant 0$. Because $(p,q) \neq (0,0)$, this forces $q < 0$. Furthermore, since $W_0(M) \cap F^0(M_\C)$ consists of rational vectors, one must have by taking complex conjugation of the above identity
\[
\Big( W_0(M) \cap F^0(M_\C)\Big) \cap \overline{M^{p,q}}  \neq \{0\}.
\]
By equation (\ref{conj}) of Proposition \ref{delbigrad}, we find that
\[
\overline{M^{p,q}} \subset M^{q,p} \oplus \bigoplus_{r<q, s<p}M^{r,s}.
\]
But we saw that $q < 0$ so that the right hand side intersects $F^0 M_\C$ in $\{0\}$. This gives the desired contradiction
\end{proof}
\begin{proof}[Proof of Lemma \ref{extchar}]
We adapt the proof of \cite[Lem. 2.(a)]{andre}, replacing the fact that the underlying space of any representation of $\pp(V)$ is canonically endowed with a mixed $\Q$-Hodge structure, by the fact that the underlying space of any rational representation of $\pp(x)$ is endowed with a mixed $\Q$-Hodge structure obtained by precomposing the complexified representation with the cocharacter $x: \s_\C \rightarrow \pp(x)_\C$. The inclusion $\pp(x) \subset \pp(x)'$ is true by definition. To prove the other inclusion, it suffices to prove that there exists a finite set of pairs of indexes $(m_i, n_i)$, a sub-$\gl(V)$-module $W$ of $\bigoplus_i T^{m_i, n_i}(V)$ and a rational vector $w \in W$ such that $\pp(x)$ is the biggest subgroup $\fix_{\gl(V)}(w)$ of $\gl(V)$ which fixes $w$. Indeed assume the latter statement. Then the components $w_i$ of $w$ are rational tensors of $V$ which are fixed by $\pp(x)$. Therefore $\pp(x)' \subset \bigcap_i \fix_{\gl(V)}(w_i)$. Since the latter was assumed to be precisely $\pp(x)$, this gives the desired inclusion.

By \cite[Prop. 3.1(a)-(b))]{delabs}, there exists a finite set of pairs of indexes $(m_i, n_i)$, a sub-$\gl(V)$-module $W$ of $\bigoplus_i T^{m_i, n_i}(V)$ and a rational line $L \subset W$ such that $\pp(x) = \stab_{\gl(V)}(L)$. By the third assertion of Lemma \ref{pinkcarac} one has $\pp(x) \subset \ppol$ for some graded-polarization $(q_k)$ of $V$. Hence, by Lemma \ref{subgroup} the morphism $x$ satisfies Pink's axioms for $\pp(x)$. Therefore post-composition of $x : \s_\C \rightarrow \pp(x)_\C$ with the complexification of the induced representation $\pp(x) \rightarrow \gl(L)$ makes $L$ a mixed $\Q$-Hodge sub-structure of some $\bigoplus_i T^{m_i,n_i}(V)$. From here, one finishes the proof as in \cite[Lem. 2(a)]{andre}, which we recall for completeness. Since it has rank one $L$ must be isomorphic to some Tate Hodge structure $\Q(N_1)$. 

If $N_1 = 0$, then $x(\s_\C)$ acts trivially on $L_\C$ and because $\pp(x)$ is the smallest $\Q$-subgroup of $\gl(V)$ such that $x(\s_\C) \subset \pp(x)_\C$, it must also act trivially on $L$. Combined with $\pp(x) = \stab_{\gl(V)}(L)$, this forces $\pp(x) = \fix_{\gl(V)}(w)$ for any generator $w$ of $L$.

If $N_1 \neq 0$ then $V$ must have non-zero weights, because otherwise all the $T^{m,n}(V)$ would be pure of weight $0$ and so would be the sub-structure $L$. As a consequence, there exists an integer $n \in \Z$ such that $\det(W_n(V))$ has non-zero weight $N_2$. Up to replacing $V$ by $V^\vee$, one can assume that $N_1$ and $N_2$ have the same sign. Let $r = \dim_\Q W_n(V)$. Then, the rank one mixed $\Q$-Hodge sub-structure 
\[
L^{\otimes |N_2|} \otimes (\bigwedge^r W_n(V))^{\otimes 2|N_1|}
\]
of 
\[
(\bigoplus_i T^{m_i,n_i})^{\otimes |N_2|} \otimes (\bigwedge^r V)^{\otimes 2|N_1|}
\]
has weight $0$ hence isomorphic to $\Q(0)$. Arguing as in the $N_1 = 0$ case, one finds that $\pp(x)$ is the fixator in $\gl(V)$ (acting on $(\bigoplus_i T^{m_i,n_i})^{\otimes |N_2|} \otimes (\wedge^r V)^{\otimes 2|N_1|}$) of any generator of the line $L^{\otimes |N_2|} \otimes (\bigwedge^r W_n(V))^{\otimes 2|N_1|}$.

In both case a representation is constructed that satisfies the asked for properties at the beginning of the proof.
\end{proof}
\subsection{Mixed Hodge classes}
In this subsection, following \cite{klingler_atyp}, we recall the formalism of mixed Hodge data, and introduce the notion of mixed Hodge class which gives a nice functorial setting to work with parameter spaces of mixed $\Q$-Hodge structures with prescribed Mumford-Tate group. The main goal, building on the previous two subsections, is a proof that connected mixed Mumford-Tate domains indeed parametrize mixed Hodge structures with prescribed Mumford-Tate group (Proposition \ref{mtbound}).
\subsubsection*{Mixed Hodge data}
The starting point is the following particular case (for faithful representations) of \cite[(1.7)]{pinkthese} whose statement is \cite[Prop. 3.1]{klingler_atyp}:
\begin{proposition}[{\cite[(1.7)]{pinkthese}}] \label{hodgedat}
Let $\pp$ be a connected $\Q$-algebraic group, let $x_0 : \s_\C \rightarrow \pp_\C$ be a morphism of $\C$-algebraic groups satisfying the axioms $(i)-(iv)$ of Lemma \ref{pinkcarac}, and let $X_\pp$ be the $\pp(\R)\pp^u(\C)$-orbit of $x_0$ in $\Hom_\C(\s_\C, \pp_\C)$, where $\pp(\R)\pp^u(\C) \subset \pp(\C)$ acts by conjugation on the target. Let $\rho : \pp \hookrightarrow \gl(V)$ be a faithful and finite-dimensional representation of $\pp$ as a $\Q$-algebraic group. Then,
\begin{itemize}
\item[(i)] any $x \in X_{\pp}$ satisfies the axioms $(i)-(iv)$ of Lemma \ref{pinkcarac};
\item[(ii)] the image $\mathcal{D}_{\pp, \rho}$ of the map (well defined by $(i)$)
\[
\psi_{\rho,X_\pp} : X_\pp \rightarrow \{\mathrm{mixed} \hspace{0.1cm} \Q\mathrm{-Hodge} \hspace{0.1cm} \mathrm{structures} \hspace{0.1cm} \mathrm{on} \hspace{0.1cm} V\}
\]
admits a unique structure of complex manifold such that the Hodge filtration on $V_\C$  corresponding to $\psi_{\rho,X_\pp}(x)$ for $x \in X_\pp$ varies holomorphically with $\psi_{\rho, X_\pp}(x)$. The action of $\pp(\R)\pp^u(\C)$ on $X_\pp$ descends to $D_{\pp, \rho}$ along $\psi_{\rho, X_\pp}$. The complex structure on $D_{\pp, \rho}$ is invariant under this action of $\pp(\R)\pp^u(\C)$ and $\pp^u(\C)$ acts holomorphically on $\mathcal{D}_{\pp, \rho}$;
\item[(iii)] The complex manifold $\mathcal{D}_{\pp, \rho}$ is independent of the choice of the faithful representation $\rho$, and will be denoted $\mathcal{D}_\pp$;
\item[(iv)] For this complex manifold structure, there is an open embedding into a flag manifold $\mathcal{D}_\pp \hookrightarrow \check{\mathcal{D}}_\pp := \pp(\C)/\exp(F^0 \para_\C)$, where $\para = \lie(\pp)$ endowed with the adjoint mixed $\Q$-Hodge structure induced by $x_0$. This embedding realizes $\mathcal{D}_\pp$ as a $\pp(\R)\pp^u(\C)$-orbit in $\check{\mathcal{D}}_\pp$
\end{itemize}
\end{proposition}
\begin{proof}
The only point which is not in \cite[(1.7)]{pinkthese} is the fact that axiom $(iv)$ of Lemma \ref{pinkcarac} is preserved under $\pp(\R)\pp^u(\C)$-conjugation of morphisms $x : \s_\C \rightarrow \pp_\C$ on the target. This is contained in \cite[(1.12)]{pinkthese}.
\end{proof}
\begin{remark}
In view of point $(iii)$, there exists a unique map $\psi_{X_\pp} : X_\pp \rightarrow \mathcal{D}_\pp$ such that for any faithful representation $\rho : \pp \hookrightarrow \gl(V)$, there exists an injective holomorphic map \[j_\rho : \mathcal{D}_\pp \rightarrow \{\mathrm{mixed} \hspace{0.1cm} \Q\mathrm{-Hodge} \hspace{0.1cm} \mathrm{structures} \hspace{0.1cm} \mathrm{on} \hspace{0.1cm} V\}\] such that $\psi_{\rho,X_\pp} = j_\rho \circ \psi_{X_\pp}$. Furthermore $\psi_{X_\pp}$ is $\pp(\R)\pp^u(\C)$-equivariant by definition of the $\pp(\R)\pp^u(\C)$-action on $\mathcal{D}_\pp$.
\end{remark}
With this result in hand, one can make sense of the following definition:
\begin{definition}[{\cite[Def. 3.4]{klingler_atyp}}]
A \textit{(connected) mixed Hodge datum} is a triple $(\pp, X_\pp, D_\pp)$, where:
\begin{itemize}
\item[(a)] $\pp$ is a connected $\Q$-algebraic group;
\item[(b)] $X_\pp$ is the $\pp(\R)\pp^u(\C)$-orbit in $\Hom_\C(\s_\C, \pp_\C)$ of a morphism $x_0 : \s_\C \rightarrow \pp_\C$ satisfying axioms $(i)-(iv)$ of Lemma \ref{pinkcarac};
\item[(c)] $D_\pp$ is a connected component of the complex manifold $\mathcal{D}_\pp$ associated in Proposition \ref{hodgedat} to $(\pp,X_\pp)$.
\end{itemize}
A complex manifold of the form $D_\pp$ for a (connected) mixed Hodge datum $(\pp, X_\pp, D_\pp)$ is called a \textit{(connected) Mumford-Tate domain}. It is a $\pp(\R)^+\pp^u(\C)$-orbit in $\mathcal{D}_\pp$.
\end{definition}
\begin{remark}
In the sequel, we will only work with connected Hodge data and connected Mumford-Tate domains in the terminology of \cite{klingler_atyp}, and we will omit the adjective "connected" which we still added in the definition in order to be consistent with the litterature.
\end{remark}
The following is an immediate consequence of \cite[(1.7)]{pinkthese}.
\begin{lemma}\label{subdatimm}
Let $(\pp, X_\pp, D_\pp)$ and $(\pp',X_{\pp'}, D_{\pp'})$ be two mixed Hodge data and let $i : \pp \rightarrow \pp'$ be a group morphism such that post-composition by $i_\C$ maps $X_\pp$ in $X_{\pp'}$. Then, the map $X_\pp \rightarrow X_{\pp'}$ descends to a holomorphic map $\phi : \mathcal{D}_\pp \rightarrow \mathcal{D}_{\pp'}$. If furthermore $i$ is injective, the map $\phi$ is a closed complex-analytic immersion.
\end{lemma}
\begin{definition}[{\cite[Def. 3.5]{klingler_atyp}}]
A \textit{morphism of mixed Hodge data} $(\pp, X_\pp, D_\pp) \rightarrow (\pp',X_{\pp'}, D_{\pp'})$ is a morphism of $\Q$-algebraic groups $i : \pp \rightarrow \pp'$ such that post-composition of morphisms by $i_\C$ induces a map $X_\pp \rightarrow X_{\pp'}$, and the induced holomorphic map $\mathcal{D}_\pp \rightarrow \mathcal{D}_{\pp'}$ maps $D_\pp$ into $D_{\pp'}$. A \textit{mixed Hodge subdatum} of some mixed Hodge datum $(\pp', X_{\pp'}, D_{\pp'})$ is a mixed Hodge datum $(\pp, X_\pp, D_\pp)$ such that $\pp$ is a $\Q$-algebraic subgroup of $\pp'$ and the inclusion $\pp \hookrightarrow \pp'$ is a morphism of mixed Hodge data.
\end{definition}
\begin{example} Fix the datum of
\begin{itemize}
\item a finite dimensional $\Q$-vector space $V$ endowed with an increasing finite filtration $W_\bul$ by $\Q$-vector subspaces;
\item for each $k \in \Z$ a $(-1)^k$-symmetric non-degenerate bilinear form $q_k : \Gr_k^W(V) \times \Gr_k^W(V) \rightarrow \Q$;
\item a sequence of integers $(h^{p,q})_{p,q \in \Z}$ such that $h^{p,q} = h^{q,p}$ for every $p,q \in \Z$, and $\sum_{p,q \in \Z} h^{p,q} = \dim_\Q V$.
\end{itemize}
Let $M$ (resp. $\check{M}$) be the set of decreasing finite filtrations $F^\bul$ of $V_\C$ such that the triple $(V, W_\bul, F^\bul)$ defines a mixed $\Q$-Hodge structure which is graded-polarized by $(q_k)_{k\in \Z}$ (resp. such that for each $k \in \Z$, the form $q_k$ satisfies condition (\ref{HR1}) of Definition \ref{pol}) and with Hodge numbers $(h^{p,q})_{p,q \in \Z}$. Clearly $\check{M}$ is closed in a product of Grassmanian hence inherits the structure of complex algebraic variety. For this structure, $M$ is clearly a semi-algebraic open subset of $\check{M}$, hence there is an induced structure of complex manifold on $M$. Let $\pp^M = \ppol$. Let $x_0$ be a bigraduation of some mixed $\Q$-Hodge structure $h_0 \in M$. Let $X^M$ be the $\pp^M(\R)(\pp^M)^u(\C)$-orbit of $x_0$ in $\Hom_\C(\s_\C, \pp^M_\C)$. Then for any connected component $M^+$ of $M$ (for the topology induced by the above defined complex-analytic structure on $M$)  the triplet $(\pp^M, X^M, M^+)$ is a mixed Hodge datum. Such a mixed Hodge datum is called a \textit{period mixed Hodge datum}
\end{example}
\begin{remark}
Let $(\pp, X_\pp, D_\pp)$ be a mixed Hodge datum and $\rho : \pp \hookrightarrow \gl(V)$ a faithful representation. Then there exists a datum $((q_k), (h^{p,q}))$ and a $\pp^M(\R)(\pp^M)^u(\C)$-conjugacy class $X^M$ as in the latter example such that $\rho$ factors through $\pp^M$ and the induced inclusion $\pp \hookrightarrow \pp^M$ defines a morphism of mixed Hodge data $(\pp, X_\pp, D_\pp) \rightarrow (\pp^M, X^M, M^+)$ for some connected component $M^+$ of $M$.
\end{remark}
\subsubsection*{Mixed Hodge classes}
Mixed Hodge data have the bad feature of fixing the conjugacy class $X_\pp$ of bigraduations, which might not contain all bigraduations of mixed Hodge structures in $D_\pp$. In particular, for $h \in D_\pp$ the conjugacy class $X_\pp$ might not contain the canonical bigraduation of $h$. Furthermore, this constrains (unnecessarily, as one is primarily interested in mixed Hodge structures, not bigraduations) the notion of morphisms of mixed Hodge data. To palliate to this problem we make the following:
\begin{definition}
\begin{itemize}
\item[(a)] Two mixed Hodge data $(\pp, X_\pp, D_\pp)$ and $(\pp',X_{\pp'}, D_{\pp'})$ are called \textit{essentialy equivalent} if there exists an isomorphism $i : \pp \cong \pp'$ such that for any faithful representation $\rho : \pp' \hookrightarrow \gl(V)$, one has $j_{\rho \circ i}(D_\pp) = j_\rho(D_{\pp'})$. This defines an equivalence relation on the set of mixed Hodge data, called the \textit{essential equivalence relation}.
\item[(b)] A \textit{mixed Hodge class} is an equivalence class, denoted $(\pp,D_\pp)$, of mixed Hodge data for the essential equivalence relation. A mixed Hodge datum in the equivalence class $(\pp, D_\pp)$ is called a \textit{representative of $(\pp, D_\pp)$}.
\item[(c)] A \textit{morphism of mixed Hodge classes} $(\pp, D_\pp) \rightarrow (\pp', D_{\pp'})$ is a morphism of $\Q$-algebraic groups $\pp \rightarrow \pp'$ which defines a morphism of mixed Hodge data between some representatives $(\pp, X_\pp, D_\pp)$ and $(\pp', X_{\pp'}, D_{\pp'})$ of the classes $(\pp, D_\pp)$ and $(\pp', D_{\pp'})$.
\item[(d)] A \textit{mixed Hodge sub-class} of a mixed Hodge class $(\pp', D_{\pp'})$ is a mixed Hodge class $(\pp,D_\pp)$ endowed with an inclusion of $\Q$-algebraic groups $\pp \hookrightarrow \pp'$ which is a morphism of mixed Hodge classes.
\end{itemize}
\end{definition}
\begin{remark}
Let $i : (\pp, D_\pp) \rightarrow (\pp', D_{\pp'})$ be a morphism of mixed Hodge classes. For any representative $(\pp, X_\pp, D_\pp)$ of $(\pp, D_\pp)$, there exists a unique representative $(\pp', X_{\pp'}, D_{\pp'})$ of $(\pp', D_{\pp'})$ such that $i : \pp \rightarrow \pp'$ defines a morphism of mixed Hodge data $(\pp, X_\pp, D_\pp) \rightarrow (\pp', X_{\pp'}, D_{\pp'})$.
\end{remark}
\begin{lemma}
Let $(\M,D_\M)$ be a mixed Hodge sub-class of a mixed Hodge class $(\pp, D_\pp)$, and $i : \M \hookrightarrow \pp$ be the inclusion. Let $(\M, X_\M, D_\M)$, resp. $(\pp, X_\pp, D_\pp)$, be representatives of $(\M,D_\M)$, resp. $(\pp, D_\pp)$, such that $i$ makes the former a mixed Hodge sub-datum of the latter. By Lemma \ref{subdatimm}, these choices give rise to a complex-analytic closed immersion
\[
\phi : D_\M \hookrightarrow D_\pp
\]
induced by the composition-by-$i$ map $X_\M \rightarrow X_\pp$. Then $\phi$ doesn't depend on the choice of representatives of the classes $(\M,D_\M)$ and $(\pp, D_\pp)$. It is denoted $\phi_i$.
\end{lemma}
\begin{proof}
Let $(\M, X'_\M, D_\M)$ and $(\pp, X'_\pp, D_\pp)$ be another pair of representatives such that the inclusion $i$ makes the first a mixed Hodge sub-datum of the second. By definition, one has a commutative diagram
\[
\xymatrix{
X_\M \ar[d]_{i_\C \circ (-)} \ar[r]^{\psi_{X_\M}} & D_\M \ar@/^/[d]^\phi \ar@/_/[d]_{\phi'} & X'_\M \ar[l]_{\psi_{X'_\M}} \ar[d]^{i_\C \circ (-)} \\
X_\pp \ar[r]_{\psi_{X_\pp}} & D_\pp & X'_{\pp} \ar[l]^{\psi_{X'_\pp}}
}
\]
We are tasked with proving that for any $h \in D_\M$ one has $\phi(h) = \phi'(h)$. Let $x \in X_\M$ and $x' \in X'_\M$ such that $\psi_{X_\M}(x) = \psi_{ X'_\M}(x') = h$. By definition, one has $\phi(h) = \psi_{X_\pp} (i_\C \circ x)$ and $\phi'(h) = \psi_{X'_\pp}(i_\C \circ x')$. Let $\rho : \pp \hookrightarrow \gl(V)$ be a faithful representation. Since
\[
j_\rho : D_\pp \hookrightarrow \{\mathrm{mixed} \hspace{0.1cm} \Q\mathrm{-Hodge} \hspace{0.1cm} \mathrm{structures} \hspace{0.1cm} \mathrm{on} \hspace{0.1cm} V\}
\]
is injective, it suffices to check that $j_\rho (\phi(h)) = j_\rho(\phi'(h))$. To see this, recall that $j_\rho \circ \psi_{X_\pp}$ maps a morphism $x_\pp : \s_\C \rightarrow \pp_\C$ satisfying Pink's axioms for $\pp$ to the mixed $\Q$-Hodge structure defined by $\rho_\C \circ x_\pp$. Therefore,
\[
j_\rho(\phi(h)) = j_\rho(\psi_{X_\pp}(i_\C \circ x)) = j_{\rho \circ i}(\psi_{X_\M}(x)),
\]
and similarly,
\[
j_\rho(\phi'(h)) = j_\rho(\psi_{X'_\pp}(i_\C \circ x')) = j_{\rho \circ i}(\psi_{X'_\M}(x')).
\]
Since by assumption $\psi_{X_\M}(x) = \psi_{X'_\M}(x')$, this implies that $\phi(h) = \phi'(h)$  as desired.
\end{proof}
\begin{convention}
Let $(\pp, D_\pp)$ be a mixed Hodge class and $(\M,D_\M)$ a mixed Hodge subclass, with underlying group inclusion $i : \M \hookrightarrow \pp$. In the sequel, we will identify $D_\M$ with the complex submanifold $\phi_i(D_\M)$ of $D_\pp$. It is an $\M(\R)^+\M^u(\C)$-orbit in $D_\pp$. 
\end{convention}
\begin{remark}
It makes sense to talk about $\M(\R)^+\M^u(\C)$-orbit in $D_\pp$ as above, because in the setting of the above convention, Lemma \ref{subgroup} ensures that $\M^u(\C) \subset \pp^u(\C)$.
\end{remark}
\subsubsection*{Mixed Hodge structures with bounded Mumford-Tate group}
We can now turn to the proof of the following which states that Mumford-Tate subdomains indeed parametrize mixed Hodge structures with bounded Mumford-Tate group:
\begin{proposition}\label{mtbound}
Let $(\pp, D_\pp)$ be a mixed Hodge class and $i : \M \hookrightarrow \pp$ the inclusion of a connected $\Q$-algebraic subgroup. Let $h \in D_\pp$. The following assertions are equivalent:
\begin{itemize}
\item[(i)] for any faithful $\Q$-representation $\rho : \pp \rightarrow \gl(V)$ of $\pp$, the mixed $\Q$-Hodge structure $j_\rho(h)$ has Mumford-Tate group contained in $\M$;
\item[(ii)] for some faithful $\Q$-representation $\rho : \pp \rightarrow \gl(V)$ of $\pp$, the mixed $\Q$-Hodge structure $j_\rho(h)$ has Mumford-Tate group contained in $\M$;
\item[(iii)] there exists a mixed Hodge sub-class $(\M, D_\M)$ of $(\pp, D_\pp)$ with underlying inclusion $i$, such that $h \in D_\M \subset D_\pp$.
\end{itemize}
\end{proposition}
\begin{proof}
$(i) \Rightarrow (ii)$: Immediate. 

$(ii) \Rightarrow (iii)$: Assume $(ii)$, let $\rho : \pp \rightarrow \gl(V)$ be as in the statement and denote by $i : \M \hookrightarrow \pp$ the inclusion. By abuse of notation, we shall identify $\pp$ and $\M$ to connected $\Q$-algebraic subgroups of $\gl(V)$ using $\rho$ and $\rho \circ i$. Let $W_\bul$ be the induced weight filtration on $V$. 

Let $(\pp, X_\pp, D_\pp)$ be some representative of $(\pp, D_\pp)$ and $x_0 \in X_\pp$ some lift of $h$. By definition of mixed Hodge data $\rho_\C \circ x_0$ satisfies Pink's axioms for $\pp$, hence by the fourth assertion of Lemma \ref{pinkcarac}, there is some graded-polarization $(q_k)$ of $j_\rho(h)$ such that $\pp \subset \ppol$.

Let $x_\Del$ be the canonical bigraduation of the mixed $\Q$-Hodge structure $j_\rho(h)$. By assumption $\pp(j_\rho(h)) \subset \M$, and by Proposition \ref{mtmor}(ii) one has the inclusion $x_\Del(\s_\C) \subset \pp(j_\rho(h))_\C$, hence $x_\Del(\s_\C) \subset \M_\C$. Since $\M$ is contained in $\pp$ as subgroups of $\gl(V)$, and $\pp$ is contained in $\ppol$ as just proven, it follows that $\M$ is too. Lemma \ref{subgroup} then implies that $x_\Del$ satisfies Pink's axioms for $\M$. Let $X_\M$ be the $\M(\R)\M^u(\C)$-orbit of $x_\Del$ in $\Hom_\C(\s_\C, \M_\C)$, let $\mathcal{D}_\M$ be the image of 
\[
\psi_{\rho \circ i, X_\M} : X_\M \rightarrow  \{\mathrm{mixed} \hspace{0.1cm} \Q\mathrm{-Hodge} \hspace{0.1cm} \mathrm{structures} \hspace{0.1cm} \mathrm{on} \hspace{0.1cm} V\}.
\] 
Finally, let $D_\M$ be the connected component of $\mathcal{D}_\M$ containing $j_\rho(x)$. By construction $(\M,X_\M, D_\M)$ is a mixed Hodge datum. Denote by $(\M,D_\M)$ the associated mixed Hodge class. 

We now proceed to prove that $(\M,D_\M)$ is a mixed Hodge sub-class of $(\pp,D_\pp)$. We saw that $\pp \subset \ppol$ and that $x_\Del(\s_\C) \subset \M_\C \subset \pp_\C$, hence Lemma \ref{subgroup} ensures that $x_\Del$ satisfies Pink's axioms for $\pp$. Using this to identify $x_\Del$ to a $\C$-morphism $\s_\C \rightarrow \pp_\C$, let $X'_\pp$ be the $\pp(\R)\pp^u(\C)$-conjugacy class of $x_\Del$ in $\Hom_\C(\s_\C, \pp_\C)$ and $\mathcal{D'_\pp}$ be the image of 
\[
\psi_{\rho, X'_\pp} : X'_\pp \rightarrow  \{\mathrm{mixed} \hspace{0.1cm} \Q\mathrm{-Hodge} \hspace{0.1cm} \mathrm{structures} \hspace{0.1cm} \mathrm{on} \hspace{0.1cm} V\},
\] 
and $D'_\pp$ be the connected component of $\mathcal{D}'_\pp$ containing $j_\rho(x)$. The $\pp(\R)^+\pp^u(\C)$-orbits $D'_\pp$ and $j_\rho(D_\pp)$ in the set of mixed $\Q$-Hodge structures on $V$ both contain $j_\rho(h)$ hence are equal, and this is easily seen to be true for any faithful representation of $\pp$. Hence $(\pp, X_\pp, D_\pp)$ and $(\pp, X'_\pp, D_\pp)$ are essentially equivalent, and by construction, the inclusion $i : \M \hookrightarrow \pp$ makes $(\M, X_\M, D_\M)$ a mixed Hodge subdatum of $(\pp, X'_\pp, D_\pp)$. Hence $(\M,D_\M)$ is a mixed Hodge subclass of $(\pp, D_\pp)$, and by construction, $j_\rho(h) \in D_\M \subset D_\pp$. This proves $(ii) \Rightarrow (iii)$.

$(iii) \Rightarrow (i)$: Assume $(iii)$, and let $\rho : \pp \rightarrow \gl(V)$ be any faithful $\Q$-representation of $\pp$. Let $(\M, X_\M, D_\M)$ be a representative of the mixed Hodge class $(\M, D_\M)$ whose existence is guaranteed by $(iii)$. Let $x \in X_\M$ such that $\psi_{\rho\circ i, X_\M}(x) = j_\rho(h)$. Then the $\C$-morphism $\rho_\C \circ i_\C \circ x : \s_\C \rightarrow \gl(V)_\C$ is a bigraduation of $j_\rho(h)$  such that $x(\s_\C) \subset \M_\C$. With the notations of Proposition \ref{mtmor}, this implies $\pp(x) \subset \M$. But by Proposition \ref{mtmor}, one has $\pp(j_\rho(h)) \subset \pp(x)$. Therefore, the Mumford-Tate group $\psi_\rho(x)$ is indeed contained in $\M$. This proves $(iii) \Rightarrow (i)$ and completes the proof.
\end{proof}
\begin{corollary}\label{mtindp}
Let $(\pp, D_\pp)$ be a mixed Hodge class and $h \in D_\pp$. For any pair of faithful representations $\rho : \pp \hookrightarrow \gl(V)$ and $\rho' : \pp \hookrightarrow \gl(V')$, the Mumford-Tate groups of the mixed $\Q$-Hodge structures $j_\rho(h)$ and $j_{\rho'}(h)$ are connected $\Q$-algebraic subgroups of $\pp$ and are equal as subgroups of $\pp$.
\end{corollary}
\begin{proof}
Let $\rho$, $\rho'$ and $h$ as in the statement. Let $(\pp, X_\pp, D_\pp)$ be a representative of the mixed Hodge class $(\pp, D_\pp)$. Let $x \in X_\pp$ be such that $\psi_{X_\pp}(x) = h$. Then by Proposition \ref{mtmor}(i), one finds that $\pp(j_\rho(h)) \subset \pp$ and $\pp(j_{\rho'}(h)) \subset \pp$. Applying equivalence between $(i)$ and $(ii)$ in Proposition \ref{mtbound} with $\M = \pp(j_\rho(h))$ gives that $\pp(j_{\rho'}(h)) \subset \pp(j_{\rho}(h))$. Applying the same equivalence with $\M = \pp(j_{\rho'}(h))$ gives that $\pp(j_\rho(h)) \subset \pp(j_{\rho'}(h))$, as desired.
\end{proof}
\begin{definition}
Let $(\pp, D_\pp)$ be a mixed Hodge class, and $h \in D_\pp$. The \textit{Mumford-Tate group of $h$}, denoted $\pp(h)$, is the Mumford-Tate group of the mixed $\Q$-Hodge structure $j_\rho(h)$ for some faithful representation $\rho : \pp \hookrightarrow \gl(V)$. It is a connected $\Q$-algebraic subgroup of $\pp$, as such independent of $\rho$.
\end{definition}
We conclude this subsection with the following two consequences of Proposition \ref{mtbound}
\begin{corollary}\label{mtconn}
Let $(\pp, D_\pp)$ be a mixed Hodge class, let $i : \M \hookrightarrow \pp$ the inclusion of a connected $\Q$-algebraic subgroup of $\pp$ ,and $X \subset D_\pp$ a connected subset such that for any $h \in X$, one has $\pp(h) \subset \M$. Then, there exists a unique mixed Hodge subclass $(\M,D_\M)$ of $(\pp, D_\pp)$ with underlying inclusion $i$ such that $X \subset D_\M \subset D_\pp$.
\end{corollary}
\begin{proof}
By the previous two results, we only have to prove that for any two mixed Hodge subclasses $(\M,D_\M), (\M,D'_{\M})$ of $(\pp, D_\pp)$ with the same underlying inclusion of $\Q$-groups $i : \M \hookrightarrow \pp$, the two (connected) subsets $D_\M$ and $D'_\M$ of $D_\pp$ are either equal or disjoint. Assume that they are not disjoint, and $h$ be a point in the intersection. Then $D_\M$ and $D'_\M$ are both $\M(\R)^+\M^u(\C)$-orbits in $D_\pp$ containing $h$. Hence they are equal.
\end{proof}
\subsubsection*{Torsion mixed Hodge subclasses}
To finish this section, we show that every mixed Hodge class contains a non-trivial mixed Hodge subclass, a fact analogous to the existence of CM points in every Shimura variety for example (and whose proof is similar).
\begin{definition}
A mixed Hodge class $(\M, D_\M)$ is called \textit{torsion} if the group $\M$ is a $\Q$-torus.
\end{definition}
Note that if $(\M, D_\M)$ is a torsion mixed Hodge class, then $\dim D_\M = 0$. Every mixed Hodge class contains a torsion mixed Hodge subclass as the following well known result asserts (see \cite[Lem. 3.2]{kp} for a similar result for limiting mixed Hodge structures):
\begin{proposition}\label{torsionex}
Let $(\pp, D_\pp)$ be a mixed Hodge class with $\dim D_\pp > 0$. There exists a strict torsion mixed Hodge subclass $(\mathbf{T}, D_{\mathbf{T}}) \subsetneq (\pp, D_\pp)$.
\end{proposition}
\begin{proof}
Let $(\pp, X_\pp, D_\pp)$ be a representative of $(\pp, D_\pp)$, let $h \in D_\pp$ and $x \in X_\pp$ such that $\psi_{X_\pp}(x) = h$. Let $G_\C$ be a Levi factor of $\pp_\C$ containing $x(\s_\C)$. Let $\LL$ be a Levi factor (defined over $\Q$) of the $\Q$-algebraic group $\pp$. Then $\LL_\C$ is a Levi factor of $\pp_\C$ hence \cite[Prop. 5.1]{borser} ensures that there exists a $u \in \pp^u(\C)$ such that $uG_\C u^{-1} = \LL_\C$. Letting $\pi : \pp \rightarrow \pp/\pp^u$ be the quotient and $i : \LL \hookrightarrow \pp$ be the inclusion, we get a commutative diagram
\[
\xymatrix{
\s_\C \ar[rrd] \ar[r]^{x_u} \ar@/^2.0pc/[rr]^{\intt_u \circ x} & \LL_\C \ar[rd]^{\sim} \ar[r]^{i_\C} & \pp_\C \ar[d]^{\pi_\C} \\ 
& & (\pp/\pp^u)_\C
}
\]
where $(\pi \circ i)_\C$ is an isomorphism by definition of Levi factors, and the morphism $x_u$ exists because $\intt_u \circ x(\s_\C) \subset uG_\C u^{-1} = \LL_\C$. Since $\intt_u \circ x$ satisfies Pink's axioms for $\pp$, we know that $\pi_\C \circ \intt_u \circ x$ is defined over $\R$. As the isomorphism $(\pi \circ i)_\C$ is defined over $\Q$, hence over $\R$, it follows that $x_u$ is defined over $\R$, and we see it from now as a morphism of real algebraic groups $x_u : \s \rightarrow \LL_\R$. Let $Z$ be a maximal $\R$-torus in $\LL_\R$ containing $x_u(\s)$. Let $\mathbf{T}$ be a maximal $\Q$-torus in $\LL$. Then $\mathbf{T}_\R$ is a maximal $\R$-torus in $\LL_\R$ hence there exists a $g \in \LL(\R)^+ \subset \pp(\R)^+$ such that $gZg^{-1} = \mathbf{T}_\R$. Let $y = \intt_{g} \circ \intt_u \circ x$. By construction we find that $y \in X_\pp$, that $h' := \psi_{X_\pp}(y) \in D_\pp$ and $y(\s_\C) \subset \mathbf{T}_\C$. It then follows from Proposition \ref{mtmor}$(i)$ that $\pp(h') \subset \mathbf{T}$. Finally Corollary \ref{mtconn} applied to $X = \{h'\}$ ensure the existence of a mixed Hodge subclass $(\mathbf{T}, D_{\mathbf{T}})$ containing $h'$. It is torsion because $\mathbf{T}$ is a $\Q$-torus by definition, and it is strict because $\dim D_{\pp} > 0$. 
\end{proof}
\subsection{Infinitesimal structure of Mumford-Tate domains}
Let $(\pp, D_\pp)$ be a mixed Hodge class and $h \in D_\pp$. By construction $D_\pp$ is a dense real-analytic open subset of its \textit{compact dual} $\check{D}_\pp$ which is $\pp(\C)$-orbit. Let $\mu_h : \pp(\C) \rightarrow \check{D}_\pp, g \mapsto g \cdot h$ be the orbit map. The stabilizer of $h$ in $\pp(\C)$ is $\exp(F_h^0 \para)$ where $\para = \lie(\pp)$ and $F_h^\bul \para$ is the Hodge filtration on $\para_\C$ defined by $h$. Therefore, one has an isomorphism
\[
\Theta^{(\pp,D_\pp)}_h : T_h D_\pp \rightarrow \para_\C / F_h^0 \para
\]
satisfying the following functoriality properties:
\begin{itemize}
\item[$(i)$] the derivative of the orbit map $d_1 \mu_h : \para_\C \rightarrow T_h D_\pp$ composed with $\Theta^{(\pp, D_\pp)}_h$ coincides with the quotient map by $F^0_h \para$;
\item[$(ii)$] If $(R,f) : (\M, D_\M) \rightarrow (\pp, D_\pp)$ is a morphism of mixed Hodge classes, and $h \in D_\M$ then 
\[
d_h f = \Theta^{(\pp, D_\pp)}_{f(h)} \circ r_0 \circ  (\Theta^{(\M, D_\M)}_{h})^{-1}
\]
where $r_0 : \mt_\C/F^0_h \mt \rightarrow \para_\C / F^0_{f(h)} \para$ is the map induced by the morphism of mixed $\Q$-Hodge structures $dR : \mt = \lie(\M) \rightarrow \para$.
\end{itemize}
The Griffiths' transversality axiom for variations of mixed Hodge structures gives a special role to the subspace
\[
T_h^{\mathrm{horiz}} D_\pp := (\Theta^{(\pp, D_\pp)}_h)^{-1}\big(F_h^{-1} \para/ F_h^0\para\big) \subset T_h D_\pp
\]
which we will refer to as the \textit{horizontal tangent space} in the sequel. 

It will sometimes be convenient to work with refined identifications induced by bigraduations. Let $(\pp, X_\pp, D_\pp)$ be a representative of $(\pp, D_\pp)$ and $x \in \psi_{X_\pp}^{-1}(h)$. Let
\[ \para_\C = \bigoplus_{p+q \leqslant 0} \para_x^{p,q} \]
be the corresponding bigraduation of $\para$. Equation (\ref{hodge}) gives an isomorphism
\[
\Theta^{(\pp,X_\pp, D_\pp)}_x : T_h D_\pp \cong \bigoplus_{\substack{p+q \leqslant 0 \\ p < 0}} \para_x^{p,q}
\]
such that
\[
\Theta^{(\pp,X_\pp, D_\pp)}_x\big(T_h^{\mathrm{horiz}}D_\pp \big) = \bigoplus_{q \leqslant 1} \para_x^{-1,q}.
\]

\subsection{Mixed Hodge varieties}
In the study of variations of mixed Hodge structures, the dependance of parallel transport (to identify two fibers of the variation) on a choice of path joining the base points forces to work with mixed Hodge structures up to integral equivalence. We recall in this section the construction, again due to Klingler, of appropriate parameter spaces that allow to deal with such considerations.

Consider a mixed Hodge class $(\pp, D_\pp)$, and let $\Gamma \subset \pp(\Q)^+ := \pp(\Q) \cap \pp(\R)^+$ be an neat arithmetic subgroup of $\pp(\R)^+$ (see \cite[(17.1)]{borelarith}). As a subgroup of $\pp(\R)^+$, the group $\Gamma$ acts on the left on the $\pp(\R)^+\pp^u(\C)$-homogeneous space $D_\pp$. Since $\Gamma$ is torsion-free, it acts freely on $D_\pp$ and since the complex structure on $D_\pp$ is $\pp(\R)^+\pp^u(\C)$-invariant, the quotient $\Gamma \quo D_\pp$ has a natural structure of complex manifold, and the projection $\pi_\Gamma : D_\pp \rightarrow \Gamma \quo D_\pp$ is holomorphic (and is even a local biholomorphism on the source).
\begin{definition}[{\cite[Def. 3.8]{klingler_atyp}}]
A \textit{mixed Hodge variety} is a smooth complex-analytic variety of the form $\period_{(\pp , D_\pp), \Gamma} = \Gamma \quo D_\pp$ for a mixed Hodge class $(\pp, D_\pp)$ and a neat arithmetic subgroup $\Gamma \subset \pp(\Q)^+$.
\end{definition}
Let $i : (\M,D_\M) \subset (\pp, D_\pp)$ be a mixed Hodge subclass. Denote by $\Gamma_\M$ the neat arithmetic subgroup $\Gamma \cap \M(\Q)^+$ of $\M(\R)^+$. Since the map $\phi_i$ is equivariant for the inclusion $\M(\R)^+ \hookrightarrow \pp(\R)^+$, it descends to a holomorphic map
\[
\phi_{i, \Gamma} : \period_{(\M,D_\M), \Gamma_\M} \rightarrow  \period_{(\pp,D_\pp), \Gamma}.
\]
\begin{lemma}
The image $\phi_{i, \Gamma}(\period_{(\M,D_\M), \Gamma_\M})$ is a smooth locally closed complex-analytic subset of the mixed Hodge variety $\period_{(\pp, D_\pp), \Gamma}$. 
\end{lemma}
\begin{proof}
By definition, the following diagram commutes
\[
\xymatrix{D_\M \ar[r]^{\phi_i} \ar[d]_{\pi_{\Gamma_\M}} & D_\pp \ar[d]^{\pi_\Gamma} \\ \period_{(\M,D_\M), \Gamma_\M} \ar[r]_{\phi_{i, \Gamma}} & \period_{(\pp, D_\pp), \Gamma}}
\]
hence $\phi_{i, \Gamma}(\period_{(\M,D_\M), \Gamma_\M}) = \pi_\Gamma(\phi_i(D_\M))$. But $\phi_i$ is a complex-analytic closed immersion, hence $\phi_i(D_\M)$ is a smooth closed complex-analytic subset of $D_\pp$, and the result follows from the fact that $\pi_\Gamma$ is a local biholomorphism on the source.
\end{proof}
\begin{definition}
Let $\period_{(\pp, D_\pp), \Gamma}$ be a mixed Hodge variety, and $i : (\M,D_\M) \subset (\pp, D_\pp)$ a mixed Hodge subclass. The \textit{special subvariety of $\period_{(\pp, D_\pp), \Gamma}$ associated to $(\M,D_\M)$} is the smooth locally closed complex-analytic subset $\phi_{i, \Gamma}(\period_{(\M,D_\M), \Gamma_\M}) $ of $\period_{(\pp, D_\pp), \Gamma}$.
\end{definition}
We will also need the following notion in the proof of Theorem \ref{mainfull}:
\begin{definition}
Let $\period_{(\pp, D_\pp), \Gamma}$ be a mixed Hodge variety. A \textit{torsion point} in $\period_{(\pp, D_\pp), \Gamma}$ is a special subvariety associated to a torsion mixed Hodge subclass of $(\pp, D_\pp)$.
\end{definition}
\section{Variations of mixed Hodge structures and their period maps}\label{sect3}
Most of the material of this section is classical, see for example \cite{petste,pearls,klingler_atyp,klingao}. 
\subsection{Definitions}
\subsubsection*{Variations of mixed Hodge structures}
Let $S$ be a connected complex manifold. Following Steenbrink-Zucker in \cite{sz}, an \textit{integral graded-polarized variation of mixed Hodge structures over $S$} is the datum of
\begin{itemize}
\item a local system of torsion-free finitely generated abelian groups $\V$ over $S$;
\item an increasing exhausting finite filtration $\W_\bul$, the \textit{weight filtration}, of $\V_\Q$ by $\Q$-local subsystems;
\item a descending exhausting finite filtration $\mathcal{F}^\bul$, the \textit{Hodge filtration}, of $\mathcal{V} := \V \otimes_{\uZ} \mathcal{O}_S$ by locally free $\mathcal{O}_S$-submodules;
\item for each $k\in \Z$ a flat $(-1)^k$-symmetric bilinear form $\underline{q_k} : \Gr_k^\W\V_\Q \otimes _{\uQ} \Gr_k^\W\V_\Q \rightarrow  \uQ$;
\end{itemize}
subject to the following axioms:
\begin{itemize}
\item denoting $\nabla = id \otimes \mathrm{d}$ the integrable connection on $\mathcal{V}$ induced by the flat structure on $\V$, Griffiths transversality is satisfied:
\[
\nabla(\mathcal{F}^p) \subset \mathcal{F}^{p-1} \otimes_{\mathcal{O}_S} \Omega^1_S;
\]
\item for every $k \in \Z$ and every $s \in S(\C)$, the filtration induced by $\mathcal{F}_s$ on $\Gr_k^{\W_s} (\V_s)_\Q$ defines a pure Hodge structure of weight $k$ on the latter, polarized by $(\underline{q_k})_s$.
\end{itemize}
\begin{remark}
In the particular case where $\W$ consists of exactly one non-trivial step, one recovers the usual notion of integral polarized variation of (pure) Hodge structures, except that the polarization is only defined over $\Q$, and not over $\Z$ as usual. Of course it doesn't make any substantial difference, as this issue can be solved by scaling the polarization form by a positive integer, which preserves the axioms of polarizations.
\end{remark}
\subsubsection*{Admissibility}
In order to get a theory of the asymptotic behaviour of $\V$ at infinity which is similar to the pure case, one has to add a condition on $\V$ which we recall now, introduced by Steenbrink-Zucker in \cite{sz} when $S$ is a curve and generalized by Kashiwara in \cite{kashiwara} for higher dimensional $S$.

To recall it, we start with some linear algebra from \cite[§2]{sz}. Let $V$ be a finite dimensional vector space over a characteristic $0$ field and $N : V \rightarrow V$ a nilpotent endomorphism of $V$. As shown for example in \cite[Lem. 6.4]{schmid}, there exists a unique increasing finite filtration $M_\bul$ of $V$ such that for every $l \geqslant 0$, $N(M_l) \subseteq M_{l-2}$ and the induced map $N^l : \Gr^{M}_{l}V  \rightarrow \Gr^{M}_{-l}V$ is an isomorphism. It is called \textit{the weight filtration of $V$ for $N$}.

Assume furthermore, $V$ is endowed with an increasing finite filtration by sub-vector spaces $W_\bul$ such that for every $k \in \Z$ one has $N(W_k) \subset W_k$ so that $N$ induces an endomorphism ${^k}N$ of $\Gr^{W}_kV$. Denote by ${^k}M_\bul$ the weight filtration of $\Gr^{W}_kV$ for ${^k}N$. Then following \cite[Def. 2.5]{sz}, a \textit{weight filtration of $V$ for $N$ relative to $W_\bul$} is an increasing finite filtration $M_\bul$ of $V$ such that
\begin{itemize}
\item $N(M_l )\subset M_{l-2}$ for every $l \in \Z$;
\item for every $k \in \Z$, the filtration induced by $M_\bul$ on $\Gr^W_kV$ coincides with ${^k}M$.
\end{itemize}
By \cite[Prop. 2.6]{sz}, this filtration is unique if it exists.

Let $\V$ be an integral graded-polarized variation of mixed Hodge structures over the punctured disk $\Delta^\times$. Let $s \in \Delta^\times$ and $T\in \gl(\V_{s,\Q})$ the image of a generator of $\pi_1(\Delta^\times)$ under the monodromy representation at $s$ associated to $\V$. If $T$ is unipotent, Deligne proved in \cite{deledpsr} that the vector bundle with connection $(\mathcal{V} := \V \otimes_\Z \mathcal{O}_{\Delta^\times}, \nabla = \id \otimes d)$ extends uniquely as a vector bundle with connection with at worse logarithmic poles at $0$ denoted $(\overline{\mathcal{V}}, \bar{\nabla})$, over $\Delta$ such that the residue of the connection at $0$ is nilpotent. The residue of the connection at $0$ is then the nilpotent logarithm $N := \log(T)$, defined by the usual power series which converges because $T$ is unipotent. Furthermore, this construction defines an exact functor, so that for every $k \in \Z$ one has a canonical extension $\overline{\mathcal{W}}_k$ of $\mathcal{W}_k := \W_k \otimes_\Q \mathcal{O}_{\Delta^\times}$ which is naturally a sub-vector bundle of $\overline{\mathcal{V}}$. Finally, remark that since the weight filtration is by local sub-systems, the $(\W_k)_s$ are preserved by the action of $T$ hence by the action of $N$. So it makes sense to talk about the weight filtration of $\V_s$ for $N$ relative to $(\W_\bul)_s$

We can now turn to the promised definition. With the notations above, the variation $\V$ is called \textit{admissible} if the following axioms hold:
\begin{itemize} 
\item[$(A_0)$] $T$ is unipotent;
\item[$(A_1)$] There exists a weight filtration of $\V_s$ for $N$ relative to $(\W_\bul)_s$;
\item[$(A_2)$] The Hodge filtration $\mathcal{F}^\bul$ extends to a filtration $\overline{\mathcal{F}}^\bul$ of  $\overline{\mathcal{V}}$ by locally free $\mathcal{O}_\Delta$-modules such that for each $k,p \in \Z$ the $\mathcal{O}_\Delta$-module $\Gr_{\overline{\mathcal{F}}}^p \Gr_k^{\overline{\mathcal{W}}} \overline{\mathcal{V}}$ is locally free.
\end{itemize}

Now let $S$ be a smooth irreducible quasi-projective complex algebraic variety and choose a log-smooth compactification $(\overline{S},E)$ of $S$. Let $\V$ be an integral graded-polarized variation of mixed Hodge structures on $S$. It is called \textit{admissible with respect to $(\overline{S},E)$} if for every holomorphic map $i : \Delta \rightarrow \overline{S}$ which maps $\Delta^\times$ in $S$ such that the pulled-back variation $i^\ast\V$ has unipotent monodromy, the varation $i^\ast \V$ over $\Delta^\times$ is admissible in the above sense. The following ensures that this property doesn't depend on the initial choice of a compactification:
\begin{lemma}[{\cite[Prop. 4.2]{bbkt}}]
Let $(\overline{S}',E')$ be an other log-smooth compactification. Then $\V$ is admissible with respect to $(\overline{S}',E')$ if and only if it is admissible with respect to $(\overline{S},E)$.
\end{lemma}
We are lead to the:
\begin{definition}[{\cite[Def. 14.49]{petste}}]
The variation $\V$ is called \textit{admissible} if it is admissible with respect to some (hence any) log-smooth compactification of $S$.
\end{definition}
\subsection{Generic Mumford-Tate group and monodromy}\label{sect32}
Let $S$ be a smooth, irreducible and quasi-projective complex algebraic variety, and $\V$ a graded-polarized and integral variation of mixed Hodge structures over $S$. Let $s \in S^\an$. 
\subsubsection*{Generic Hodge tensors}
\begin{definition}
A \textit{generic Hodge tensor at $s$ for $\V$} is a tensor $\zeta \in T^{a,b}(\V_s)$ such that for any (possibly constant) path $\tau : s \rightarrow s'$ in $S^\an$, the parallel transport of $\zeta$ along $\tau$ to $T^{a,b}(\V_{s'})$ is a Hodge tensor for the mixed $\Q$-Hodge structure on $T^{a,b}(\V_{s'})$ induced from that on $\V_{s'}$.
\end{definition}
We will denote by $\hod^{\bul,\bul}(\V^\otimes)_s$ the graded vector space of generic Hodge tensors at $s$ for $\V$. There is a chain of inclusions of graded vector spaces
\[
\hod^{\bul,\bul}(\V^\otimes)_s \subset \hod^{\bul,\bul}(\V_s^\otimes) \subset T^{\bul,\bul}(\V_s)
\]
We will also sometimes refer to tensors in $\hod^{\bul,\bul}(\V_s^\otimes) -\hod^{\bul,\bul}(\V^\otimes)_s$ as \textit{exceptional Hodge tensors at $s$}. From the definition, $\hod^{\bul,\bul}(\V^\otimes)_s$ is preserved by the action of $\pi_1(S^\an,s)$ on $T^{\bul,\bul}(\V_s)$. A consequence of the results recalled in this paragraph is that this action factors through a finite quotient of $\pi_1(S^\an,s)$.
\begin{definition}
The \textit{generic Mumford-Tate group of $(\V,S)$ at $s$} is the subgroup $\pp_s$ of $\gl(\V_s)$ fixing all generic Hodge tensors at $s$ for $\V$.
\end{definition}
The following is immediate from the definition.
\begin{lemma}
Let $s \in S^\an$. Then $\pp(\V_s) \subseteq \pp_s$ as subgroups of $\gl(\V_s)$.
\end{lemma}
Before proceeding, we record the following two easy results for later use.
\begin{lemma}[{\cite[Lem. 4]{andre}}]\label{exgen}
There exists a countable union $S^\circ$ of proper closed analytic subvarieties of $S^\an$ such that for any $s \in S^\an - S^\circ$, the inclusion $\pp(\V_s) \subset \pp_s$ is an equality. 
\end{lemma}
\begin{definition}
A point $s \in S^\an$ such that the inclusion $\pp(\V_s) \subset \pp_s$ is an equality is called \textit{Hodge generic}
\end{definition}
\begin{lemma}\label{transmt}
Let $s,s' \in S^\an$ and $\tau : [0,1] \rightarrow S^\an$ be a continuous path joining $s$ to $s'$. Let $\tau_\ast^\gl : \gl(\V_s) \rightarrow \gl(\V_{s'})$ be the induced parallel transport isomorphism. Then $\tau_\ast^\gl(\pp_s) = \pp_{s'}$. In particular, $\dim \pp_s = \dim \pp_{s'}$.
\end{lemma}
\begin{proof}
Let $\tau_\ast : \V_s \rightarrow \V_{s'}$ be the parallel transport operator at the level of $\V$. By definition, for $g \in \gl(\V_s)$, one has $\tau_\ast^\gl(g) = \tau_\ast \circ g \circ \tau_\ast^{-1}$. Hence the image by $\tau_\ast^\gl$ of the fixator of tensors $(t_\alpha)_\alpha$ is the fixator of the tensors $(\tau_\ast(t_\alpha))_\alpha$. Since by definition, $\tau_\ast$ maps generic Hodge tensors at $s$ bijectively to generic Hodge tensors at $s'$, the result follows. 
\end{proof}
\subsubsection*{Algebraic monodromy group}
Let $\rho_s : \pi_1(S^\an,s) \rightarrow \gl(\V_s)$ the monodromy representation, at the base point $s$, associated to the local system underlying $\V$.
\begin{definition}
The algebraic monodromy group of $\V$ at $s$ is the identity connected component $\mon_s$ of the $\Q$-Zariski closure of $\im(\rho_s)$.
\end{definition}
Assume until the end of the subsection that $\V$ is admissible. The following result of André, due to Deligne in the pure case, is a cornerstone of the theory and will be of constant use in the sequel:
\begin{theorem}[\cite{andre}]\label{andre}
The algebraic monodromy group $\mon_s$ of $\V$ at $s$ is a normal subgroup of $\pp_s$, and is contained in the derived subgroup $\pp_s^\der$ of $\pp_s$. In particular, the radical of $\mon_s$ is unipotent.
\end{theorem}
\begin{proof}
The proof of \cite[Thm. 1]{andre} shows that $\mon_s$ is a subgroup of $\pp_s^\der$ and is normal in $\pp_s$. The second statement is \cite[Cor. 2]{andre}.
\end{proof}
Let us record the following fact which is an easy consequence of the previous result, and will be useful for us.
\begin{corollary}\label{neatetal}
There exists a finite étale cover $\pi : S'\rightarrow S$ such that after replacing $S$ by $S'$ and $\V$ by $\pi^\ast \V$, for any $s \in S^\an$, the image of $\rho_{s}$ is contained in a neat arithmetic subgroup of $\pp_{s}(\Q)$.
\end{corollary}
\begin{proof}
Let $s \in S^\an$ and let $\tilde{\mon}_s$ be the $\Q$-Zariski closure of $\im(\rho_s)$ in $\gl(V_\Q)$. It is a $\Q$-algebraic subgroup of $\gl(V_\Q)$ whose $\Q$-points contain $\im(\rho_s)$. Its identity connected component is by definition $\mon_s$. Therefore $\tilde{\mon}_s(\Q)/\mon_s(\Q)$ is finite. The kernel of the composition of $\rho_s$ with the projection $\tilde{\mon}_s(\Q) \rightarrow \tilde{\mon}_s(\Q)/\mon_s(\Q)$ is a finite index normal subgroup of $\pi_1(S^\an,s)$. Denote by $\pi : S' \rightarrow S$ the corresponding finite étale cover of $S$, whose construction is easily seen by direct computation not to depend on the choice of base point $s \in S^\an$. In the sequel of the proof, we replace $S$ by $S'$ and $\V$ by $\pi^\ast(\V)$. By construction of $\pi$, for any $s \in S^\an$ one has that $\im(\rho_s) \subset \mon_s(\Q) \subset \pp_s(\Q)$.

Since the local system $\V$ is integral one has on the other hand that $\im(\rho_s) \subset \gl(\V_{\Z,s})$. This implies that $\im(\rho_s)$ is contained in the arithmetic subgroup $\Gamma_0 = \GL(\V_{\Z,s}) \cap \pp_s(\Q)$ of $\pp_s(\Q)$. By \cite[Prop. 17.4]{borel} there exists a neat finite index subgroup $\Gamma$ of $\Gamma_0$, which is therefore a neat arithmetic subgroup $\Gamma$ of $\pp_s(\Q)$. After passing to the finite etale cover of $S$ corresponding to the kernel of $\rho_s : \pi_1(S^\an,s) \rightarrow \Gamma_0 \rightarrow \Gamma_0/\Gamma$, the image of $\rho_s$ becomes contained in $\Gamma$.
\end{proof}
\begin{remark}
Let  $s \in S^\an$, let $W_\bul$ be the weight filtration on $\V_s$ and $(q_k)$ the graded-polarization at $s$. As by definition $\W_\bul$ and $(\underline{q}_k)$ are flat, the image of $\rho_s$ lies in $\ppol(\Q)$ (without having to pass to a finite étale cover).
\end{remark}
\subsection{Generic mixed Hodge class and the normalized period map}
Let $\V$ be an admissible and graded-polarized integral variation of mixed Hodge structures over a smooth, irreducible and quasi-projective complex algebraic variety $S$. Fix a base point $s_0 \in S^\an$,  let $\pi : \widetilde{S^\an} \rightarrow S^\an$ be the universal cover of $S^\an$ and fix a lift $\tilde{s}_0$ of $s_0$ in $\widetilde{S^\an}$. Denote by $\rho = \rho_{s_0}$ be the monodromy representation at $s_0$, by $W_\bul = (\W_\bul)_{s_0}$ the weight filtration on $\V_{s_0}$, by $q_k = (\underline{q_k})_{s_0}$ ($k \in \Z$) the polarisations on the graded pieces and by $(h^{p,q})_{p,q \in \Z}$ the Hodge numbers of $(\V_{s_0},W_\bul, \mathcal{F}^\bul_{s_0})$.
\subsubsection*{Construction of the mixed Hodge data}
Let $M_{s_0}$ be the parameter space for mixed $\Q$-Hodge structures on $\V_{s_0}$ with weight filtration $W_\bul$, Hodge numbers $(h^{p,q})$ and graded-polarized by the $q_k$'s.  Let $m_0$ be the point of $M_{s_0}$ corresponding to the mixed $\Q$-Hodge structure $(\V_{s_0}, W_\bul, \mathcal{F}_{s_0})$, and let $x_0 : \s_\C \rightarrow \gl(\V_{s_0})_\C$ be the canonical bigraduation of $m_0$. 
\begin{lemma}
The morphism $x_0$ satisfies Pink's axioms for $\pp_{s_0}$ and $\pp^{M_{s_0}}$.
\end{lemma}
\begin{proof}
By flatness of the weight filtration and the polarization, one has $\pp_{s_0} \subset \pp^{M_{s_0}} := \ppol$. Furthermore, as $\pp(m_0) = \pp(\V_{s_0}) \subset \pp_{s_0}$ as subgroups of $\gl(\V_{s_0})$, one has by Proposition \ref{mtbound} that $x_0(\s_\C) \subset (\pp_{s_0})_\C \subset \pp^{M_{s_0}}_\C$. Hence by Lemma \ref{subgroup} the bigraduation $x_0$ satisfies Pink's axioms for $\pp_{s_0}$ and $\pp^{M_{s_0}}$, as claimed.
\end{proof}
 Let $X^{M_{s_0}}$ (resp. $X_{\pp_{s_0}}$) be the $\pp^{M_{s_0}}(\R)\uu_{\pp^{M_{s_0}}}(\C)$-orbit (resp. $\pp_{s_0}(\R)\uu_{\pp_{s_0}}(\C)$) of $x_0$ in $\Hom_\C(\s_\C, \pp^{M_{s_0}}_\C)$ (resp. in $\Hom_\C(\s_\C, (\pp_{s_0})_\C)$). Let ${M_{s_0}}^+$ be the connected component of ${M_{s_0}}$ containing $m_0$, and $D_{\pp_{s_0}}$ be the connected component of $\mathcal{D}_{\pp_{s_0}}$ containing $m_0$. By construction, the inclusion $\pp_{s_0} \hookrightarrow \pp^{M_{s_0}}$ makes $(\pp_{s_0}, X_{\pp_{s_0}}, D_{\pp_{s_0}})$ a mixed Hodge sub-datum of $(\pp^{M_{s_0}}, X^{M_{s_0}}, {M_{s_0}}^+)$. A priori, they both depend on the choice of $s_0$.
\subsubsection*{The lifted period map}
Since $\widetilde{S^\an}$ is simply-connected, the local system $\pi^\ast \V$ on $\widetilde{S^\an}$ is trivial and inherits a graded-polarized integral variation of mixed Hodge structures. By triviality of the local system $\pi^\ast \V$, there is for every $\tilde{s} \in \widetilde{S^\an}$ a canonical isomorphism $i_{\tilde{s},\tilde{s}_0} : \V_{\pi(\tilde{s})} = (\pi^\ast \V)_{\tilde{s}} \cong \V_{s_0}$. Let $\tilde{\Phi}_{\tilde{s}_0} : \widetilde{S^\an} \rightarrow M_{s_0}$ be the map which associates to $\tilde{s} \in \widetilde{S^\an}$ the mixed $\Q$-Hodge structure $(\V_{s_0}, W_\bul, i_{\tilde{s},\tilde{s}_0}(\mathcal{F}_{s}))$ image of the mixed $\Q$-Hodge structure on $(\pi^\ast \V)_{\tilde{s}}$ under the isomorphism $i_{\tilde{s}, \tilde{s}_0}$. It has Hodge numbers $(h^{p,q})$, has weight filtration $W_\bul$ and is graded-polarized by the $q_k$'s by flatness of $\W_\bul$ and $\underline{q_k}$. Hence $(\V_{s_0}, W_\bul, i_{\tilde{s},\tilde{s}_0}(\mathcal{F}_{s}))$ indeed defines a point of $M_{s_0}$. The following is immediate from the definitions:
\begin{lemma}\label{holper}
The map $\tilde{\Phi}_{\tilde{s_0}}$ is holomorphic and $\rho$-equivariant, where $\pi_1(S^\an,s_0)$ acts on $\widetilde{S^\an}$ by deck transformations, and $\pp^{M_{s_0}}(\Q)$ acts on the $\pp^{M_{s_0}}(\R)\uu_{\pp^{M_{s_0}}}(\C)$-homogeneous space $M_{s_0}$ as a subset of $\pp^{M_{s_0}}(\R)\uu_{\pp^{M_{s_0}}}(\C)$.
\end{lemma}
\begin{remark}
The $\rho$-equivariance of $\tilde{\Phi}_{\tilde{s_0}}$ ensures that for any arithmetic subgroup $\Gamma^{M_{s_0}}$ of $\pp^{M_{s_0}}(\R)$ containing $\im(\rho)$, the map $\tilde{\Phi}_{\tilde{s}_0}$ descends to an analytic map between orbifolds
\[
\Phi_{s_0} : S^\an \rightarrow \Gamma^{M_{s_0}} \quo M_{s_0}
\]
which is classically referred to as the period map of $\V$. However this construction does not exploit the fact that $\V$ might carry non-trivial generic Hodge tensors. The sequel of this subsection is dedicated to explain how one takes this information into account.
\end{remark}
As constructed, the map $\tilde{\Phi}_{\tilde{s}_0}$ depends on the choice of $\tilde{s}_0$. However the maps resulting from two choices of points of $\widetilde{S^\an}$ can be canonically identified. More precisely the construction immediately implies the following:
\begin{lemma}\label{idpermap}
Let $\tilde{s}_1 \in \widetilde{S^\an}$ be another base point, and $s_1 = \pi(\tilde{s}_1)$. Let $f_{\tilde{s}_0, \tilde{s}_1} : (\widetilde{S^\an}, \tilde{s}_0) \rightarrow (\widetilde{S^\an}, \tilde{s}_1)$ be the deck transformation of $\widetilde{S^\an}$ mapping $\tilde{s}_0$ to $\tilde{s}_1$. Let $g_{\tilde{s}_0, \tilde{s}_1} : M_{s_0} \rightarrow M_{s_1}$ the biholomorphism induced by the isomorphism $i_{\tilde{s}_0, \tilde{s}_1} : \V_{s_0} \rightarrow \V_{s_1}$. Then the diagram
\[
\xymatrix{\widetilde{S^\an} \ar[rr]^{\tilde{\Phi}_{\tilde{s}_0}} \ar[d]_{f_{\tilde{s}_0, \tilde{s}_1}} & & M_{s_0} \ar[d]^{g_{\tilde{s}_0, \tilde{s}_1}} \\ \widetilde{S^\an} \ar[rr]_{\tilde{\Phi}_{\tilde{s}_1}} & & M_{s_1}}
\]
commutes.
\end{lemma}
\subsubsection*{The normalized lifted period map}
Let $(\pp_{s_0}, D_{\pp_{s_0}})$ (resp. $(\pp^{M_{s_0}}, M_{s_0}^+)$) be the essential equivalence class of $(\pp_{s_0}, X_{\pp_{s_0}}, D_{\pp_{s_0}})$ (resp. $(\pp^{M_{s_0}}, X^{M_{s_0}}, M_{s_0}^+)$). 
\begin{lemma}\label{uniqmhc}
One has $\tilde{\Phi}_{\tilde{s_0}}(\widetilde{S^\an}) \subset D_{\pp_{s_0}}$. Furthermore:
\begin{itemize}
\item[(i)] $(\pp_{s_0}, D_{\pp_{s_0}})$ is the unique mixed Hodge subclass of $(\pp^{M_{s_0}}, M_{s_0}^+)$ with underlying inclusion $\pp_{s_0} \hookrightarrow \pp^{M_{s_0}}$ whose associated Mumford-Tate domain contains $\tilde{\Phi}_{\tilde{s}_0}(\widetilde{S^\an})$;
\item[(ii)] $(\pp_{s_0}, D_{\pp_{s_0}})$ is minimal for the inclusion among mixed Hodge subclasses of $(\pp^{M_{s_0}}, M_{s_0}^+)$ whose associated Mumford-Tate domain contains $\tilde{\Phi}_{\tilde{s}_0}(\widetilde{S^\an})$.
\end{itemize}
\end{lemma}
\begin{proof}
$(i)$ By definition of the generic Mumford-Tate group $\pp_{s_0}$ of $\V$ at $s_0$ and Lemma \ref{transmt}, every $h \in \tilde{\Phi}_{\tilde{s}_0}(\widetilde{S^\an})$ satisfies $\pp(h) \subset \pp_{s_0}$. By Corollary \ref{mtconn}, there exists a unique mixed Hodge subclass $(\pp_{s_0}, D'_{\pp_{s_0}})$ of $(\pp^{M_{s_0}}, M_{s_0}^+)$ with underlying inclusion $\pp_{s_0} \hookrightarrow \pp^{M_{s_0}}$ such that $\tilde{\Phi}_{\tilde{s}_0}(\widetilde{S^\an}) \subset D'_{\pp_{s_0}}$. As $m_0 \in \tilde{\Phi}_{\tilde{s}_0}(\widetilde{S^\an})\cap D_{\pp_{s_0}}$, one has $m_0 \in D_{\pp_{s_0}}' \cap D_{\pp_{s_0}} \neq \emptyset$ and the proof of Corollary \ref{mtconn} shows that this mixed Hodge subclass is precisely the essential equivalence class of the mixed Hodge sub-datum $(\pp_{s_0}, X_{\pp_{s_0}}, D_{\pp_{s_0}})$ of $(\pp^{M_{s_0}}, X^{M_{s_0}}, M_{s_0}^+)$, and we denote it by $(\pp_{s_0}, D_{\pp_{s_0}})$. 

$(ii)$ Assume that there is a mixed Hodge subclass $(\M,D_\M) \subseteq (\pp_{s_0}, D_{\pp_{s_0}})$ such that $\tilde{\Phi}_{\tilde{s}_0}(\widetilde{S^\an}) \subset D_\M$. By Lemma \ref{exgen}, there exists $s \in S^\an$ such that $\pp(\V_s) = \pp_s$. Therefore, for any lift $\tilde{s}$ of $s$ to $\widetilde{S^\an}$, one has $\pp(\tilde{\Phi}_{\tilde{s}_0}(\tilde{s})) = \pp_{s_0}$. On the other hand, by assumption $\tilde{\Phi}_{\tilde{s}_0}(\tilde{s}) \in D_\M$, so that by Lemma \ref{mtbound} one has $\pp(\tilde{\Phi}_{\tilde{s}_0}(\tilde{s})) \subset \M$ as subgroups of $\pp^{M_{s_0}}$. This proves $\pp_{s_0} \subset \M$ hence $\pp_{s_0} = \M$ as subgroups of $\pp^{M_{s_0}}$. By the uniqueness assertion $(i)$, this proves the desired minimality of $(\pp_{s_0}, D_{\pp_{s_0}})$.
\end{proof}
In particular, $\tilde{\Phi}_{\tilde{s}_0}$ factors uniquely through a holomorphic map $\widetilde{S^\an} \rightarrow D_{\pp_{s_0}}$, which we still denote by $\tilde{\Phi}_{\tilde{s}_0}$.
 \begin{corollary}
 The essential equivalence class of $(\pp_{s_0}, X_{\pp_{s_0}}, D_{\pp_{s_0}})$ does not depend on the choice of $s_0 \in S^\an$. 
 \end{corollary}
 \begin{proof}
Let $s_1 \in S^\an$ and $\tilde{s}_1 \in \widetilde{S^\an}$ be a lift of $s_1$. Unwrapping definitions, it suffices to prove that $g_{\tilde{s}_0, \tilde{s}_1}(D_{\pp_{s_0}}) \subset D_{\pp_{s_1}}$, where we use the notation introduced in Lemma \ref{idpermap}. This follows from the uniqueness statement in Lemma \ref{uniqmhc} and the commutativity of the diagram in Lemma \ref{idpermap}.
 \end{proof}
 Denote by $(\pp, D_\pp)$ the corresponding mixed Hodge class. Summing up, we have shown that for any $\tilde{s} \in \widetilde{S^\an}$, the map $\tilde{\Phi}_{\tilde{s}}$ factors as a holomorphic map $\widetilde{S^\an} \rightarrow D_{\pp_{\pi(\tilde{s})}}$ and that the map obtained for any other choice of base point $\tilde{s}' \in \widetilde{S^\an}$ can be canonically identified as in Lemma \ref{idpermap}. 
\begin{definition}
\begin{itemize}
\item[(a)] The mixed Hodge datum $(\pp_{s_0}, X_{\pp_{s_0}}, D_{\pp_{s_0}})$ is called the \textit{generic mixed Hodge datum of $\V$ at $s_0$};
\item[(b)] The mixed Hodge class $(\pp,D_\pp)$ is called the \textit{generic mixed Hodge class of $\V$};
\item[(c)] The map $\tilde{\Phi}_{\tilde{s}_0} : \widetilde{S^\an} \rightarrow D_{\pp_{s_0}}$ is called the \textit{(normalized) lifted period map of $\V$ at $\tilde{s}_0$}.
\end{itemize}
\end{definition}
As a consequence of Lemma \ref{holper} and Griffiths' transversality axiom, one finds:
\begin{lemma}\label{holpernorm}
The normalized lifted period map of $\V$ at $\tilde{s}_0$ is holomorphic and horizontal, i.e. for any $\tilde{s} \in \widetilde{S^\an}$ the derivative $d_{\tilde{s}} \tilde{\Phi}_{\tilde{s}_0}$ of $\tilde{\Phi}_{\tilde{s}_0}$ at $\tilde{s}$ has image in $T_{\tilde{\Phi}_{\tilde{s}_0}(\tilde{s})}^{\mathrm{horiz}} D_{\pp_{s_0}}$.
\end{lemma}
\subsubsection*{The normalized period map}
We saw above that $\tilde{\Phi}_{\tilde{s}_0}(\widetilde{S^\an}) \subset D_{\pp_{s_0}}$, which implies that
\[
\im \Phi^M \subset \phi_{i, \Gamma^M}(\period_{(\pp, D_\pp), \Gamma^M \cap \pp(\Q)^+}).
\]
In general $\Phi^M$ does not necessarily factor through a map $S^\an \rightarrow \period_{(\pp, D_\pp), \Gamma^M \cap \pp(\Q)^+}$, as it might happen that $\im(\rho)$ is not contained in full in $\pp(\R)^+$.  Assume that $\im(\rho)$ is contained in a neat arithmetic subgroup of $\pp(\R)^+$ which, by Corollary \ref{neatetal}, can be arranged after replacing $S$ by a finite étale cover and $\V$ by its pullback along this cover. Then $\tilde{\Phi}_{\tilde{s}_0} : \widetilde{S^\an} \rightarrow D_{\pp_{s_0}}$ is equivariant under $\rho : \pi_1(S^\an, s_0) \rightarrow \Gamma \subset \pp_{s_0}(\Q)^+$ and descends to a holomorphic map $\Phi : S^\an \rightarrow \period_{(\pp, D_\pp), \Gamma}$, which does not depend on the initial choice of $\tilde{s}_0$ as a consequence of Lemma \ref{idpermap} and the fact that $\Gamma$ contains all parallel transport operators.
\begin{definition}
Assume that $\im(\rho)$ is contained in a neat arithmetic subgroup $\Gamma$ of $\pp(\R)^+$. The holomorphic map between complex manifolds $\Phi : S^\an \rightarrow \period_{(\pp, D_\pp), \Gamma}$ constructed above is called \textit{the (normalized) period map of $\V$}.
\end{definition}
Remarkably, the operation that associates to $\V$ its normalized period map essentially loses no information, as the following classical result asserts: 
\begin{proposition}\label{pervarequiv}
Let $(\pp, D_\pp)$ be a mixed Hodge class, let $\Gamma \subset \pp(\Q)^+$ be a neat arithmetic subgroup and $s \in S^\an$. The map which associates to an integral and graded-polarized variation of mixed Hodge structures over $S$ with generic mixed Hodge class $(\pp, D_\pp)$ the normalized period map of $\V$ and the faithful representation $\pp \cong \pp_s \hookrightarrow \gl(\V_s)$ is a bijection between:
\begin{itemize}
\item the set of integral and graded-polarized variations of mixed Hodge structures on $S$ with generic mixed Hodge class $(\pp, D_\pp)$ and whose associated monodromy representation has image in $\Gamma$;
\item the set of pairs $(\Phi, \rho)$ where $\Phi$ is a holomorphic horizontal map from $S^\an$ to $\period_{(\pp, D_\pp), \Gamma}$ and $\rho$ is a faithful representation of $\pp$.
\end{itemize}
\end{proposition}
%
\subsection{Definable fundamental domains and algebraic period maps}
In general, a mixed Hodge variety $\period_{(\pp, D_\pp), \Gamma}$ is not algebraic but merely a (smooth) complex-analytic variety, and the period map is only a holomorphic map. However in recent years, the use of o-minimal techniques enabled Hodge theorists to bound the transcendence of these analytic objects. We recall the aspects relevant to this work in the present section.

We refer to \cite{vdd} for a general introduction to o-minimality and to \cite{klingicm} and references therein for a general discussion of the applications of o-minimality to Hodge theory. Our main reference for the theory of complex-analytic definable spaces for some o-minimal structure is \cite{bbt}. We start with some general recollection on fundamental sets in the o-minimal setting.
\begin{definition}[{\cite[Def. 2.1]{bbkt}}]\label{fundset}
Let $X$ be a locally compact Hausdorff topological space, endowed with a definable space structure (in an o-minimal structure) in the sense of \cite[Chap. 10, (1.2)]{vdd}. Let $\Gamma$ be a group acting on $X$ by definable homeomorphisms. A \textit{fundamental set} for the action of $\Gamma$ on $X$ is an open definable subset $\mathscr{F} \subseteq X$ such that:
\begin{itemize}
\item[(a)] $\Gamma \cdot \mathscr{F} = X$;
\item[(b)] the set $\{\gamma \in \Gamma \hspace{0.1cm} : \hspace{0.1cm} \gamma \cdot \mathscr{F} \cap \mathscr{F} \neq \emptyset\}$ is finite.
\end{itemize}
\end{definition}
The existence of such fundamental sets is enough to endow the quotient with a definable structure as the following asserts.
\begin{proposition}[{\cite[Prop. 2.3]{bbkt}}]\label{ominfund}
In the setting of Definition \ref{fundset}, if $\mathscr{F}$ is a fundamental set for the action of $\Gamma$ on $X$, there exists a unique structure of definable space (in the same o-minimal structure) on $\Gamma \quo X$ such that the restriction $\mathscr{F} \rightarrow \Gamma \quo X$ to $\mathscr{F}$ of the quotient map is definable.
\end{proposition}
Let $(\pp, D_\pp)$ be a mixed Hodge class and $\Gamma$ a neat arithmetic subgroup of $\pp(\R)^+$. By Proposition \ref{hodgedat}$(iv)$, the space $D_\pp$ is a connected component of the real semi-algebraic subset $\mathcal{D}_\pp$ of a projective complex algebraic variety, the flag variety $\check{\mathcal{D}}_\pp$. This endows $D_\pp$ with the structure of a definable complex-analytic space in the o-minimal structure $\alg$. Furthermore elements of $\Gamma$ act on $D_\pp$ by $\alg$-definable homeomorphisms, hence we are in the setting of Definition \ref{fundset}.
\begin{theorem}[{\cite[Thm. 4.4]{bbkt}}]\label{defpermap}
Let $\V$ be an admissible and graded-polarized integral variation of mixed Hodge structures on a smooth and irreducible quasi-projective complex algebraic variety $S$, let $(\pp, D_\pp)$ be its generic mixed Hodge datum, and assume that the image of the monodromy representation associated to $\V$ lies in a neat arithmetic subgroup $\Gamma$ of $\pp(\R)^+$. Let $\Phi : S^\an \rightarrow \period_{(\pp, D_\pp), \Gamma}$  be the associated period map. Then,
\begin{itemize}
\item[(i)] there exists an $\alg$-definable fundamental set $\mathscr{F}$ for the action of $\Gamma$ on $D_\pp$ such that the induced (by Proposition \ref{ominfund}) structure of $\alg$-definable space on $\period_{(\pp, D_\pp), \Gamma}$ is functorial for morphisms of mixed Hodge varieties;
\item[(ii)] the map $\Phi$ is definable in $\anexp$ for the structure of $\anexp$-definable space on $S^\an$ (resp. $\period_{(\pp, D_\pp), \Gamma}$) extending the structure of $\alg$-definable space coming from its structure of algebraic variety (resp. the structure of $\alg$-definable space defined in $(i)$).
\end{itemize}
\end{theorem}
This definability is key for some remarkable algebraization theorems in Hodge theory. We start by recalling the o-minimal Chow theorem of Peterzil-Starchenko (in the restricted generality that we need), which is not specific to Hodge theory but will be used in the sequel.
\begin{theorem}[{\cite[Cor. 4.5]{PS}, \cite[Thm. 3.3]{bbt}}]\label{petstar}
Fix an o-minimal structure. Let $X$ be a smooth complex algebraic variety and endow $X^\an$ with the induced structure of complex-analytic definable space in the fixed o-minimal structure. Let $Y \subset X^\an$ be a closed complex-analytic and definable subset of $X^\an$. Then $Y$ is the analytification of an algebraic subset of $X$. 
\end{theorem}
A remarkable application of Theorems \ref{defpermap} and \ref{petstar} is the following, which recovers a fundamental result of Cattani-Deligne-Kaplan to be discussed later:
\begin{theorem}[{\cite[Cor. 6.7]{bbkt}}] \label{algspec}
Let $(\M,D_\M) \subsetneq (\pp, D_\pp)$ be a strict mixed Hodge subclass. Then, any complex-analytic irreducible component of $\Phi^{-1}(\phi_{i, \Gamma}(\period_{(\M,D_\M), \Gamma_\M}))$ is an irreducible algebraic subvariety of $S$.
\end{theorem}
Building on a wide generalization of the result of Peterzil-Starchenko, Bakker-Brunebarbe-Tsimerman proved the following strong algebraicity result, conjectured by Griffiths in the case of pure variations of Hodge structures.
\begin{theorem}[{\cite[Thm. 1.1]{bbtmixte}}]\label{facto}
There exists a complex algebraic quasi-projective variety $T$, a closed definable complex-analytic immersion $i : T^{\mathrm{def}} \hookrightarrow \period_{(\pp,D_\pp), \Gamma}$ and a dominant regular morphism $f : S \rightarrow T$ such that $\Phi = i^\an \circ f^\an$.
\end{theorem}
We will sometimes use the above result through the following corollary:
\begin{corollary}\label{opensm}
There exists a non-empty Zariski open subset $S'$ of $S$ such that $\Phi((S')^\an)$ is a smooth locally closed subset of $\period_{(\pp, D_\pp),\Gamma}$.
\end{corollary}
\begin{proof}
Let $(T,i,f)$ as in Theorem \ref{facto}. As $f$ is dominant $f(S)$ contains a Zariski-dense Zariski-open subset $U$ of $T$. Let $U^\sm$ be the smooth locus of $U$ and $S' = f^{-1}(U^\sm)$ which is a non-empty Zariski open subset of $S$. By construction $\Phi((S')^\an) = i^\an((U^\sm)^\an)$ is a smooth locally closed complex analytic subset of $\period_{(\pp, D_\pp), \Gamma}$ as desired.
\end{proof}
\subsection{A preparation lemma}
The results of this work will be concerned with establishing density (for the Zariski or euclidean topology) and emptiness results for Hodge loci. The aim of this subsection is to explain that to prove such properties of Hodge loci one can freely replace $S$ by a finite étale cover or a non-empty Zariski-open subset, and to collect in a convenient definition all the assumptions that can be made on a mixed variation of Hodge structures up to making changes of $(S,\V)$ of the type described above.
\begin{lemma}\label{stabdens}
Let $S$ be a smooth and irreducible complex algebraic quasi-projective variety. Let $i : S' \rightarrow S$ be a dominant morphism. Let $\Sigma \subset S$ be any subset. Then:
\begin{itemize}
\item[(i)] if $i^{-1}(\Sigma)$ is dense in $S'^\an$ for the euclidean topology, then $\Sigma$ is dense in $S^\an$ for the euclidean topology;
\item[(ii)] if $i^{-1}(\Sigma)$ is Zariski dense in $S'$, then $\Sigma$ is Zariski dense in $S$.
\end{itemize}
In particular, this holds if $i$ is a finite étale cover or an open immersion.
\end{lemma}
\begin{proof}
Assume that $\Sigma$ is contained in a strict subset $Z$ of $S$ closed for the euclidean (resp. Zariski) topology. Then $i^{-1}(\Sigma)$ is contained in $i^{-1}(Z)$ which is a closed subset of $S'$ (for the same topology). One simply has to argue it is a strict subset. Since $i$ is dominant and $S$ is irreducible (hence a non-empty Zariski open subset of $S$ is dense in $S(\C)$ for both the euclidean and the Zariski topology) one has $Z \cap i(S') \subsetneq i(S')$ hence $i^{-1}(i(S')-Z) \neq \emptyset$. On the other hand, one tautologically has $i^{-1}(Z) \cap i^{-1}(i(S')-Z) = \emptyset$. This shows that $i^{-1}(Z) \subsetneq S'$.
\end{proof}
\begin{definition}
Let $\V$ be an integral variation of mixed Hodge structures on a smooth and irreducible quasi-projective variety $S$. We say that $\V$ is \textit{effective} if the following hold:
\begin{itemize}
\item[(a)] $\V$ is graded-polarized;
\item[(b)] $\V$ is admissible;
\item[(c)] the image of the monodromy representation of $\V$ at $s$ is contained in a neat arithmetic subgroup of $\pp_s(\Q)^+$ for some (equivalently any) $s \in S^\an$;
\item[(d)] the image of the period map of $\V$ is a smooth locally closed complex-analytic subset of $\period_{(\pp, D_\pp), \Gamma}$.
\end{itemize}
\end{definition}
Let $S$ be a smooth and irreducible quasi-projective variety and $\V$ a variation of mixed Hodge structures on $S$. Let $S' \subset S$ (resp. $\pi : S' \rightarrow S$) be a non-empty Zariski open subset (resp. a finite étale cover). Then the local system $\V \vert_{S'}$ (resp. $\pi^\ast \V$) is easily seen to still carry a variation of mixed Hodge structures. In the sequel, by "up to replacing $S$ by a non-empty Zariski open subset $S'$" we will mean replacing $(S, \V)$ by $(S', \V\vert_{S'})$. By "up to replacing $S$ by a finite étale cover $\pi : S' \rightarrow S$", we will mean replacing $(S, \V)$ by $(S', \pi^\ast \V)$.
\begin{proposition}\label{pbeff}
Let $S$ be a smooth and irreducible complex quasi-projective variety and $\V$ be an admissible and graded-polarized integral variation of mixed Hodge structures on $S$. After successively replacing finitely many times $S$ by non-empty Zariski-open subsets or finite étale covers, the variation $\V$ can be made effective.
\end{proposition}
\begin{proof}
Let $(S,\V)$ be as in the statement. The properties of being admissible and graded-polarized are preserved after passing to non-empty Zariski-open subset (\cite[Lem. 3.5]{gao15}) or to a finite étale cover (obvious from the definition as axioms $(A_1)$ and $(A_2)$ are insensitive to finite étale covers). We are therefore tasked with proving that properties $(c)$ and $(d)$ can be achieved after passing to a finite étale cover or to a non-empty Zariski-open subset. For $(c)$, this is precisely the content of Corollary \ref{neatetal}. For $(d)$ this is Corollary \ref{opensm}.
\end{proof}
\subsubsection*{The pushforward variation}
Let $\V$ be an effective integral variation of mixed Hodge structures on a smooth and irreducible quasi-projective complex algebraic variety $S$, denote by $(\pp, D_\pp)$ its generic mixed Hodge class and $\Phi : S^\an \rightarrow \period_{(\pp, D_\pp), \Gamma}$ its period map. By Theorem \ref{facto}, there exists a quasi-projective complex algebraic variety $Y$, a dominant morphism $f : S \rightarrow Y$ and a definable closed complex-analytic immersion $i : Y^\deef \hookrightarrow \period_{(\pp, D_\pp), \Gamma}^\deef$ such that $\Phi = i^\an \circ f^\an$. Since $\V$ is effective and $i^\an$ is a closed complex-analytic immersion, one finds that $Y_0 := f^\an(S^\an) = (i^\an)^{-1}(\Phi(S^\an))$ is a smooth locally-closed complex-analytic subset of $Y^\an$, hence a smooth quasi-projective variety. Denote $\Phi_0 = i^\an \vert_{Y_0} : Y_0 \rightarrow \period_{(\pp, D_\pp), \Gamma}$.
\begin{lemma}\label{auxper}
The map $\Phi_0$ is the period map of an effective integral variation of mixed Hodge structures $\V_0$ on $Y_0$.
\end{lemma}
\begin{proof}
Let $\rho : \pp \hookrightarrow \gl(V)$ be a faithful representation associated with a point $s \in S^\an$. By Proposition \ref{pervarequiv} the pair $(\Phi_0, \rho)$ defines an integral and graded-polarized variation of mixed Hodge structures $\V_0$ on $Y_0$ with normalized period map $\Phi_0$. We need to prove that $\V_0$ is admissible. Let $\overline{Y_0}$ (resp. $\overline{S}$) be log-smooth compactifications of $Y_0$ and $S$ such that $f : S \rightarrow Y_0$ extends as $\bar{f} : \overline{S} \rightarrow \overline{Y_0}$ (such compactifications exist, see the proof of \cite[Lem. 3.5]{gao15} for instance). Let $j : \Delta \rightarrow \overline{Y_0}^\an$ be a holomorphic map such that $j(\Delta^\times) \subset (Y_0)^\an$ and $j^\ast \V_0$ has unipotent monodromy. Since $f$ surjects $S$ onto $Y_0$, there exists an integer $k \geqslant 1$ such that denoting by $\pi_k : \Delta^\times \rightarrow \Delta^\times, z \mapsto z^k$ the corresponding cover of $\Delta^\times$, the map $j$ admits (up to shrinking $\Delta$ and restricting $j$ to avoid the branch locus of $f$) a lift $\tilde{j} : \Delta \rightarrow \overline{S}^\an$ which fits in a commutative diagram
\[
\xymatrix{\Delta^\times \ar[r]^{\tilde{j}} \ar[d]_{\pi_k} & S^\an \ar[d]^f \\ \Delta^\times \ar[r]_{j} & (Y_0)^\an}.
\]
Since by construction of $\V_0$ one has $\V = (f^\an)^\ast \V_0$, the commutativity of the above diagram shows that $\tilde{j}^\ast \V = (\pi_k)^\ast(j^\ast \V_0)$. By assumption $j^\ast \V_0$ has unipotent monodromy hence the latter identity ensures that $\tilde{j}^\ast \V$ has unipotent monodromy. It then follows from the admissibility of $\V$ that $\tilde{j}^\ast \V$ is admissible. As properties $(A_1)$ and $(A_2)$ are insensitive to finite étale covers, it follows that $j^\ast \V_0$ is admissible. This finishes the proof.
\end{proof}

\section{Hodge loci as intersection loci}\label{sect4}
The aim of this section is to recall the realization made in \cite{klingler_atyp} of the Hodge locus as union of intersection loci, so-called special subvarieties of $S$ for $\V$ and the notion of (a)typicality introduced in \cite{bu} generalizing \cite{bku}.
\subsection{The Hodge locus}
Let $\V$ be an integral, graded-polarized and admissible variation of mixed Hodge structures on a smooth irreducible quasi-projective complex variety $S$. Let $(\pp, D_\pp)$ be the generic mixed Hodge class of $(S, \V)$. The \textit{Hodge locus (of $S$ for $\V$)} is the subset 
\[\HL = \{s \in S^\an \hspace{0.1cm} : \hspace{0.1cm} \hod^{\bul,\bul}(\V^\otimes)_s \subsetneq \hod^{\bul,\bul}(\V_s^\otimes)\} = \{s \in S^\an \hspace{0.1cm} : \hspace{0.1cm}\pp(\V_s) \subsetneq \pp_s\} \]
of $S$. 
%
The following is immediate from the definitions:
\begin{lemma}\label{hlpb}
Let $i : S' \rightarrow S$ be either an open immersion or a finite étale cover. Then $\mathrm{HL}(S', (i^\ast \V)^\otimes) = i^{-1}(\HL)$.
\end{lemma}
Combining the latter to Lemma \ref{stabdens} and Proposition \ref{pbeff}, one can always assume that $\V$ is effective in order to prove (complex-analytic or Zariski) density results of the full Hodge locus or of the Hodge locus of some type.
\subsection{Mixed Hodge classes of subvarieties}
Let $Z \subsetneq S$ be a strict irreducible subvariety and denote by $Z^\sm$ its smooth locus.
\subsubsection*{Construction}
\begin{definition}
The generic mixed Hodge class of $Z$ is the generic mixed Hodge class $(\pp_Z, D_{\pp_Z})$ of $\V\vert_{Z^\sm}$.
\end{definition}
Let $s_0 \in Z^\sm(\C)$. One has an inclusion $i_{Z,s_0} : \pp_{Z,s_0} \subset \pp_{s_0}$ and the following is immediate from the construction of the generic mixed Hodge datum of a variation at some base point.
\begin{lemma}\label{min}
The inclusion $i_{Z,s_0}$ realizes $(\pp_{Z,s_0}, X_{\pp_{Z,s_0}}, D_{\pp_{Z,s_0}})$ as a mixed Hodge subdatum of $(\pp_{s_0}, X_{\pp_{s_0}}, D_{\pp_{s_0}})$, and the realization of $(\pp_Z, D_{\pp_Z})$ as a mixed Hodge subclass of $(\pp, D_\pp)$ does not depend on the choice of $s_0$.
\end{lemma}
From now on, we identify $(\pp_Z, D_{\pp_Z})$ as a mixed Hodge subclass of $(\pp, D_\pp)$ via the inclusion described in the previous lemma.
\begin{lemma}\label{minsubv}
Let $Z$ be an irreducible subvariety of $S$. Then $\Phi(Z^\an) \subset \phi_{i, \Gamma}(\period_{(\pp_Z, D_{\pp_Z}), \Gamma_{\pp_Z}})$
\end{lemma}
\begin{proof}
As $\Phi^{-1}(\phi_{i, \Gamma}(\period_{(\pp_Z, D_{\pp_Z}), \Gamma_{\pp_Z}}))$ is a closed complex-analytic subset of $S^\an$, it suffices to prove that it contains $Z^\sm(\C)$. Let $s_0 \in Z^\sm(\C)$, let $\pi_Z : \widetilde{(Z^\sm)^\an} \rightarrow (Z^\sm)^\an$ be the universal cover of $(Z^\sm)^\an$, let $\tilde{s}_0$ be a lift of $s_0$ to $\widetilde{(Z^\sm)^\an}$ and let $\widetilde{S^\an}$. From the constructions, one has the following commutative diagram in the complex-analytic category:
\[
\xymatrix{(Z^\sm)^\an \ar[d]^i & \widetilde{(Z^\sm)^\an} \ar[l]_{\pi_Z}\ar[r]^{\tilde{\Phi}_{Z, \tilde{s}_0}} \ar[d]^{\tilde{i}} & D_{\pp_{Z,s_0}} \ar[d] \ar[rd] & \\ S^\an & \widetilde{S^\an} \ar[r]_{\tilde{\Phi}_{\tilde{i}(\tilde{s}_0)}} \ar[l]^{\pi_S} & D_{\pp_{s_0}} \ar[r] & M_{s_0}}
\]
where we denote by $i : (Z^\sm)^\an \rightarrow S^\an$ the inclusion, by $\tilde{i} : \widetilde{(Z^\sm)^\an}  \rightarrow \widetilde{S^\an}$ the natural (complex-analytic) map (which set theoretically maps the homotopy class of a continuous path in $(Z^\sm)^\an$ based at $s_0$ to the homotopy class of the same path seen as a continuous path in $S^\an$), by $\tilde{\Phi}_{Z, \tilde{s}_0}$ the lifted period map of $\V\vert_{Z^\sm}$ based at $\tilde{s}_0$ and by $\tilde{\Phi}_{\tilde{i}(\tilde{s}_0)}$ is the lifted period map of $\V$ based at $\tilde{i}(\tilde{s}_0)$. The triangle on the right consists of the obvious inclusions. The commutativity of the right square ensures that $\tilde{\Phi}_{\tilde{i}(\tilde{s}_0)}(\tilde{i}(\widetilde{(Z^\sm)^\an}))$ is contained in $D_{\pp_{Z,s_0}} \subset D_{\pp_{s_0}}$. The commutativity of the left square ensures that $\pi_S(\tilde{i}(\widetilde{(Z^\sm)^\an})) = Z^\sm(\C)$. Combining the two shows that $\Phi(Z^\sm(\C)) \subset \phi_{i, \Gamma}(\period_{(\pp_Z, D_{\pp_Z}), \Gamma_{\pp_Z}})$ as desired.
\end{proof}
\subsubsection*{Functoriality}
\begin{lemma}\label{fonc}
Let $Z \subset Z'$ be two nested irreducible subvarieties of $S$. Then $(\pp_Z, D_{\pp_Z}) \subseteq (\pp_{Z'}, D_{\pp_{Z'}})$ as mixed Hodge subclasses of $(\pp, D_\pp)$.
\end{lemma}
The only non-trivial case is when $Z \subset Z' - (Z')^\sm$, because then there is no easily comparable representatives of $(\pp_Z, D_{\pp_Z})$ and $(\pp_{Z'}, D_{\pp_{Z'}})$. This motivates the following construction.

Let $Z$ be an irreducible subvariety of $S$. For $s \in Z^\sm(\C)$ denote by $\hod^{\bul,\bul}(\V^\otimes)_{Z,s}$ the graded vector space $\hod^{\bul,\bul}((\V\vert_{Z^\sm})^\otimes)_s$ of generic Hodge tensors at $s$ for $\V\vert_{Z^\sm}$ as defined in Sect. \ref{sect32}. By definition, $\pp_{Z,s}$ is the fixator in $\gl(\V_s)$ of tensors in $\hod^{\bul,\bul}(\V^\otimes)_{Z,s}$. For a singular point $s \in Z(\C)-Z^\sm(\C)$, choose a simply-connected open neighborhood $U$ of $s$ in $S^\an$ and some $s' \in U \cap Z^\sm(\C)$. Let $\tau$ be a continuous path in $U$ from $s'$ to $s$ (there is a unique homotopy class of such paths). Define $\hod^{\bul,\bul}(\V^\otimes)_{Z,s}$ as the parallel transport of $\hod^{\bul,\bul}(\V^\otimes)_{Z,s'}$ along $\tau$. One easily checks that this doesn't depend on the choice of $U$ and $s'$. Let $\pp_{Z,s}$ be the fixator in $\gl(\V_s)$ of tensors in $\hod^{\bul,\bul}(\V^\otimes)_{Z,s}$. By construction, for $U$, $s' \in Z^\sm(\C) \cap U$ and $\tau$ as above, the isomorphism $\tau_\ast^\gl : \gl(\V_{s'}) \cong \gl(\V_s)$ given by parallel transport along $\tau$ induces an isomorphism $\pp_{Z,s'} \cong \pp_{Z,s}$ which is independent of the above choices.
\begin{lemma}\label{inclmtz}
Let $s \in Z^\an$. Then $\pp(\V_s) \subset \pp_{Z,s} \subset \pp_s$ as subgroups of $\gl(\V_s)$. Furthermore, there exists a point $s \in Z^\an$ such that the inclusion $\pp(\V_s) \subseteq \pp_{Z,s}$ is an equality.
\end{lemma}
\begin{proof}
The claimed chain of inclusions is equivalent to
\[
\hod^{\bul,\bul}(\V^\otimes)_s \subset \hod^{\bul,\bul}(\V^\otimes)_{Z,s} \subset \hod^{\bul,\bul}(\V_s^\otimes)
\]
The first inclusion is true by definition. The second one is a closed condition on $s\in Z^\an$ (it amounts to asking that some tensors are Hodge) and is true by definition for $s \in Z^\sm(\C)$. As $Z$ is irreducible the subset $Z^\sm(\C)$ is dense in $Z^\an$ for the euclidean topology, and the first assertion follows.

The second part of the statement is Lemma \ref{exgen} applied to $\V\vert_{Z^\sm}$, using that $Z^\sm(\C) \neq \emptyset$.
\end{proof}
Therefore, for $s \in Z^\an$, the first inclusion shows that the canonical bigraduation of $(\V_s, (\W_\bul)_s, \mathcal{F}^\bul_s)$ factors through $(\pp_{Z,s})_\C$ and denoting $(\pp_{Z,s}, X_{\pp_{Z,s}}, D_{\pp_{Z,s}})$ the associated mixed Hodge datum, we have proved so far the following generalization of Lemma \ref{min} to singular points:
\begin{lemma}
For any $s \in Z^\an$, the mixed Hodge subdatum  $(\pp_{Z,s}, X_{\pp_{Z,s}}, D_{\pp_{Z,s}})$ represents $(\pp_Z, D_{\pp_Z})$ and the inclusion $i_{Z,s} : \pp_{Z,s} \hookrightarrow \pp_s$ realizes the inclusion of mixed Hodge classes $(\pp_Z, D_{\pp_Z}) \hookrightarrow (\pp, D_\pp)$
\end{lemma}
\begin{proof}[Proof of Lemma \ref{fonc}]
Let $Z \subset Z'$ be nested irreducible subvarieties of $S$. We need to prove that for some $s \in Z^\an$, one has the inclusion $\pp_{Z,s} \subset \pp_{Z',s}$ as subgroup of $\pp_s$. Let $s \in Z^\an$ such that $\pp(\V_s) = \pp_{Z,s}$ (it exists by Lemma \ref{inclmtz}). Since $s \in Z'^\an$, we find that $\pp(\V_s) \subset \pp_{Z',s}$. Combining both gives the desired result.
\end{proof}
\subsection{Special subvarieties}
\begin{definition}
A strict irreducible subvariety $Z$ of $S$ is a \textit{special subvariety of $S$ for $\V$} if for any intermediate (for the inclusion) irreducible subvariety $Z \subsetneq Z' \subset S$, the inclusion
\[
(\pp_Z, D_{\pp_Z}) \subseteq (\pp_{Z'}, D_{\pp_{Z'}})
\]
of mixed Hodge subclasses of $(\pp, D_\pp)$ is strict.
\end{definition}
The link between special subvarieties and the Hodge locus of $S$ for $\V$ relies on the following fundamental result, first proven in \cite{bps} and reproven recently independently in \cite{bbkt} (see Theorem \ref{algspec}):
\begin{theorem}[{\cite[Thm. 1]{bps}}]\label{bps}
The Hodge locus $\HL$ of $S$ for $\V$ is a countable union of irreducible algebraic subvarieties of $S$ for $\V$.
\end{theorem}
\begin{corollary}
The Hodge locus $\HL$ of $S$ for $\V$ is the union of special subvarieties of $S$ for $\V$.
\end{corollary}
\begin{proof}
Let $Z \subset S$ be a special subvariety of $S$ for $\V$. Let $s \in Z^\an$. Applying the definition of special subvarieties to $Z' = S$, one finds that $\pp(\V_s) \subset \pp_{Z,s} \subsetneq \pp_s$. It follows that $s \in \HL$.

Conversely, by Theorem \ref{bps} the Hodge locus is a countable union of irreducible algebraic subvarieties of $S$ for $\V$. Let $Z$ be any of these. We are tasked with proving that $Z$ is contained in a special subvariety of $S$ for $\V$. If it is itself a special subvariety, there is nothing to prove. Otherwise, there is some $Z \subsetneq Z' \subset S$ such that for any $s \in Z^\an$ one has $(\pp_{Z}, D_{\pp_Z}) = (\pp_{Z'}, D_{\pp_{Z'}})$ as mixed Hodge subdata of $(\pp, D_\pp)$. Since $Z \subset \HL$, one has for $(\pp_{Z}, D_{\pp_Z}) \subsetneq (\pp, D_\pp)$, hence $(\pp_{Z'}, D_{\pp_{Z'}}) \subsetneq (\pp, D_\pp)$. Therefore, for any $s \in Z'^\an$, one has $\pp(\V_s) \subset \pp_{Z',s} \subsetneq \pp_s$ hence $Z' \subset \HL$. One can repeat this process with $Z$ replaced by $Z'$ and so on. The process ends because $S$ is noetherian, and outputs a special subvariety of $S$ for $\V$ containing $Z$. This finishes the proof.
\end{proof}
Thereby, the study of the Hodge locus boils down to that of special subvarieties. The latter have a nice characterization formulated by Klingler in \cite{klingler_atyp}, which was the main motivation of all the work we have done so far. Until the end of this subsection, we assume that $\V$ is effective, denote by $(\pp, D_\pp)$ the generic mixed Hodge class of $\V$ and $\Phi : S^\an \rightarrow \period_{(\pp, D_\pp), \Gamma}$ the period map of $\V$, where $\Gamma \subset \pp(\Q)^+$ is a neat arithmetic subgroup containing the image of the monodromy representation.
\begin{proposition}\label{specarac}
Let $Z \subsetneq S$ be a strict irreducible subvariety. Then, $Z$ is a special subvariety of $S$ for $\V$ if and only if $(\pp_Z,D_{\pp_Z})$ is a strict mixed Hodge subclass of $(\pp, D_\pp)$ and $Z^\an$ is a complex-analytic irreducible component of $\Phi^{-1}(\phi_{i, \Gamma}(\period_{(\pp_Z, D_{\pp_Z}), \Gamma_{\pp_Z}}))$.
\end{proposition}
\begin{proof}
Assume that $(\pp_Z, D_{\pp_Z})$ is a strict mixed Hodge subclass of $(\pp, D_\pp)$ and that $Z^\an$ is a complex-analytic irreducible component of $\Phi^{-1}(\phi_{i, \Gamma}(\period_{(\pp_Z, D_{\pp_Z}), \Gamma_{\pp_Z}}))$. Let $Z \subset Z' \subset S$ be an intermediate irreducible subvariety such that $(\pp_{Z'}, D_{\pp_{Z'}}) = (\pp_Z, D_{\pp_Z})$. By Lemma \ref{minsubv}, it follows that $\Phi(Z'^\an) \subset  \phi_{i, \Gamma}(\period_{(\pp_Z, D_{\pp_Z}), \Gamma_{\pp_Z}})$, hence $Z = Z'$ by the assumption that $Z^\an$ is an irreducible component of $\Phi^{-1}(\phi_{i, \Gamma}(\period_{(\pp_Z, D_{\pp_Z}), \Gamma_{\pp_Z}})$. This proves that $Z$ is special and thereby the "if" part of the statement.

Conversely, assume that $Z$ is special. This readily forces $(\pp_Z, D_{\pp_Z}) \subsetneq (\pp, D_\pp)$. Let $Z'^\an$ be a complex-analytic irreducible component of $\Phi^{-1}(\phi_{i, \Gamma}(\period_{(\pp_Z, D_{\pp_Z}), \Gamma_{\pp_Z}})$ containing $Z^\an$. Theorem \ref{algspec} guarantees that $Z'^\an$ is the analytification of an irreducible algebraic subvariety $Z'$ of $S$ (hence the notation). We want to prove that $(\pp_Z , D_{\pp_Z}) = (\pp_{Z'}, D_{\pp_{Z'}})$. Let $s \in Z^\an$ such that $\pp(\V_s) = \pp_{Z,s}$, whose existence is guaranteed by Lemma \ref{inclmtz}. By the proof of Lemma \ref{fonc}, the inclusion $i_{Z,s} : \pp_{Z,s} \hookrightarrow \pp_s$ factors through $i_{Z',s} : \pp_{Z',s} \hookrightarrow \pp_s$, resulting in an inclusion $i_{Z,Z',s} : \pp_{Z,s} \hookrightarrow \pp_{Z',s}$ as subgroups of $\pp_s$. Furthermore $i_{Z,Z',s}$ induces an inclusion of mixed Hodge data $(\pp_{Z,s}, X_{\pp_{Z,s}}, D_{\pp_{Z,s}}) \hookrightarrow (\pp_{Z',s}, X_{\pp_{Z',s}}, D_{\pp_{Z',s}})$ as mixed Hodge subdata of $(\pp_s, X_{\pp_s}, D_{\pp_s})$. Therefore, one has an inclusion $(\pp_Z, D_{\pp_Z}) \subseteq (\pp_{Z'}, D_{\pp_{Z'}})$ as mixed Hodge subclasses of $(\pp, D_\pp)$, and to prove that this inclusion is an equality, it suffices to prove that $i_{Z,Z',s} : \pp_{Z,s} \hookrightarrow \pp_{Z',s}$ is an equality. To see this, let $s' \in Z'^\an$ be a point such that $\pp(\V_{s'}) = \pp_{Z',s'}$ (again, it is possible by Lemma \ref{inclmtz}). Let $\tau : [0,1] \rightarrow Z'^\an$ be a continuous path with $\tau(0) = s$ and $\tau(1) = s'$. Let $\tilde{s} \in \widetilde{S^\an}$ be a lift of $s$ to the universal cover, and $\tilde{s'}$ be the lift of $s'$ to $\widetilde{S^\an}$ which is the endpoint of the canonical lift of $\tau$ to $\widetilde{S^\an}$. Let $\tilde{\Phi}_{\tilde{s}}$ be the lifted period map based at $\tilde{s}$. Since $\Phi(Z'^\an) \subset \phi_{i, \Gamma}(\period_{(\pp_Z, D_{\pp_Z}), \Gamma_{\pp_Z}})$, there exists a $\gamma \in \Gamma$ such that $\tilde{\Phi}_{\tilde{s}}(\tilde{s}') \in \gamma \cdot D_{\pp_{Z,s}} \subset D_{\pp_s}$. From the construction of the lifted period map, and Proposition \ref{mtbound} one finds:
\[
(\tau_\ast^\gl)^{-1}(\pp(\V_{s'})) = \pp(\tilde{\Phi}_{\tilde{s}}(\tilde{s}')) \subset \gamma \pp_{Z,s} \gamma^{-1},
\]
as subgroups of $\pp_s$. As we chose $s'$ such that $\pp(\V_{s'}) = \pp_{Z',s'}$, and since by construction $(\tau_\ast^\gl)^{-1}(\pp_{Z',s'}) = \pp_{Z',s}$, we find:
\[
\pp_{Z,s} \subset \pp_{Z',s} \subset \gamma \pp_{Z,s} \gamma^{-1}
\]
as subgroups of $\pp_s$. The first and last members of the latter chain of inequality have the same dimension, and all members are connected $\Q$-algebraic subgroups, hence they are equal which proves in particular that $\pp_{Z,s} = \pp_{Z',s}$ as subgroups of $\pp_s$. We have shown that $(\pp_Z , D_{\pp_Z}) = (\pp_{Z'}, D_{\pp_{Z'}})$. Since $Z$ was assumed to be special, it follows that $Z = Z'$. Therefore $Z^\an$ is a complex-analytic irreducible component of $\Phi^{-1}(\phi_{i, \Gamma}(\period_{(\pp_Z, D_{\pp_Z}), \Gamma_{\pp_Z}})$ which proves the "only if" part of the statement.
\end{proof}
\subsection{(A)typical and transverse special subvarieties}
\subsubsection*{Recollections on transverse intersections}
We start by introducing some language to work with transverse intersections, which will be useful in the proofs of the main results. First recall the following standard statement in complex geometry
\begin{lemma}[{\cite[Thm. II.6.2, Prop. II.4.32]{demaillybook}}]\label{intersection}
Let $X,Y$ be irreducible locally closed complex-analytic subsets of a complex manifold $Z$. Let $W$ be a complex-analytic irreducible component of $X \cap Y$. Then:
\begin{itemize}
\item[(i)] $W$ is an irreducible locally closed complex-analytic subset of $Z$;
\item[(ii)] $\dim W \geqslant \dim X + \dim Y - \dim Z$;
\end{itemize}
\end{lemma}
In the setting where Lemma \ref{intersection}(ii) is an equality, we say that \textit{$X$ and $Y$ intersect transversally in $Z$ along $W$.} The following closely related differential theoretic definition of this notion will be crucial:
\begin{definition}
Let $f : X \rightarrow Z$ be a smooth map between smooth manifolds, let $x \in X$ and let $Y$ be a submanifold of $Z$. The map $f$ is said to be \textit{transverse to $Y$ at $x$} if 
\[
T_{f(x)} Y + \im(d_x f) = T_{f(x)} Z.
\]
\end{definition}
The following is immediate from the definitions:
\begin{lemma}[{\cite[Chap. I,\S 5.]{pollguill}}]\label{equivdifftrans}
Let $X,Y$ be smooth and irreducible locally closed complex-analytic subsets of a complex manifold $Z$, and let $x \in X \cap Y$. Consider the following assertions:
\begin{itemize}
\item[(i)] there exists a complex-analytic irreducible component $W$ of $X \cap Y$ containing $x$ such that $X$ and $Y$ intersect transversally along $W$ in $Z$;
\item[(ii)] the inclusion $i : X_\R \rightarrow Z_\R$ (seen as a smooth map between smooth manifolds) is transverse to $Y_\R$ at $x$.
\end{itemize}
Then $(ii)$ implies $(i)$.
\end{lemma}
However the converse does not hold in general (for example of the parabola and its complex tangent line at $(0,0)$ in $\C^2$). We will explain in Proposition \ref{typtotrans} how to bypass this issue in our Hodge theoretic setting.
The pertinence of this differential theoretic formulation of transversality for this work comes from the following stability result:
\begin{theorem}[{\cite[Chap. 1,\S 6, Stability Theorem]{pollguill}}]\label{stabtrans}
Let $X,S$ and $Z$ be smooth manifolds, let $Y$ be a smooth submanifold of $Z$ and let $F : X \times S \rightarrow Z$ be a smooth map. Let $\pi : X \times S \rightarrow S$ be the projection to $S$, and for $s \in S$ let $f_s : X \rightarrow Z, x \mapsto F(x,s)$. Assume that there exists a point $(x_o,o) \in X \times S$ such that $f_o$ is transverse to $Y$ at $x_o$. Then, there exists
\begin{itemize}
\item[(i)] an open neighborhood $\Omega$ of $o$ in $S$,
\item[(ii)] a smooth map $\sigma : \Omega \rightarrow X$ such that $\sigma(o) = x_o$,
\end{itemize}
such that for every $s \in \Omega$, the smooth map $f_s$ is transverse to $Y$ at $\sigma(s)$.
\end{theorem}
\begin{proof}
Using \textit{op. cit.}, it suffices to construct an open neighborhood $\Omega$ of $o$ and map $\sigma : \Omega \rightarrow X$ with $\sigma(o) = x_o$ such that for any $s \in S$ one has $F(\sigma(s),s) \in Y$. For this, it suffices to show that the restriction of $\pi$ to $F^{-1}(Y)$ is a submersion at $(x_o,o)$. This follows easily from a dimension count, using that $F$ is transverse to $Y$ at $(x_o,o)$ since $f_o$ is transverse to $Y$ at $x_o$.
\end{proof}
\subsubsection*{(A)typical special subvarieties and the Zilber-Pink Conjecture}
Back to our original setting, let $Z$ be a special subvariety of $S$ for $\V$. Let $(\pp_Z, D_{\pp_Z}) \subsetneq (\pp,D_\pp)$ be its generic mixed Hodge class. By Proposition \ref{specarac} one has that $\Phi(Z^\an)$ is a complex-analytic irreducible component of the intersection of the locally-closed complex-analytic subsets $\Phi(S^\an)$ and $\phi_{i,\Gamma}(\period_{(\pp_Z, D_{\pp_Z}), \Gamma_{\pp_Z}})$ of the complex manifold $\period_{(\pp,D_\pp), \Gamma}$. The following definition was introduced in spirit in \cite{klingler_atyp}, later refined by \cite{bku} in the pure case and finally appeared in the present form in \cite{bu} for the mixed setting:
\begin{definition}
A special subvariety $Z$ of $S$ for $\V$ is called
\begin{itemize}
\item[(a)]\textit{atypical} if either $\Phi(Z^\an)$ is contained in the singular locus of $\Phi(S^\an)$ or 
\[
\dim \Phi(Z^\an) > \dim \Phi(S^\an) + \dim \phi_{i,\Gamma}(\period_{(\pp_Z, D_{\pp_Z}), \Gamma_{\pp_Z}} - \dim\period_{(\pp,D_\pp), \Gamma}.
\]
\item[(b)] \textit{typical} otherwise.
\end{itemize}
The union of atypical (resp. typical) special subvarieties of $S$ for $\V$ is called the \textit{atypical} (resp. \textit{typical}) \textit{Hodge locus of $S$ for $\V$} and is denoted $\HL_\atyp$ (resp. $\HL_\typ$).
\end{definition}
With this definition in hand, we can state the main conjecture motivating this work:
\begin{conjecture}\label{zpconj}
Let $\V$ be an admissible integral and graded-polarized variation of mixed Hodge structures on a smooth and irreducible quasi-projective complex algebraic variety $S$. Then
\begin{itemize}
\item[(a)] $\HL_\atyp$ is a Zariski-closed subset of $S$;
\item[(b)] If $\HL_\typ$ is non-empty, it is dense in $S$ for the Zariski topology.
\end{itemize}
\end{conjecture}
Part $(a)$ of Conjecture \ref{zpconj} was stated in this generality in \cite{bu}, refining the statement in \cite{klingler_atyp}. Part $(b)$ is a direct generalization of a conjecture stated in \cite{bku}, and is the main object of study of the present work which gives strong evidence for its validity. In particular we show that $(a)$ implies $(b)$.
\subsubsection*{Transverse special subvarieties}
The purpose of this section is to define the notion of transverse special subvariety, a weakened notion of typicality, which appears to us to be the right setting to study density phenomena for Hodge loci (see Remark \ref{typtrans}). We start with the following:
\begin{definition}
Let $Z$ be an irreducible algebraic subvariety of $S$ and $(\M,D_\M)$ a mixed Hodge subclass of $(\pp, D_\pp)$. We say that $Z$ is \textit{defined by $(\M, D_\M)$} if it is a complex-analytic irreducible component of $\Phi^{-1}(\phi_{i, \Gamma}(\period_{(\M,D_\M), \Gamma_\M})$.
\end{definition}
\begin{remark}\label{defgen}
By Proposition \ref{specarac}, a special subvariety of $S$ for $\V$ is always defined by its generic mixed Hodge class, and an irreducible subvariety of $S$ defined by a strict mixed Hodge subclass of $(\pp,D_\pp)$ is special. We however emphasize on the fact that the usefulness of this notion is that a special subvariety of $S$ for $\V$ can also be defined by a Hodge subclass of $(\pp,D_\pp)$ which contains strictly its generic Hodge class (see Remark \ref{typtrans}).
\end{remark}
\begin{definition}\label{deftrans}
Let $Z$ be a special subvariety of $S$ for $\V$ and $(\M, D_\M)$ a mixed Hodge subclass of $(\pp, D_\pp)$. We say that $Z$ is $(\M,D_\M)$-\textit{transverse} if 
\begin{itemize}
\item $\Phi(Z^\an)$ is not contained in the singular locus of $\Phi(S^\an)$;
\item $Z$ is defined by $(\M,D_M)$;
\item $\dim \Phi(Z^\an) = \dim \Phi(S^\an) + \dim D_\M - \dim D$.
\end{itemize}
The \textit{transverse Hodge locus of type $(\M,D_\M)$} is the union $\HLM_\trans$ of special subvarieties of $S$ for $\V$ which are $(p\M p^{-1},p \cdot D_M)$-transverse for some $p \in \pp(\Q)^+$. The \textit{(full) transverse Hodge locus} is
\[
\HL_\trans = \bigcup_{(\M,D_\M) \subsetneq (\pp, D_\pp)} \HLM_\trans.
\]
\end{definition}
\begin{remark}\label{typtrans}
Obviously, a typical special subvariety of $S$ for $\V$ is automatically transverse but the resulting containment $\HL_\typ \subset \HL_\trans$ is strict in general as a special subvariety defined transversally by some mixed Hodge datum $(\M,D_\M)$ may in principle have a generic mixed Hodge datum strictly smaller than $(\M,D_\M)$ (in which case it is atypical). However, assuming Conjecture \ref{zpconj}$(a)$ one has that $\HL_\typ$ is dense in $S$ for the Zariski (resp. euclidean) topology if and only if $\HL_\trans$ is. Ultimately the necessity of introducing the notion of transverse special subvarieties comes from the lack of current knowledge on the distribution of $\HL_\atyp$, in particular its zero-dimensional part (see \cite{bu} for more detail).
\end{remark}
To conclude this section, we prove the following technical result that will be crucial to work with the differential theoretic notion of transversality. Its statement and proof arose from discussions with O. Gabber and D. Urbanik, but any mistake is my responsibility.
\begin{proposition}\label{typtotrans}
Let $(\pp, D_\pp)$ be a mixed Hodge class and $(\M, D_\M)$ be a mixed Hodge subclass. Let $\I \subset D_\pp$ be a smooth locally-closed complex analytic subset which is definable in $\anexp$ and assume that there exists a $p_0 \in \pp(\R)^+\pp^u(\C)$ and a complex-analytic irreducible component $Z$ of $\I \cap (p \cdot D_\M)$ such that $\I$ and $p \cdot D_\M$ intersect transversally along $Z$. There exists a non-empty euclidean open neighborhood $B$ of $p_0$ in $\pp(\R)^+ \pp^u(\C)$ and a (strict) subset $B_{\mathrm{crit}} \subset B$ of Lebesgue measure $0$ such that for every $p \in B - B_{\mathrm{crit}}$, the inclusion $i_p : p \cdot D_\M \hookrightarrow D_\pp$ is transverse to $\I$ at every $z \in \I \cap (p \cdot D_\M)$.
\end{proposition}
\begin{proof}
Consider the diagram
\[
\xymatrix{\pp(\C) \times D_\M \ar[r]^{\hspace{0.6cm}F} \ar[d]_\pi & \check{D}_\pp \\ \pp(\C) & }
\]
in the category of $\anexp$-definable complex-analytic spaces. Let $X_\I = F^{-1}(\I)$. Since $F$ is holomorphic and everywhere a submersion, the subset $X_\I$ of $\pp(\C) \times D_\M$ is a locally closed complex-analytic analytic of dimension $d + \dim \pp(\C)$ where $d = \dim D_\M + \dim \I - \dim D_\pp$. Let $\pi_\I : X_\I \rightarrow \pp(\C)$ be the restriction of $\pi$ to $X_\I$ and let $P_\I = \pi_\I(X_\I)$. Then $P_\I$ is a definable subset of $\pp(\C)$.

We first claim that it contains a non-empty euclidean open neighborhood of $p_0$. For $z \in X_\I$ let $d_z$ be the dimension of the germ of $\pi_\I^{-1}(\pi_\I(z))$ at $z$. By \cite[V.3, Cor. 2]{compan}, one has $d_z \geqslant d$ for every $z \in X_\I$. Furthermore, by assumption, there exists a $z_0 \in X_\I$ such that $d_{z_0} = d$ and $\pi_\I(z_0) = p_0$. By upper semi-continuity of $z \in X_\I \mapsto d_z$ (\cite[V.3, Thm. 3]{compan}), there exists a non-empty euclidean open neighborhood $B_0$ of $z_0$ such that for any $z \in B_0$ one has $d_z = d$. In particular $p_0$ is in the interior of $P_\I$ hence by definable cell decomposition (\cite[III. (2.11)]{vdd}), our claim would follow from the fact that $P_\I$ has dimension $\dim \pp(\C)$. Let $P_\I(2d)$ be the definable subset of $P_\I$ consisting of those $p \in U_\I$ such that $\dim \pi_\I^{-1}(p) = 2d$ (this is the definable dimension which is a real invariant hence the factor $2$). On the one hand $B_0$ is contained in $\pi_\I^{-1}(P_\I(2d))$ hence 
\[
\dim \pi_\I^{-1}(P_\I(2d)) = \dim_\R X_\I = 2d + \dim_\R \pp(\C).
\]
On the other hand, by \cite[Chap. 4, (1.5)]{vdd}, we have
\[
\dim \pi_\I^{-1}(P_\I(2d)) = 2d + \dim P_\I.
\]
Combining the above two equalities then yields $\dim U_\I = \dim_\R \pp(\C)$, and the claim follows.

Let $B' \subset \pp(\C)$ be a non-empty euclidean open neighborhood of $p_0$ in $P_\I$ and let $B = B' \cap \pp(\R)^+\pp^u(\C)$. Then $B$ is a non-empty euclidean open neighborhood of $p_0$ in $\pp(\R)^+\pp^u(\C)$ such that for any $p \in B$ the subsets $p \cdot D_\M$ and $\I$ of $D_\pp$ have non-empty intersection. Consider now the diagram of smooth manifolds
\[
\xymatrix{\pp(\R)^+\pp^u(\C) \times D_\M \ar[r]^{\hspace{0.7cm}\sigma} \ar[d]_q & D_\pp \\ \pp(\R)^+ & }
\]
and let $Y_\I = \sigma^{-1}(\I)$ and $q_\I : Y_\I \rightarrow \pp(\R)^+\pp^u(\C)$ the restriction of $q$ to $Y_\I$. Then $Y_\I$ is a smooth submanifold of $\pp(\R)^+\pp^u(\C) \times D_\M$ as $\sigma$ is everywhere a submersion, and $B$ is contained in $q_\I(Y_\I)$ by construction. By Sard's theorem \cite[40]{pollguill}, the set of critical values of $q_\I$ in $B$ is the complement of a Lebesgue measure $0$ subset $B_{\mathrm{crit}}$ of $B$. Therefore, every $p \in B - B_{\mathrm{crit}}$ is a regular value of $q_\I$. To prove the statement, it remains to show that a $p \in \pp(\R)^+\pp^u(\C)$ is a regular value of $q_\I$ if and only if the inclusion $i_p : p \cdot D_\M \hookrightarrow D_\pp$ is transverse to $\I$ at any point of $\I \cap (p \cdot D_\M)$, which is an easy dimension computation and is left to the reader.
\end{proof}
\section{Weakly special subvarieties and Ax-Schanuel}\label{sect5}
We now recollect the necessary notions and results about weakly special subvarieties. This is simply to set notations as there are very good accounts of the theory in the literature, see notably \cite{bu} for a thorough discussion of the Ax-Schanuel theorem. 
\subsection{Weakly special subvarieties as intersection loci}
Let $\V$ be an integral, graded-polarized and admissible variation of mixed Hodge structures on a smooth, irreducible and quasi-projective complex algebraic variety $S$. Let $i : Z \hookrightarrow S$ be the inclusion of an irreducible subvariety. The \textit{algebraic monodromy group of $Z$ for $\V$ at $s \in Z^\sm(\C)$} is the identity component $\mon_{Z,s}$ of the Zariski closure of the image of 
\[
r_{s} = \rho_s \circ i_\ast \circ n_\ast  : \pi_1((Z^\mathrm{nor})^\an, s) \rightarrow \pi_1(Z^\an,s) \rightarrow \pi_1(S^\an,s) \rightarrow \gl(\V_s),
\]
where $\rho_s : \pi_1(S^\an,s) \rightarrow \gl(\V_s)$ is the monodromy representation of $\V$ at $s$ and $n : Z^\mathrm{nor} \rightarrow Z$ is the normalization of $Z$, which is one-to-one when restricted to $n^{-1}(Z^\sm)$ allowing to abuse notation by writing $s$ for the $\C$-point of $Z^\mathrm{nor}$ corresponding to $s \in Z^\sm(\C)$. One easily checks that $\dim \mon_{Z,s}$ doesn't depend on the choice of $s \in Z^\sm(\C)$, and we denote it by $\dim \mon_Z$. This allows to make sense of the following:
\begin{definition}
Let $Z \subset S$ be a strict irreducible subvariety. It is called \textit{weakly special} if for any intermediate (for the inclusion) irreducible subvariety $Z \subsetneq Z' \subseteq S$, one has:
\[
\dim \mon_{Z} < \dim \mon_{Z'}.
\]
\end{definition}
Weakly special subvarieties have an intersection theoretic characterization similar to Proposition \ref{specarac} which will be of constant use in Sect. \ref{sect7}. We recall it now.
\begin{definition}[{\cite[Def. 2.5]{klingao}}]
Let $(\pp, D_\pp)$ be a mixed Hodge class. A \textit{weak Mumford-Tate subdomain} of $D_\pp$ is the $\NN(\R)^+\NN^u(\C)$-orbit of some $h \in D_\M$, where $(\M, D_\M)$ is a mixed Hodge subclass of $(\pp, D_\pp)$ and  $\NN$ is a normal $\Q$-subgroup of $\M$ whose radical is unipotent. The datum $((\M, D_\M),\NN,h)$ is called the \textit{weak Hodge class} of the generic weak Mumford-Tate subdomain.
\end{definition}
One has:
\begin{proposition}[{\cite[Prop. 2.6]{klingao}}]
Let $(\pp,D_\pp)$ be a mixed Hodge class. A generic weak Mumford-Tate subdomain of $D_\pp$ is a smooth closed complex-analytic subspace of $D_\pp$.
\end{proposition}
Let $(\pp, D_\pp)$ be a mixed Hodge class and $\NN$ be a normal subgroup of $\pp$. Gao and Klingler define in \cite[Prop. 5.1]{klingao} a "quotient" mixed Hodge class $(\pp/\NN, D_\pp/\NN)$ together with a holomorphic map $q_{\NN, D_\pp} : D_\pp \rightarrow D_\pp/\NN$ such that each fiber of $q_{\NN, D_\pp}$ is an $\NN(\R)^+\NN^u(\C)$-orbit in $D_\pp$, and conversely any $\NN(\R)^+\NN^u(\C)$-orbit of $D_\pp$ arises as a fiber of $q_{\NN, D_\pp}$. Notably, one has:
\begin{lemma}\label{weakmtcar}
Let $(\pp, D_\pp)$ be a mixed Hodge class and $(\M, D_\M)$ a mixed Hodge subclass with underlying inclusion $i : \M \hookrightarrow \pp$. The weak Mumford-Tate subdomain $D$ of $D_\pp$ with weak Hodge class $((\M,D_\M),\NN,h)$ is the fiber $q_{\NN, D_\M}^{-1}(q_{\NN, D_\M}(h))$ of $q_{\NN, D_\M}$ seen as a subset of $D_\pp$ through the natural inclusion $D_\M \hookrightarrow D_\pp$.
\end{lemma}
Let $(\pp, D_\pp)$ be a mixed Hodge class, let $(\M, D_\M)$ be a mixed Hodge subclass with underlying inclusion $i : \M \hookrightarrow \pp$, let $\NN$ be a normal subgroup of $\M$ whose radical is unipotent, let $\Gamma$ be a neat arithmetic subgroup of $\pp(\Q)^+$ and $\pi_\Gamma : D_\pp \rightarrow \period_{(\pp, D_\pp), \Gamma}$ be the quotient by $\Gamma$. Then \[\Gamma_{\M/\NN} := \Big[\Gamma \cap \M(\Q)/\Gamma \cap \NN(\Q)\Big] \cap (\M/\NN)(\Q)^+\] is a neat arithmetic subgroup of  $(\M/\NN)(\Q)^+$ (see \cite[Chap. III, \S 7]{amrt}). By \cite{klingao}, the quotient $q_{\NN, D_\M} : D_\M \rightarrow D_\M/\NN$ descends to a holomorphic map $p_{\NN, D_\M, \Gamma} : \period_{(\M,D_\M), \Gamma_\M} \rightarrow \period_{(\M/\NN, D_\M/\NN), \Gamma_{\M/\NN}}$.
\begin{definition}
Let $\period_{(\pp, D_\pp), \Gamma}$ be a mixed Hodge variety. A \textit{weak Hodge subvariety} of $\period_{(\pp, D_\pp), \Gamma}$ is a complex-analytic locally closed subset arising as the image under $\pi_\Gamma$ of a weak Mumford-Tate subdomain of $D_\pp$. We denote by $\mathscr{W}_{((\M, D_\M), \NN, h)}$ the image under $\pi_\Gamma$ of the weak Mumford-Tate subdomain of $D_\pp$ with weak Hodge class $((\M, D_\M), \NN, h)$.
\end{definition}
As in Lemma \ref{weakmtcar}, one has:
\begin{lemma}
Let $\period_{(\pp, D_\pp), \Gamma)}$ be a mixed Hodge variety, let $(\M, D_\M)$ be a mixed Hodge subclass with underlying inclusion $i : \M \hookrightarrow \pp$ and a let $\NN$ be a normal subgroup of $\M$. Then $\mathscr{W}_{((\M, D_\M), \NN, h)}$ is the image under $\phi_{i, \Gamma}$ of a fiber of $p_{\NN, D_\M, \Gamma}$.
\end{lemma}
Record the following, which is a direct consequence of the theorem of the fixed part:
\begin{lemma}[{\cite[Thm. 7.12]{bz}}]\label{monorb}
Let $Z \subsetneq S$ be an irreducible subvariety whith generic mixed Hodge class is $(\pp_Z, D_{\pp_Z})$, and denote by $i_Z : \pp_Z \hookrightarrow \pp$ the inclusion in the generic Mumford-Tate domain. Let $\mon_Z$ be the algebraic monodromy group of $Z$, which is a normal subgroup of $\pp_Z$ whose radical is unipotent. Then, the set $p_{\mon_Z, D_{\pp_Z}, \Gamma}(\phi_{i_Z, \Gamma}^{-1}(\Phi(Z^\an)))$ is a singleton $\{[h_Z]\}$.
\end{lemma}
We can now turn to the intersection-theoretic characterization of generic weakly special subvarieties of $S$ for $\V$ of \cite{klingler_atyp}. 
\begin{proposition}\label{weakcarac}
Let $Z \subsetneq S$ be an irreducible subvariety with generic mixed Hodge class $(\pp_Z, D_{\pp_Z})$ and denote by $i_Z : \pp_Z \hookrightarrow \pp$ the inclusion in the generic Mumford-Tate group. Then $Z$ is a weakly special subvariety of $S$ for $\V$ if and only if $Z$ is a complex-analytic irreducible component of $\Phi^{-1}(\mathscr{W}_{((\pp_Z, D_{\pp_Z}), \mon_Z, h_Z)})$.
\end{proposition}
\subsection{Mixed Ax-Schanuel theorem}
We now recall the statement of the Ax-Schanuel theorem for admissible and graded-polarized integral variations of mixed Hodge structure, which is a key ingredient to our proofs. It was first shown independently in \cite{klingao} and \cite{chiu}. Note that after these proofs, a remarkable new perspective on Ax-Schanuel type theorems was discovered in \cite{diffaxschan} yielding a new proof of this statement in \cite{btandre}. 

Start with some notations. Lemma \ref{monorb} applied to $Z = S$ ensures the existence of a point $h_S \in D_\pp/\mon_S$ such that letting $D_{\mon_S} := q_{\mon_S}^{-1}(h_S)$ , the following diagram commutes:
\[
\xymatrix{\widetilde{S^\an} \ar[rr]^{\tilde{\Phi}} \ar[d]_{\pi_S} & & D_{\mon_S} \subset D_\pp \ar[d]^{\pi_\Gamma} \\ S^\an \ar[rr]_{\Phi} & & \period_{(\pp, D_\pp), \Gamma}}
\]
\begin{theorem}[Ax-Schanuel Theorem]\label{axschan}
Let $\Delta = S \times_{\period_{(\pp, D_\pp), \Gamma}} D_{\mon_S}$ and $W \subset S \times D_{\mon_S}$ be an algebraic subvariety. Let $p_S : S \times D_{\mon_S} \rightarrow S$ be the projection to $S$. If a complex-analytic irreducible component $U$ of $\Delta \cap W$ satisfies
\[
\dim U > \dim \Delta + \dim W - \dim S \times D_{\mon_S}
\]
then $p_S(U)$ is contained in a strict weakly special subvariety of $S$ for $\V$.
\end{theorem}
\subsection{Geometric Zilber-Pink}
A remarkable use of the Ax-Schanuel theorem was the proof by Baldi-Urbanik of their Geometric Zilber-Pink theorem. We recall its general statement and explain how it relates to the Zilber-Pink conjecture.
\begin{definition}[{\cite{bu}}]\label{monodratyp}
A weakly special subvariety $Z$ of $S$ for $\V$ is called \textit{monodromically atypical} if
\[
\dim \Phi(Z^\an) > \dim \Phi(S^\an) + \dim \mathscr{W}_{((\pp_Z, D_{\pp_Z}), \mon_Z, h_Z)} - \dim \mathscr{W}_{((\pp, D_\pp), \mon, h_S)}.
\]
Furthermore, it is called \textit{maximal} if it is maximal for the inclusion among monodromically atypical weakly special subvarieties of $S$ for $\V$.
\end{definition}
With this terminology in hand, Baldi-Urbanik's Geometric Zilber-Pink theorem reads:
\begin{theorem}[{\cite{bu}}]\label{geozp}
There exists a finite set $\Sigma_{(S,\V)}$ of pairs $((\M, D_\M), \NN)$ where $(\M, D_\M)$ is a mixed Hodge subclass of $(\pp, D_\pp)$ and $\NN$ is a normal subgroup of $\M$ whose radical is unipotent, such that for any maximal monodromically atypical weakly special subvariety $Z$ of $S$ for $\V$ there exists $((\M, D_\M), \NN) \in \Sigma_{(S, \V)}$ such that $\NN = \mon_Z$ as subgroups of $\pp$.
\end{theorem}
This statement in particular contains finiteness statements for some classes of atypical special subvarieties of $S$ for $\V$, which is the form in which geometric Zilber-Pink theorems first appeared (see for instance \cite[Thm. 1.3]{axlindul} and \cite[Thm. 3.1]{bku}). Let $\Sigma_{(S,\V)}^\mathrm{gen}$ be the set of triples $((\M, D_\M) , \NN) \in \Sigma_{(S, \V)}$ such that $(\M, D_\M) = (\pp, D_\pp)$ and $\Sigma_{(S, \V)}^\mathrm{sp} = \Sigma_{(S,\V)} - \Sigma_{(S,\V)}^\mathrm{gen}$.
\begin{definition}
A special subvariety $Z$ of $S$ for $\V$ has \textit{$\V$-factorwise positive dimension} if for every $(\pp, D_\pp), \NN) \in \Sigma{(S, \V)}^\mathrm{gen}$ the inequality
\[
\dim \Phi_{/\NN}(Z^\an) > 0
\]
holds. 
\end{definition}
\begin{remark}\label{finposdim}
In particular if a special subvariety $Z$ of $S$ for $\V$ satisfies $\dim \Phi_{/\NN}(Z^\an) > 0$, it has $\V$-factorwise positive dimension. The latter has the advantage of being a condition which depends only on the generic Mumford-Tate group $\pp$ (which is why we used this notion in the introduction) but has the inconvenient of consisting in a priori infinitely many conditions.
\end{remark}
We denote by $\HL_{\mathrm{atyp}, \mathrm{f-pos}}$ the union of special subvarieties of $S$ for $\V$ which are atypical and have $\V$-factorwise positive dimension. One then deduces the following finiteness result from Baldi-Urbanik's Geometric Zilber-Pink theorem, which somehow motivates the name of the result: it gives a finiteness result for maximal atypical special subvarieties which are not constructed from atypical special points for sub-quotient variations.
\begin{corollary}\label{zpcor}
There exists a finite set $\{Z_1, \cdots, Z_l\}$ of (strict) special subvarieties of $S$ for $\V$ such that 
\[
\HL_{\mathrm{atyp}, \mathrm{f-pos}} \subset \bigcup_{k = 1}^l Z_k.
\]
In particular $\HL_{\mathrm{atyp}, \mathrm{f-pos}}$ is not dense in $S$ for the Zariski topology.
\end{corollary}
\begin{proof}
Let $\{Z_1, \cdots, Z_l\}$ be the set of special subvarieties of $S$ for $\V$ arising as irreducible components of preimages under $\Phi$ of Hodge subvarieties of $\period_{(\pp, D_\pp, \Gamma)}$ of the form $\pi_\Gamma(D_\M)$ with for $((\M, D_\M), \NN)\in \Sigma_{(S, \V)}^\mathrm{sp}$. These are indeed strict special subvarieties of $S$ for $\V$ by Proposition \ref{specarac}, using the fact that for $((\M, D_\M), \NN)\in \Sigma_{(S, \V)}^\mathrm{sp}$, one has $(\M, D_\M) \subsetneq (\pp, D_\pp)$. Theey are in finite number because $\Sigma_{(S, \V)}^\mathrm{sp}$ is finite by Theorem \ref{geozp}. We will prove that $\HL_{\mathrm{atyp}, \mathrm{f-pos}} \subset \bigcup_{k = 1}^l Z_k.$

Let $Z$ be an atypical special subvariety of $S$ for $\V$ which has $\V$-factorwise positive dimension. We need to show the existence of a $k \in \{1, \cdots, l\}$ such that $Z \subset Z_k$.

We first claim that $Z$ is a monodromically atypical weakly special subvariety of $S$ for $\V$. By Proposition \ref{specarac} we know that $Z$ is an irreducible component of $\Phi^{-1}(\phi_{i, \Gamma}(\period_{(\pp_Z, D_{\pp_Z}), \Gamma_{\pp_Z}}))$, and by Lemma \ref{monorb} one has
\[
\Phi(Z^\an) \subset \mathscr{W}_{((\pp_Z, D_{\pp_Z}), \mon_Z, h_Z)} \subset \phi_{i, \Gamma}(\period_{(\pp_Z, D_{\pp_Z}), \Gamma_{\pp_Z}})).
\]
This shows that $Z$ is an irreducible component of $\Phi^{-1}( \mathscr{W}_{((\pp_Z, D_{\pp_Z}), \mon_Z, h_Z)})$ hence is a weakly special subvariety. It remains to show that it is monodromically atypical. For this, since $Z$ is atypical by asumption, it suffices to show that
\[
\dim D_{\pp_Z} - \dim D_\pp \geqslant \dim q_{\mon_Z, \pp_Z}^{-1}(h_Z) - \dim q_{\mon, \pp}^{-1}(h_S).
\]
The proof is parallel to \cite[Lem. 5.6]{bku}, we recall it briefly. The above inequality is equivalent to
\[
\dim D_{\pp_Z}/\mon_Z \geqslant \dim D_\pp/\mon.
\]
Since $\mon_Z \subset \mon$, the inclusion $D_{\pp_Z} \hookrightarrow D_\pp$ descends to a map $\bar{i} : D_{\pp_Z}/\mon_Z \rightarrow D_\pp/\mon$. The image of $\bar{i}$ is a Mumford-Tate subdomain of $D_\pp/\mon$ which contains $h_S = \bar{i}(h_Z)$. Since $(\pp, D_\pp)$ is the generic mixed Hodge class of $\V$ the point $h_S$ cannot be contained in a proper Mumford-Tate subdomain of $D_\pp/\mon$. Therefore $\bar{i}$ is a surjection and the claimed dimension inequality follows.

Since $Z$ is a monodromically atypical weakly special subvariety of $S$ for $\V$, it is contained in a maximal one, say $Y$, and by the Geometric Zilber-Pink Theorem \ref{geozp} there exists $((\M, D_\M), \NN) \in \Sigma_{(S, \V)}$ such that $\mon_Y = \NN$. We claim that $((\M, D_\M), \NN)$ actually belongs to $\Sigma_{(S, \V)}^\mathrm{sp}$. Indeed, if $((\M, D_\M), \NN) \in \Sigma_{(S, \V)}^\mathrm{gen}$ the fact that $\mon_Y = \NN$ forces $\dim \Phi_{/\NN}(Y^\an) = 0.$ Since $Z \subset Y$, this implies that $\dim \Phi_{/\NN}(Z^\an) = 0$ which contradicts the assumption that $Z$ has $\V$-factorwise positive dimension. Unwrapping definitions, this in turn implies that $Y$ hence $Z$ is contained in $Z_k$ for some $k \in \{1, \cdots, l\}$, as desired. The non-density statement follows immediately.
\end{proof}
\section{A purely unipotent example}\label{sect6}
In this section, we investigate in a concrete example what can be reasonably expected regarding the distribution of transverse Hodge loci, beyond the pure case. We start by recalling the situation in the pure case. In \cite{ku} we introduced a condition called factorwise $\V$-admissibility on any strict Hodge sub-datum $(\M,D_\M)$ of the generic Hodge datum $(\pp, D)$ of an integral variation of Hodge structure $\V$ on a smooth and irreducible quasi-projective complex variety $S$. In substance, we proved there that the transverse Hodge locus of type $\M$ of $S$ for $\V$ is dense in $S^\an$ for the euclidean topology if and only if $(\M,D_\M)$ is factorwise $\V$-admissible. The factorwise $\V$-admissibility condition has a natural generalization to the mixed case, which will be shown in Sect. \ref{sect8} to be necessary for the transverse Hodge locus of prescribed type to be non-empty. This generalization is as follows:
\begin{definition}\label{vlikely}
Let $\V$ be an admissible integral graded-polarized variation of mixed Hodge structures on a smooth and irreducible quasi-projective complex algebraic variety $S$. Let $(\pp, D_\pp)$ be its generic mixed Hodge class and let $\tilde{\Phi} : \widetilde{S^\an} \rightarrow D_\pp$ be the lifted period map at some $\tilde{s} \in \widetilde{S^\an}$. A strict mixed Hodge subclass $(\M,D_\M)$ of $(\pp, D_\pp)$ is said \textit{$\V$-likely} if for every $\NN \in \Np$ the following inequality holds
\begin{equation}\label{likelyn}
\dim q_\NN(D_\M) + \dim (q_\NN \circ \tilde{\Phi})(\widetilde{S^\an}) - \dim D_\pp/\NN \geqslant 0,
\end{equation}
where $\Np$ is the set of normal subgroups of $\pp$ whose radical is unipotent, and for $\NN \in \Np$ we denoted by $q_\NN : D_\pp \rightarrow D_\pp/\NN$ the quotient.
\end{definition}
As demonstrated in a concrete example in the present section, the situation is different and somewhat more complicated in the mixed case than it is in the pure case recalled above. More precisely, we prove:
\begin{proposition}\label{contrexc}
There exists a smooth and irreducible quasi-projective complex curve $C$ and an admissible integral graded-polarized variation of mixed Hodge structures $\V$ on $C$ with generic mixed Hodge class $(\pp,D_\pp)$ such that:
\begin{itemize}
 \item[(i)] There exists a $\V$-likely strict mixed Hodge subclass $(\M,D_M)$ of $(\pp, D_\pp)$ and for any such mixed Hodge subclass:
\begin{itemize}
\item[($i_1)$] the transverse Hodge locus $\mathrm{HL}(C, \V^\otimes, \M)_\trans$ of type $\M$ is dense in $C(\C)$ for the Zariski topology;
\item[$(i_2)$] the transverse Hodge locus $\mathrm{HL}(C, \V^\otimes, \M)_\trans$ of type $\M$ is not dense in $C(\C)$ for the euclidean topology;
\end{itemize}
\item[(ii)] The full transverse Hodge locus $\mathrm{HL}(C, \V^\otimes)_\trans$ is dense in $C(\C)$ for the euclidean topology.
\end{itemize}
\end{proposition}
This shows that in general the transverse Hodge locus of prescribed type might not be dense in the base for the euclidean topology even if it is non-empty. However, this leaves open the eventuality that the full transverse Hodge locus is either empty or dense in the base for the euclidean topology. We also provide an example showing that this weaker statement does not hold either in general:
\begin{proposition}\label{contrexc1}
There exists a smooth and irreducible quasi-projective complex curve $C$ and an admissible integral graded-polarized variation of mixed Hodge structures $\V$ on $C$ such that 
\begin{itemize}
\item[(i)] $\mathrm{HL}(C, \V^\otimes)_\trans$ is dense in $C(\C)$ for the Zariski topology;
\item[(ii)] $\mathrm{HL}(C, \V^\otimes)_\trans$ is not dense in $C(\C)$ for the euclidean topology;
\end{itemize}
\end{proposition}
The examples are constructed on a curve in $(\C^\times)^n$ for some integer $n \geqslant 1$, and to that end, the subsections $6.1-6.3$ are dedicated to recollections on the properties of natural variations of mixed Hodge structures on subvarieties of $(\C^\times)^n$ (for an integer $n \geqslant 1$), and the description of their Hodge locus.

Before starting, we make a notational remark. In what follows, the notation $\C^\times$ will be used both to denote the complex algebraic variety $\mathrm{Spec}(\C[x,y]/\braket{xy-1})$ (without its natural group structure) and the $\C$-points of the latter endowed with the natural Archimedean topology. We prefer reserving the notation $\mult_m$ for the multiplicative group, seen as a $\Q$-algebraic group. We hope this won't create any confusion.
\subsection{The Hodge variety $(\C^\times)^n$}
Fix an integer $n \geqslant 1$. We start by recalling following \cite{carlext} that $(\C^\times)^n$ is canonically isomorphic to $\mathrm{Ext}_{\Z-\mathrm{MHS}}(\Z(0), \Z(1)^n)$ where for $k \in \Z$, we recall that $\Z(k)$ is the Tate Hodge structure, pure of type $(-k,-k)$ with underlying abelian group $(2\pi i)^k \Z$. In the language of the present paper, this realizes $(\C^\times)^n$ as a connected Hodge variety. Let us give some detail on this realization.

Let $V_\Z = \Big( \bigoplus_{1 \leqslant k \leqslant n} (2\pi i) \Z \Big)  \oplus \Z \subset \C^{n+1} = V_\C$, and set $\ee^{(1)} = (1,0,\cdots,0)$, $\ee^{(2)} = (0,1,\cdots,0)$, $\cdots$, $\ee^{(n)} = (0,0, \cdots, 1, 0)$ and $e_0 = (0,0,\cdots, 0,1)$. These form a $\C$-basis of $V_\C$. Furthermore the family $(2\pi i \ee^{(1)}, \cdots  2\pi i \ee^{(n)}, e_0)$ forms a $\Z$-basis of $V_\Z$. Define a finite ascending filtration on $V_\Q$ by
\[
W_{-3} V_\Q = 0 \subsetneq W_{-2}V_\Q = \bigoplus_{1 \leqslant k \leqslant n} (2 \pi i) \Q \cdot \ee^{(k)} = W_{-1}V_\Q \subsetneq W_0V_\Q = V_\Q,
\]
and set $h^{-1,-1} = n$, $h^{0,0} = 1$ and $h^{p,q} = 0$ for $(p,q) \in \Z^2-\{(-1,-1), (0,0)\}$. Let $q_{-2}$ (resp. $q_0$) be the positive definite symeuclidean bilinear form on $\Gr^W_{-2} V_\Q \cong \bigoplus_{1 \leqslant k \leqslant n} (2 \pi i) \Q \cdot \ee^{(k)}$ (resp. on $\Gr^W_0 V_\Q \cong \Q \cdot e_0$) with orthonormal basis $(2\pi i \ee^{(1)}, \cdots ,2 \pi i \ee^{(n)})$ (resp. $e_0$). Let $D_\pp$ be the mixed period domain parametrizing $\Q$-Hodge structures on $V_\Q$ with weight filtration $W_\bul$, graded-polarized by $(q_{-2}, q_0)$ and with Hodge numbers $(h^{p,q})$.

For $\zz = (z_1, \cdots, z_n) \in \C^n$, define
\[
F_{\zz}^{-1} V_\C = V_\C \supsetneq F_{\zz}^0V_\C = \C\cdot \Big(e_0 + \sum_{1 \leqslant k \leqslant n} z_k e^{(k)}_{-2}\Big) \supsetneq F_{\zz}^1V_\C = 0.
\]
\begin{lemma}
The map $\psi : \C^n \rightarrow D_\pp$ which sends $\zz \in \C^n$ to $(V_\Q, W_\bul, F_{\zz}^\bul)$ is a biholomorphism. 
\end{lemma}
\begin{proof}
First note that this map is well defined as for any $\zz \in \C^n$ the mixed Hodge structure $\psi(\zz)$ has by construction the right weight filtration and Hodge numbers. Two different $\zz \in \C^n$ clearly give different $F^0$ so the map is injective. Conversely, any mixed Hodge structure on $V_\Z$ with weight filtration and Hodge numbers as above is fully determined by the zeroth piece of its Hodge filtration which is forced to be of rank $1$, so that it must arise as in the previous construction for some $\zz \in \C^n$. Thus $\psi$ is a bijection between $\C^n$ and $D_\pp$. Finally, the formulae show that the Hodge filtration varies holomorphically with $\zz$ and unicity in Proposition \ref{hodgedat}(b) ensures that $\psi$ is a biholomorphism.
\end{proof}
\begin{lemma}\label{bigradc}
Let $\zz \in \C^n$. The associated mixed Hodge structure $\psi(\zz)$ on $V_\Q$ has only one possible bigrading which is explicitly given by:
\[
V_{\zz}^{-1,-1} = (W_{-2}V_\Q)_\C = \bigoplus_{1 \leqslant k \leqslant n}\C \cdot e^{(k)}_{-2} \hspace{0.1cm} \mathrm{and} \hspace{0.1cm} V_{\zz}^{0,0} = F_{\zz}^{0}V_\C = \C\cdot \Big(e_0 +\sum_{1 \leqslant k \leqslant n} z_k e^{(k)}_{-2}\Big).
\]
\end{lemma}
\begin{proof}
Let $\zz \in \C^n$ and $V_\C = \bigoplus_{p,q \in \Z} V_{\zz}^{p,q}$ be a bigrading of $\psi(\zz)$. Since the only non-zero Hodge numbers are $h^{-1,-1}$ and $h^{0,0}$, for $(p,q)\in \Z^2-\{(-1,-1), (0,0)\}$ one has $V_{\zz}^{p,q} = 0$ which forces:
\[
(W_{-2}V_\Q)_\C = \bigoplus_{p+q \leqslant -2} V_{\zz}^{p,q} = V_{\zz}^{-1,-1},
\]
and
\[
F_{\zz}^0V_\C = \bigoplus_{p\geqslant 0} V_{\zz}^{p,q} = V_{\zz}^{0,0}.
\]
\end{proof}
In the remaining of the section, we set $\pp = (\mult_a)^n \rtimes \mult_m$ which we see as a $\Q$-algebraic subgroup of $\gl(V_\Q)$ through the following inclusion (where we use the basis $(2 \pi i \cdot \ee^{(1)}, \cdots, 2 \pi i \cdot\ee^{(n)}, e_0)$ of $V_\Q$ to write elements of $\gl(V_\Q)$ in matrix form):
\[
(\mult_a)^n \rtimes \mult_m \rightarrow \gl(V_\Q), (\xxi = (\xi_1, \cdots \xi_n), \xi) \mapsto M(\xxi,\xi) := \left(\begin{array}{cccc}\xi  &  \cdots & 0 & \xi_1 \\ \vdots & \ddots & \vdots & \vdots \\0 & 0 & \xi & \xi_n \\0 & 0 & 0 & 1\end{array}\right),
\]
where the upper left block in $M(\xxi, \xi)$ is $\xi \cdot I_n$ with $I_n$ the $n \times n$ identity matrix. We denote by $\uu = (\mult_a)^n$ the unipotent radical of $\pp$.
\begin{lemma}
Let $\zz \in \C^n$ and let $\pp(\zz)$ be the Mumford-Tate group of the $\Q$-mixed Hodge structure $\psi(\zz)$, seen as a subgroup of $\gl(V_\Q)$. Then $\pp(\zz) \subset \pp$.
\end{lemma}
\begin{proof}
Let $\zz \in \C^n$. Writing elements of $\gl(V_\C)$ in matrix form for the basis $(2 \pi i \cdot \ee^{(1)}, \cdots, 2 \pi i \cdot \ee^{(n)}, e_0)$, the bigrading of $\psi(\zz)$ exhibited in Lemma \ref{bigradc} corresponds to the morphism
\[
x_{\zz} : \s_\C \rightarrow \gl(V_\C), (\lambda_1, \lambda_2) \mapsto M\Big(\Big(z_1 \Big(1- \frac{1}{\lambda_1\lambda_2}\Big), \cdots, z_n \Big(1- \frac{1}{\lambda_1\lambda_2}\Big)\Big), \frac{1}{\lambda_1\lambda_2}\Big).
\]
Because, by Lemma \ref{bigradc}, $\psi(\zz)$ admits only one bigrading $x_{\zz}$ must be the Deligne bigrading of $\psi(\zz)$. The explicit description above shows that $x_{\zz}(\s_\C)$ is contained in $\pp_\C$ and Proposition \ref{mtmor} then shows that $\pp(\zz)$ is contained in $\pp$.
\end{proof}
Define a left action of $\pp(\C)$ on $\C^n$ by setting that $g := (\xxi, \xi)$ acts on $\C^n$ by
\[
g \cdot \zz = (\xi z_1 + \xi_1, \cdots, \xi z_n + \xi_n).
\]
One easily checks that this indeed defines a group action by biholomorphisms. On the other side, set $\pp(\C)$ to act on $D_\pp$ by conjugation on the target of the (unique) bigrading associated to a point of $D_\pp$.
\begin{lemma}\label{equivc}
The map $\psi$ is $\pp(\C)$-equivariant for the above defined actions, and $\uu(\C)$ acts transitively on both.
\end{lemma}
\begin{proof}
Let  $\zz \in \C^n$ and $g = (\xxi, \xi) \in \pp(\C)$. One has to check that:
\[
g x_{\zz} g^{-1} = x_{(\xi z_1 + \xi_1, \cdots, \xi z_n + \xi_n)}.
\]
This follows from a direct matrix computation which we omit. The transitivity assertion is obvious from the above formula.
\end{proof}
Let $\Gamma = \mult_a^n(\Z) \subset \pp(\Q)$ which is a neat arithmetic subgroup of $\pp(\Q)$. For the above action $\Gamma$ acts on $\C^n$ by additive translation, and the map $e = (e^{2\pi i \cdot}, \cdots, e^{2\pi i \cdot}) : \C^n \rightarrow (\C^\times)^n$ identifies $(\C^\times)^n$ to the quotient $\Gamma \quo \C^n$. Since $\Gamma$ is neat, it acts on $D_\pp$ properly discontinuously and freely. Therefore, Lemma \ref{equivc} shows that the map $\psi$ descends to a biholomorphism
\[
\bar{\psi} : (\C^\times)^n \rightarrow \Gamma \quo D_\pp.
\]
This is the claimed realization of $(\C^\times)^n$ as a (connected) Hodge variety.
\subsection{Families of extensions of $\Z(0)$ by $\Z(1)^n$}\label{sect62}
Fix an integer $n \geqslant 1$. Let $S$ be a smooth, irreducible and quasi-projective complex algebraic variety, and denote by $(\overline{S}, E)$ a log-smooth projective compactification of $S$. Let $q = (q_1,\cdots,q_n) : S^\an \rightarrow (\C^\times)^n$ be a holomorphic map. We explain in this subsection how to construct an admissible and graded-polarized integral variation of mixed Hodge structures on $S$ from the previous data. We follow closely the exposition of \cite{dellocal}.

Let $\mathcal{V} = \bigoplus_{1 \leqslant k \leqslant n} \mathcal{O}_{S^\an} \cdot e_{-2}^{(k)} \oplus \mathcal{O}_{S^\an}\cdot e_0$ be a free $\mathcal{O}_{S^\an}$-module of rank $n+1$ on $S$, endowed with the integrable connection
\[
\nabla := d -  \left(\begin{array}{cccc}0  &  \cdots & 0 & \frac{d q_1}{q_1}  \\ \vdots & \ddots & \vdots & \vdots \\0 & 0 & 0 & \frac{d q_n}{q_n}  \\0 & 0 & 0 & 0\end{array}\right).
\]
For any base point $s \in S^\an$, denote by \[(m_1(q,\cdot), \cdots,  m_n(q,\cdot)) : \pi_1(S^\an,s) \rightarrow \pi_1((\C^\times)^n, q(s)) \cong \Z^2\] the map induced by $q$ at the level of fundamental groups.
\begin{lemma}\label{monoint}
Let $s \in S^\an$. The monodromy representation $\rho : \pi_1(S^\an,s) \rightarrow \gl(\mathcal{V}_s)$ associated to $(\mathcal{V}, \nabla)$ at the basepoint $s$ is given, in the basis $(e_{-2}^{(1)}, \cdots, e_{-2}^{(n)}, e_0)$, by the formula:
\[
\forall \gamma \in \pi_1(S^\an,s), \rho(\gamma) = \left(\begin{array}{cccc}1  &  \cdots & 0 & (2\pi i)m_1(q,\gamma)  \\ \vdots & \ddots & \vdots & \vdots \\0 & 0 & 1 & (2\pi i)m_n(q, \gamma)  \\0 & 0 & 0 & 1\end{array}\right).
\]
\end{lemma}
\begin{proof}
Let $\gamma : [0,1] \rightarrow S^\an$ be a continuous loop. The horizontal sections of $(\mathcal{V}, \nabla)$ along the loop $\gamma$ are linear combinations of $e_{-2}^{(1)}, \cdots,  e_{-2}^{(n)}$ and
\[
t \in [0,1] \mapsto  e_0 + \sum_{1 \leqslant k \leqslant n} \log(q_k \circ \gamma(t)) e_{-2}^{(k)}
\]
for some local determination of the logarithm. The result follows immediately.
\end{proof}
This endows the local system of solutions to $(\mathcal{V}, \nabla)$ with an integral structure $\V$ given on the fiber over $s \in S^\an$ by
\[ \V_s = \Big(\bigoplus_{1 \leqslant k \leqslant n} (2 \pi i) \Z e_{-2}^{(k)}\Big) \oplus \Z (e_0 + \sum_{1 \leqslant k \leqslant n} \log(q_k(z)) e_{-2}^{(k)}) \subset \mathcal{V}_s\]
for some determination of the logarithm (note that the result is independent of this choice). We have associated to the datum $(C,q)$ an integral local system which we now endow with a variation of mixed Hodge structures.

Consider the following filtrations by $\mathcal{O}_{S^\an}$-submodules
\[
\mathcal{W}_{-3} = 0 \subsetneq \mathcal{W}_{-2} = \bigoplus_{1 \leqslant k \leqslant n} \mathcal{O}_{S^\an} \cdot e_{-2}^{(k)} = \mathcal{W}_{-1} \subsetneq \mathcal{W}_0 = \mathcal{V},
\]
and
\[
\mathcal{F}^{-1} = \mathcal{V} \supsetneq \mathcal{F}^{0} = \mathcal{O}_{S^\an} \cdot e_0 \supsetneq \mathcal{F}^{1} = 0.
\]
The following identities are immediate from the construction
\[
\nabla(\mathcal{W}_k)  \subseteq \mathcal{W}_{k} \otimes \Omega^1_C.
\] 
\[
\nabla(\mathcal{F}^k)  \subseteq \mathcal{F}^{k-1} \otimes \Omega^1_C
\] 
The first one says that the filtration $\mathcal{W}_\bul$ is by flat subbundles, and it is easily seen to give a filtration of $\V_\Q$ by integral sub-local systems $\W_\bul$. We call it the "weight filtration". The second identity is Griffiths transversality for $\mathcal{F}^\bul$. We call the latter filtration the "Hodge filtration". The following is immediate:
\begin{lemma}
The fitrations $(\W_\bul, \mathcal{F}^\bul)$ endow $\V$ with the structure of an integral and graded-polarized variation of mixed Hodge structures on $S$, which is an extension of $\underline{\Z(0)}$ by $\underline{\Z(1)}^n$. The associated period map identifies naturally to $\bar{\psi}^{-1} \circ q$.
\end{lemma}
It remains to study the behavior near the points of $E$. This is the purpose of the following.
\begin{proposition}
Assume that for any holomorphic map $i : \Delta \rightarrow \overline{S}^\an$ such that $i(\Delta^\times) \subset S^\an$, the composite $q \circ i$ is meromorphic at $0$. Then, the associated variation $(\V, \W_\bul, \mathcal{F}^\bul)$ is admissible.
\end{proposition}
\begin{proof}
Let $i : \Delta \rightarrow \overline{S}^\an$ be a holomorphic map such that $i(\Delta^\times) \subset S$. We need to prove that $i^\ast (\V, \W_\bul, \mathcal{F}^\bul)$ is admissible. We make the usual identification $\pi_1(\Delta^\times) \cong \Z$, and denote by $\rho_i$ the monodromy representation associated to $i ^\ast \V$. According to Lemma \ref{monoint}, there exists a $n$-uple of integers $(m_1,\cdots, m_n)$ such that in the basis $(e_{-2}^{(1)}, \cdots, e_{-2}^{(n)}, e_0)$, one has:
\[
T := \rho_i(1) =  \left(\begin{array}{cccc}1  &  \cdots & 0 & (2\pi i) m_1 \\ \vdots & \ddots & \vdots & \vdots \\0 & 0 & 1 & (2\pi i) m_n  \\0 & 0 & 0 & 1\end{array}\right).
\]
Clearly this shows that the monodromy is unipotent, and actually $(T-1)^2 = 0$ so that $N := \log(T) = T-1$. Besides, since the monodromy is trivial on the graded pieces, we are in the setting of the following:
\begin{lemma}
Let $V$ be a $\Q$-vector space endowed with a finite increasing filtration $W_\bul$ and a nilpotent endomorphism $N$ which preserves $W$. Assume that the induced endomorphism $\Gr^W(N) : \Gr^W(V) \rightarrow \Gr^W(V)$ is zero. Then $V$ admits a weight filtration for $N$ relative to $W_\bul$ if and only if for every $k \in \Z$
\[
N(W_k) \subset W_{k-2}.
\]
\end{lemma}
Let $z \in \Delta^\times$. We have $N(e_0) = \sum_{1 \leqslant k \leqslant n} m_k e_{-2}^{(k)} \in (\W_{-2})_z$. Therefore the lemma shows that $\V_z$ admits a weight filtration for $N$ relative to $(\W_\bul)_z$.

It remains to establish the extension property. We first recall the construction of Deligne's canonical extension. Consider the connection on $\mathcal{V}$ defined by
\[
D = d - \left(\begin{array}{cccc}0  &  \cdots & 0 & m_1\frac{dz}{z} - \frac{d q_1}{q_1}  \\ \vdots & \ddots & \vdots & \vdots \\0 & 0 & 0 & m_n\frac{dz}{z} - \frac{d q_n}{q_n}  \\0 & 0 & 0 & 0\end{array}\right).
\]
The pair $(\mathcal{V}, D)$ is easily seen to have trivial associated monodromy representation, and the $D$-horizontal sections of $\mathcal{V}$ are well defined on the whole punctured disk and give another trivialization of $\mathcal{V}$:
\[
\mathcal{V} = \Big(\bigoplus_{1 \leqslant k \leqslant n} \mathcal{O}_{\Delta^\times} \cdot e_{-2}^{(k)} \Big) \oplus \mathcal{O}_{\Delta^\times} \cdot \tilde{e_0}
\]
where $\tilde{e_0} = e_0 + \sum_{1 \leqslant k \leqslant n} \log(\frac{q_k(i(z))}{z^{m_k}})e_{-2}^{(k)}$.
The canonical extension of $\mathcal{V}$ to $\Delta$ is then defined as the extension as a constant bundle for the above trivialization:
\[
\mathcal{V} = \Big(\bigoplus_{1 \leqslant k \leqslant n} \mathcal{O}_{\Delta} \cdot e_{-2}^{(k)} \Big) \oplus \mathcal{O}_{\Delta} \cdot \tilde{e_0}
\]
By assumption, the map $^iq := q \circ (i \vert_{\Delta^\times})$ is meromorphic at $0$ i.e. its components $(^iq_k)_{1 \leqslant k \leqslant n}$ are meromorphic at $0$ as functions $\Delta^\times \rightarrow \C^\times$. For $k \in \{1,\cdots, n\}$, there is by definition a unique writing
\[
^iq_k(z) = z^{l_k} (a_k + f_k(z))
\]
where $l_k \in \Z$, $a_k \in \C^\times$ and $f_k : \Delta \rightarrow \C$ is a holomorphic map vanishing at $0$.
\begin{lemma}\label{egal}
For $k \in \{1,\cdots, n\}$, one has $l_k = m_k$.
\end{lemma}
\begin{proof}
Let $k \in \{1,\cdots, n\}$. Consider a continuous loop $\gamma : [0,1] \rightarrow \Delta^\times$ whose homotopy class is the generator $1$ of $\pi_1(\Delta^\times)$. By definition $m_k$ is the index of $^i\gamma_k := ^iq_k \circ \gamma$ around $0$. If $l_k = 0$, one therefore finds
\[
2 \pi i m_k = \int_{^i\gamma_k} \frac{dz}{z} = \int_\gamma \frac{q'(u)}{q(u)}du = \int_\gamma \frac{f'_k(u)}{a_k + f_k(u)}du
\]
which is zero by the theorem of residues since the integrand of the last integral is holomorphic on the disk. Since we assumed $l_k = 0$, this proves that $m_k = l_k$. If $l_k \neq 0$, the same formula gives
\[
2 \pi i m_k = \int_{^i\gamma_k} \frac{dz}{z} = \int_\gamma \frac{q'(u)}{q(u)}du = \int_\gamma \frac{l_k du}{u} du + \int_\gamma \frac{f'_k(u)}{a_k + f_k(u)}du
\]
The second integral in the right hand side is zero as explained in the $l_k = 0$ case. The first integral in the right hand side is $2 \pi i l_k$ times the the index of $\gamma$ which is $1$ by definition. This finishes the proof.
\end{proof}
It follows from the above Lemma \ref{egal} that for $k \in \{1,\cdots, n\}$, one has for $z \in \Delta^\times$:
\[
\log(\frac{^iq_k(z)}{z^{m_k}}) = \log(a_k + f_k(z))
\]
which has a finite limit as $z \rightarrow 0$. Define the extension with the following formula, which makes sense by what precedes:
\[
\overline{\mathcal{F}^0} = \mathcal{O}_\Delta \cdot \Big( \tilde{e_0} - \sum_{1 \leqslant k \leqslant n} \log(\frac{q_k(i(z))}{z^{m_k}})e_{-2}^{(k)} \Big)
\]
This defines a locally free $\mathcal{O}_\Delta$-submodule of $\overline{\mathcal{V}}$, which is everywhere a supplement of 
\[
\overline{\mathcal{W}_{-2}} = \bigoplus_{1 \leqslant k \leqslant n} \mathcal{O}_{\Delta} \cdot e_{-2}^{(k)}.
\]
This is what we wanted to prove.
\end{proof}
\subsection{Special subvarieties of $(\C^\times)^n$: torus versus mixed period domain}
Fix an integer $n \geqslant 1$. The purpose of this subsection is to prove the equivalence between two notions of "special subvarieties" of $(\C^\times)^n$ and to use this to give a manageable description of the Hodge locus of the variations of mixed Hodge structures of the type described in the previous subsection. Special subvarieties will always be considered to be proper subvarieties (i.e. we remove $(\C^\times)^n$ from the list). 

We will call \textit{$t$-special} a special subvariety of $(\C^\times)^n$ seen as a torus. Recall that $t$-special subvariety of $(\C^\times)^n$ are translates of subtori of $(\C^\times)^n$ by torsion points, that is subvarieties of the form $\zetaa \cdot T(\C)$, where $\zetaa = (e^{2 \pi i \alpha_j})_{1 \leqslant j \leqslant n}$ (with $(\alpha_j) \in \Q^n$) acts by component-wise multiplication on $(\C^\times)^n$, and $T$ is a $\C$-algebraic subtorus of $\mult_{m,\C}^n$. Following \cite[(3.2.6)]{bombgub}, a subgroup $\Lambda \subset \Z^n$ is called \textit{primitive} if $\Lambda = (\Lambda \otimes_\Z \R) \cap \Z^n$. By \cite[Thm. 3.2.19]{bombgub}, there is a bijection between subgroups of $\Z^n$ and algebraic subgroups of $\mult_{m,\C}^n$. By \cite[Cor. 3.2.8]{bombgub}, primitive subgroups of $\Z^n$ correspond under this bijection to algebraic subtori of $\mult_{m,\C}^n$. Explicitely, this bijection associates to $\Lambda \subset \Z^n$ the algebraic subgroup of $\mult_{m,\C}^n$ with $\C$-points
\[
T_\Lambda(\C) = \Big\{\xx = (x_1, \cdots, x_n) \in (\C^\times)^n \hspace{0.1cm} : \hspace{0.1cm} \forall \lambdaa = (\lambda_1, \cdots, \lambda_n) \in \Lambda, \prod_{1 \leqslant k \leqslant n} x_k^{\lambda_k} = 1\Big\}
\]

An \textit{$h$-special subvariety} of $(\C^\times)^n$ is defined to be a subvariety of the form \[\bar{\psi}^{-1}(\period_{(\M, D_\M), \Gamma_\M})\] where $\period_{(\M, D_\M), \Gamma_\M}$ is a mixed Hodge subvariety of $\period_{(\pp, D_\pp),\Gamma}$. Recall that a mixed Hodge subvariety of $\period_{(\pp, D_\pp), \Gamma}$ is a subvariety of the form $\pi_\Gamma(D_\M)$ for some mixed Hodge subclass $(\M, D_M)$ of $(\pp, D_\pp)$, where $\pi_\Gamma : D_\pp \rightarrow \period_{(\pp, D_\pp), \Gamma}$ is the quotient map. Therefore, we are tasked with classifying strict mixed Hodge subclasses of $(\pp, D_\pp)$.
\begin{lemma}\label{structc}
Let $(\M,D_M)$ be a strict mixed Hodge subclass of $(\pp, D_\pp)$. There exists a strict (unipotent) $\Q$-algebraic subgroup $\M^u$ of $\pp$ contained in $\uu = (\mult_a)^n$ and normalized by $\mult_m$, and a $g \in \uu(\Q)$ such that:
\[ 
\M = g(\M^u \rtimes \mult_m)g^{-1}.
\]
\end{lemma}
\begin{proof}
Let $\M^u$ be the unipotent radical of $\M$. By assumption, there exists some $\psi(\zz) \in D_\pp$ whose bigraduation $x_{\zz}$ factors through $\M_\C$, and from the formula for $x_{\zz}$ one sees that $\bar{x}_{\zz} : \s_\C \rightarrow (\M/\M^u)_\C$ has positive dimensional image. Consider a Levi decomposition $\M = \M^u \rtimes \G$. By the previous remark, one has $\dim \G > 0$. As $\G$ is a closed reductive $\Q$-subgroup of $\pp$, \cite[Prop. 5.1]{borser} ensures the existence of a $g \in \uu(\Q)$ such that $g^{-1} \G g \subset \mult_m$. Because $\dim \G > 0$, this forces $g^{-1} \G g = \mult_m$ as subgroups of $\pp$, i.e. $\G = g \mult_m g^{-1}$. In particular, $\M^u$ is normalized by $g \mult_m g^{-1}$. Finally, since the group $\uu$ is commutative and $g \in \uu(\Q)$, we have that $\M^u$ is also normalized by $\mult_m$ and that $\M^u \rtimes g \mult_m g^{-1} = g (\M^u \rtimes \mult_m) g^{-1}$. This proves the statement.
\end{proof}
We can now turn to listing strict mixed Hodge subclasses of $(\pp, D_\pp)$. Let $\Lambda \subset \Z^n$ be a primitive subgroup. We define
\[
D_{\Lambda} = \psi \Big(\Big\{ (z_1, \cdots, z_n) \in \C^n \hspace{0.1cm} : \hspace{0.1cm} \forall \lambdaa = (\lambda_1, \cdots, \lambda_n) \in \Lambda,  \sum_{1 \leqslant k \leqslant n} \lambda_k z_k  = 0 \Big\}\Big) \subset D_\pp.
\]
Consider the $\Q$-subgroup $\uu_{\Lambda}$ of $\uu$ whose $R$-points, for any $\Q$-algebra $R$, are
\[
\uu_{\Lambda}(R) = \Big\{ M(\xxi, 1) :  \xxi = (\xi_1, \cdots, \xi_n) \in R^n, \forall \lambdaa = (\lambda_1, \cdots, \lambda_n) \in \Lambda,  \sum_{1 \leqslant k \leqslant n} \lambda_k \xi_k  = 0 \Big\},
\]
and $\M_{\Lambda} = \uu_{\Lambda} \rtimes \mult_m$ whose $R$-points, for any $\Q$-algebra $R$, are
{\small \[
\M_{\Lambda}(R) = \Big\{ M(\xxi, \xi) :  \xi \in R^\times, \xxi = (\xi_1, \cdots, \xi_n) \in R^n, \forall \lambdaa = (\lambda_1, \cdots, \lambda_n) \in \Lambda,  \sum_{1 \leqslant k \leqslant n} \lambda_k \xi_k  = 0 \Big\}.
\]}
One easily checks that these indeed define connected $\Q$-algebraic subgroups of $\uu$ and $\pp$ respectively. Finally, for $g = (\alpha_1, \cdots, \alpha_n) \in \uu(\Q) = \Q^n$, set
\[
g \cdot D_{\Lambda} = \psi\Big(\Big\{ (z_1 + \alpha_1, \cdots, z_n + \alpha_n) \hspace{0.1cm} : \hspace{0.1cm} (z_1, \cdots, z_n) \in \C^n, \psi(z_1, \cdots, z_n) \in D_{\Lambda} \Big\}\Big) \subset D_\pp.
\]
the translate of $D_{\Lambda}$ by $g$ for the $\pp(\C)$-action on $D_\pp$ defined previously.
\begin{proposition}\label{classc}
The map which associates $(g \M_{\Lambda} g^{-1}, g \cdot D_{\Lambda})$ to a pair $(\Lambda, g)$ is a bijection between
\begin{itemize}
\item the set of pairs $(\Lambda, g)$ where $\Lambda \subset \Z^n$ is a primitive subgroup and $g \in \uu(\Q)$;
\item the set of strict mixed Hodge subclasses of $(\pp, D_\pp)$.
\end{itemize}
\end{proposition}
\begin{proof}
\textit{Step $1$: The pairs $(g \M_{\Lambda} g^{-1}, g \cdot D_{\Lambda})$ are indeed mixed Hodge subclasses of $(\pp, D_\pp)$.}

It suffices to prove this statement for $g = 1$ because the conjugate of a strict mixed Hodge subclass of $(\pp, D_\pp)$ by a rational element $g \in \pp(\Q)$ is still a strict mixed Hodge subclass of $(\pp,D_\pp)$. Let $\Lambda \subset \Z^n$ be a primitive subgroup. We want to prove that $D_{\Lambda}$ is the $\M_{\Lambda}(\R)\uu_{\Lambda}(\C)$-orbit of a point in $D_\pp$ whose associated canonical bigraduation factors through $(\M_{\Lambda})_\C$. Let $\psi(z_1, \cdots, z_n) \in D_{\Lambda}$. The associated canonical bigraduation is
\[
x_{\zz} : \s_\C \rightarrow \gl(V_\C), (\mu_1, \mu_2) \mapsto M\Big(\big(z_1\big(1- \frac{1}{\mu_1 \mu_2}\big), \cdots, z_n\big(1- \frac{1}{\mu_1 \mu_2}\big) \big), \frac{1}{\mu_1 \mu_2}\Big) .
\]
whose image clearly lies in $(\M_{\Lambda})_\C$. Furthermore, given an element $\M_{\Lambda}(\R)\uu_{\Lambda}(\C)$ seen as a matrix $M(\xxi, \xi)$ with $\xxi \in \C^n$, $\xi \in \R^\times$ and $ \sum_{1 \leqslant k \leqslant n} \lambda_k \xi_k  = 0$ for any $\lambdaa = (\lambda_1, \cdots, \lambda_n) \in \Lambda$, one has for $\psi(\zz) \in D_{\Lambda}$:
\[
M(\xxi, \xi) \cdot \psi(\zz) = \psi(\xxi + \xi \zz) = \psi(\xi_1 + \xi z_1, \cdots, \xi_n + \xi z_n).
\]
For any $\lambdaa = (\lambda_1, \cdots, \lambda_n) \in \Lambda$, one has:
\[
\sum_{1 \leqslant k \leqslant n} \lambda_k (\xi_k + \xi z_k) = \sum_{1 \leqslant k \leqslant n} \lambda_k \xi_k + \xi \sum_{1 \leqslant k \leqslant n}\lambda_k z_k = 0
\]
since $\psi(\zz) \in D_{\Lambda}$ and $M(\xxi, \xi) \in \M_{\Lambda}(\R)\uu_{\Lambda}(\C)$. Therefore, $M(\xxi, \xi) \cdot \psi(\zz) \in D_{\Lambda}$, and $\M_{\Lambda}(\R)\uu_{\Lambda}(\C)$ acts on $D_{\Lambda}$ and is easily seen to act transitively from the formula above. This finishes Step $1$.

\noindent \textit{Step $2$: Any strict mixed Hodge subclass of $(\pp, D_\pp)$ is of the form $(g \M_{\Lambda} g^{-1}, g \cdot D_{\Lambda})$.}

Let $(\M, D_M)$ be a strict mixed Hodge subclass of $(\pp, D_\pp)$. By Lemma \ref{structc}, there exists a $g \in \uu(\Q)$ and a strict algebraic $\Q$-subgroup $\M^u$ of $\uu$ normalized by $\mult_m$ such that
\[
\M = g ( \M^u \rtimes \mult_m) g^{-1}.
\]
The group $\M^u$ is a $\Q$-algebraic subgroup of $(\mult_a)^n$ which a vector group, as an abelian $\Q$-algebraic unipotent group. Hence, the $\Q$-algebraic subgroups of $(\mult_a)^n$ are in bijection with $\Q$-vector subspaces of $\Q^n$ (as these are all Lie subalgebras for the trivial bracket). Let $W$ be such a $\Q$-vector subspace, and denote by $r$ its codimension. Then $W$ can be defined in $\Q^n$ by $r$ independent linear equations in $n$ variables with rational coefficients. Multiplying each equation by the denominators of its coefficient, we can assume that the equations have integer coefficients. Denote the $j^{\mathrm{th}}$ equation $(1 \leqslant j \leqslant r)$ by
\[
\sum_{1 \leqslant k \leqslant n} \lambda_k^{(j)} z_k = 0.
\]
Since the equations are independent, the vectors $\lambda^{(j)}  := (\lambda_1^{(j)}, \cdots, \lambda_n^{(j)}) \in \Z^n$ span a rank $r$ subgroup $\Lambda_0$ of $\Z^n$. Let $\Lambda = (\Lambda_0 \otimes_\Z \R) \cap \Z^n$ be the corresponding primitive subgroup of $\Z^n$. Then, by construction
\[
W = \Big\{ (w_1, \cdots, w_n) \in \Q^n \hspace{0.1cm} : \hspace{0.1cm} \forall \lambdaa = (\lambda_1, \cdots, \lambda_n) \in \Lambda, \sum_{1 \leqslant k \leqslant n} \lambda_k w_k = 0\Big\}.
\]
It follows that the subgroup of $\uu$ corresponding to the Lie subalgebra $W$ of $\mathrm{Lie}(\uu)$ is $\uu_\Lambda$. So we have proved that for any mixed Hodge subclass $(\M, D_\M)$ of $(\pp, D_\pp)$, there exists some $g \in \uu(\Q)$ and some primitive subgroup $\Lambda \subset \Z^n$ such that $\M = g \M_\Lambda g^{-1}$. It remains to see that $g \cdot D_\M = g \cdot D_\Lambda$. This is clear from the proof in Step $1$ that $D_\Lambda$ is a $\M_\Lambda(\R)\uu_\Lambda(\C)$-orbit and the following chain of equivalences for $\zz = (z_1, \cdots, z_n) \in \C^n$:
{\small\begin{eqnarray*}
x_{\zz}(\s_\C) \subset (\M_\Lambda)_\C & \Leftrightarrow & \forall (\mu_1, \mu_2) \in \s(\C), \forall \lambdaa =  (\lambda_1,\cdots, \lambda_n) \in \Lambda, \sum_{1 \leqslant k \leqslant n} \lambda_k z_k \big(1 - \frac{1}{\mu_1\mu_2}\big) = 0 \\
& \Leftrightarrow & \forall \lambdaa =  (\lambda_1,\cdots, \lambda_n) \in \Lambda, \sum_{1 \leqslant k \leqslant n} \lambda_k z_k = 0 \\
& \Leftrightarrow & \psi(\zz) \in D_\Lambda.
\end{eqnarray*}}
\end{proof}
\begin{corollary}
$t$-special and $h$-special subvarieties of $(\C^\times)^n$ coïncide.
\end{corollary}
\begin{proof}
Let $\Lambda \subset \Z^n$ be a primitive subgroup and $g = (\alpha_1, \cdots, \alpha_n) \in \uu(\Q) = \Q^n$. Denote by $\zetaa$ the torsion point $(e^{2i\pi \alpha_1}, \cdots e^{2 i \pi \alpha_n})$ of $(\C^\times)^n$. The result will follow from the fact that $\bar{\psi}(\pi_\Gamma (g \cdot D_\Lambda)) = \zetaa \cdot T_\Lambda(\C)$. This follows from an immediate computation, using the fact that under the identifications $D_\pp \cong \C^n$ and $\period_{(\pp, D_\pp), \Gamma} \cong (\C^\times)^n$, the uniformization $\pi_\Gamma$ identifies to $\C^n \rightarrow (\C^\times)^n, (z_1, \cdots, z_n) \mapsto (e^{2i\pi z_1}, \cdots, e^{2i\pi z_2})$.
\end{proof}
\begin{corollary}\label{HL}
Let $S$ be smooth and irreducible quasi-projective complex algebraic variety and $q : S^\an \rightarrow (\C^\times)^n$ be a holomorphic map. Let $\V$ be the associated integral and graded-polarized variation of mixed Hodge structure on $S$. Then:
\begin{itemize}
\item[(i)] $\V$ has generic mixed Hodge class equal to $(\pp, D_\pp)$ if and only if $q(S^\an)$ is not contained in some $\zetaa \cdot T_\Lambda$ for some primitive subgroup $\Lambda \subset \Z^n$ of positive rank and $\zetaa = (e^{2\pi i \alpha_1}, \cdots, e^{2 \pi i \alpha_n})$ with $(\alpha_1, \cdots, \alpha_n) \in \Q^n$;
\item[(ii)] If $\V$ has generic mixed Hodge class equal to $(\pp, D_\pp)$, 
\[
\mathrm{HL}(S, \V^\otimes) = \bigcup_{\substack{(\alpha_1,\cdots, \alpha_n) \in \Q^n \\ \Lambda \subset  \Z^n \hspace{0.1cm} \mathrm{primitive} \\ \mathrm{rk}(\Lambda) > 0}} q^{-1}\Big((e^{2\pi i \alpha_1}, \cdots, e^{2 \pi i \alpha_n}) \cdot T_\Lambda(\C)\Big).
\]
\end{itemize}
\end{corollary}
\begin{proof}
Immediate from definitions and the elucidation of mixed Hodge subvarieties of $\period_{(\pp, D_\pp), \Gamma}$.
\end{proof}
We can readily prove Proposition \ref{contrexc1}.
\begin{proof}[Proof of Proposition \ref{contrexc1}]
Let $C = \mathbb{A}^1_\C- \{0\}$ and let $\V$ be the variation whose period map is the identity $\C^\times \rightarrow \C^\times$. By Corollary \ref{HL} we find that $\mathrm{HL}(C, \V^\otimes)_\trans = \mu_\infty$ where $\mu_\infty \subset \C^\times = C(\C)$ is the set of roots of unity in $\C^\times$, which is dense in $C(\C)$ for the Zariski topology but not for the euclidean topology.
\end{proof}
\subsection{Proof of Proposition \ref{contrexc}}
We reformulate the elucidation made in the previous section of special subvarieties of $(\C^\times)^n$ in the case where $n = 2$. A positive rank subgroup $\Lambda$ of $\Z^2$ is either $\Z^2$ or free of rank $1$. In the former case, the corresponding mixed Hodge subclass of $(\pp, D_\pp)$ is $(\mult_m, \{(0,0)\})$. In the latter case, if we denote by $(b,-a)$ a generator of $\Lambda$ (with $(a,b) \neq (0,0)$), one checks that $\Lambda$ is primitive if and only if $a$ and $b$ are coprime. In that case, the corresponding special subvariety is explicitely given by
\[
Y_{a,b} := T_{\Z \cdot (b,-a)}(\C) = \{ (x_1, x_2) \in (\C^\times)^2 \hspace{0.1cm} : \hspace{0.1cm} x_1^b = x_2 ^a\} = \{(x^a, x^b) \hspace{0.1cm} : \hspace{0.1cm} x \in \C^\times\}.
\]
We denote by $(\M_{a,b}, D_{a,b}) = (\M_{\Z \cdot (b,-a)}, D_{\Z \cdot (b,-a)})$ the corresponding mixed Hodge class, and by $\uu_{a,b} = \uu_{\Z \cdot (b,-a)}$ the unipotent radical of $\M_{a,b}$. With these notations, the classification of Proposition \ref{classc} reads:
\begin{corollary}\label{classn2}
The strict mixed Hodge subclasses of $(\pp, D_\pp)$ are exactly those of the following form:
\begin{itemize}
\item[(i)] $(g \mult_m g^{-1}, \{g\})$ for some $g = (\alpha_1, \alpha_2) \in \Q^2$;
\item[(ii)] $(g\M_{a,b}g^{-1}, g \cdot D_{a,b})$ for some $(a,b) \in \Z^2-\{(0,0)\}$ and $g = (\alpha_1, \alpha_2) \in \Q^2$.
\end{itemize}
\end{corollary}
Consider the algebraic curve $C$ in $(\C^\times)^2$ defined in coordinates $(x_1,x_2)$ by the equation $x_1^2 - 3x_2 = 0$. This is a smooth and irreducible algebraic curve, and we denote by $q : C \rightarrow (\C^\times)^2$ the inclusion. By Sect. \ref{sect62}, this defines an admissible integral graded-polarized variation of mixed Hodge structures $\V$ over $S$ which is an extension of $\underline{\Z(0)}$ by $\underline{\Z(1)}^2$ and has period map $q$.
\begin{lemma}\label{genericc}
The generic mixed Hodge class of $\V$ is $((\mult_a)^2 \rtimes \mult_m, \C^2)$.
\end{lemma}
\begin{proof}
It suffices, by Corollary \ref{HL}, to prove that $C$ not contained in a subvariety of $(\C^\times)^2$ of the form $\zeta \cdot Y_{a,b}$ for some $(a,b) \in \Z^2-\{(0,0)\}$ and some $\zeta = (e^{2i\pi\alpha_1}, e^{2i\pi\alpha_2})$ with $(\alpha_1,\alpha_2) \in \Q^2$. Let $(a,b)\in \Z^2-\{(0,0)\}$ and $\zeta = (e^{2i\pi\alpha_1}, e^{2i\pi\alpha_2})$ as above. If $(x_1,x_2) \in C \cap \zeta \cdot Y_{a,b}$, one has $|x_1|^{2a} = 3^a |x_1|^b$. Since $(a,b) \neq (0,0)$, one cannot have $b = 2a$ because otherwise one would have $3^a = 1$ which implies $a = 0$ hence $b = 0$. Therefore, $|x_1| = e^{\frac{a}{2a-b}\log(3)}$. Since, $C$ contains points with first coordinate of arbitrary large module, this shows that $C \cap \zeta \cdot Y_{a,b} \subsetneq C$, which is the desired result.
\end{proof}
We now give the exhaustive list of $\V$-likely strict mixed Hodge subclasses $(\M,D_\M)$ of $(\pp, D_\pp)$. Start with the following:
\begin{lemma}\label{whodge}
The set $\Np$ is $\{\mult_a^2, \{0\}\} \cup \{\uu_{a,b} \hspace{0.1cm} : \hspace{0.1cm} (a,b) \in \Z^2-\{(0,0)\}\}$. For $(a,b) \in \Z^2 - \{(0,0)\}$, the quotient mixed Hodge class $(\pp, D_\pp)/\uu_{a,b}$ identifies to $(\mult_a \rtimes \mult_m, \C)$, and under this identification, the quotient map $q_{\uu_{a,b}} : D_\pp \rightarrow D_\pp/\uu_{a,b}$ identifies to
\[
q_{\uu_{a,b}} : \C^2 \rightarrow \C, (z_1,z_2) \mapsto bz_1 - az_2,
\]
and $p_{\uu_{a,b}}$ identifies to the map 
\[
p_{\uu_{a,b}} : (\C^\times)^2 \rightarrow (\C^\times), (x_1,x_2) \mapsto x_1^bx_2^{-a}.
\]
\end{lemma}
\begin{proof}
The description of $\Np$ is an easy exercise which we leave to the reader. Let $(a,b) \in \Z^2 - \{(0,0)\}$. Consider the map $\pi_{\uu_{a,b}} : \pp \rightarrow \mult_a \rtimes \mult_m$ which on $R$-points for any $\Q$-algebra $R$ maps $((\xi_1,\xi_2), \xi)$ to $(b\xi_1 - a\xi_2, \xi)$. One easily checks that $\pi_{\uu_{a,b}}$ is a morphism of $\Q$-algebraic groups whose kernel is exactly $\uu_{a,b}$. The claimed description of $(\pp, D_\pp)/\uu_{a,b}$, $q_{\uu_{a,b}}$ and $p_{\uu_{a,b}}$ follows.
\end{proof}
In view of the classification of Corollary \ref{classn2}, the following gives a complete description of $\V$-likely strict mixed Hodge subclasses of $(\pp, D_\pp)$.
\begin{lemma}\label{likelyc}
Let $(\M,D_\M) \subsetneq (\pp, D_\pp)$ be a strict mixed Hodge subclass.
\begin{itemize}
\item[(i)] If $(\M,D_\M) = (g\mult_mg^{-1}, \{g\})$ for some $g = (\alpha_1, \alpha_2) \in \Q^2$, it is not $\V$-likely;
\item[(ii)] If $(\M,D_\M) = (g\M_{a,b}g^{-1}, g \cdot D_{a,b})$ for some $g = (\alpha_1, \alpha_2) \in \mult_a^2(\Q)$ and $(a,b) \in \Z^2-\{(0,0)\}$, one has:
\begin{itemize}
\item[$(ii_1)$] if $(a,b) = (1,2)$ then $(\M,D_\M)$ is not $\V$-likely;
\item[$(ii_2)$] otherwise $(\M,D_\M)$ is $\V$-likely.
\end{itemize}
\end{itemize}
\end{lemma}
\begin{proof}
$(i)$ is immediate. Let $(a,b), (c,d) \in \Z^2-\{(0,0)\}$ and $g = (\alpha_1, \alpha_2) \in \Q^2$. Then:
\[
p_{\uu_{c,d}}(C^\an) = \{\frac{x^{d-2c}}{3^c} \hspace{0.1cm} : \hspace{0.1cm} x \in \C^\times\} = \left\{\begin{array}{cc}\{1\} & \mathrm{if} (c,d) = (1,2) \\\C^\times & \mathrm{otherwise.}\end{array}\right.
\]
On the other hand $p_{\uu_{c,d}}(g \cdot Y_{a,b}) = \C^\times$ if $(a,b) \neq (c,d)$ and equal to $\{e^{2i\pi g}\}$ otherwise. Therefore, if $(a,b) = (1,2)$ the condition (\ref{likelyn}) fails for $\NN = \uu_{1,2}$:
\[
\dim p_{\uu_{1,2}}(C^\an) + \dim p_{\uu_{1,2}}(g \cdot Y_{1,2}) = 0+0 < 1 = \dim D_\pp/\uu_{1,2}.
\]
Therefore $(g\M_{1,2}g^{-1}, g \cdot D_{1,2})$ is not $\V$-likely. If $(a,b) \neq (1,2)$, one has for $(c,d) \neq (1,2)$
\[
\dim p_{\uu_{c,d}}(C^\an) + \dim p_{\uu_{c,d}}(g \cdot Y_{a,b}) = 1+ \dim p_{\uu_{c,d}}(g \cdot Y_{a,b}) \geqslant 1 = \dim D_\pp/\uu_{c,d},
\]
and for $(c,d) = (1,2)$:
\[
\dim p_{\uu_{1,2}}(C^\an) + \dim p_{\uu_{1,2}}(g \cdot Y_{a,b}) = 0+1 \geqslant 1 = \dim D_\pp/\uu_{1,2}.
\]
This proves that for $(a,b) \neq (1,2)$ the mixed Hodge subclass $(g\M_{a,b}g^{-1}, g \cdot D_{a,b})$ of $(\pp, D_\pp)$ is $\V$-likely.
\end{proof}
Finally, Proposition \ref{contrexc} is a consequence of the following.
\begin{proposition}
There exists a $\V$-likely strict mixed Hodge subclass $(\M,D_M) \subsetneq (\pp,D_\pp)$ and for any such mixed Hodge subclass,
\begin{itemize}
\item[(i)] The transverse Hodge locus $\mathrm{HL}(C, \V^\otimes, \M)_\trans$ of type $\M$ is not dense in $C^\an$ for the euclidean topology, but is dense in $C$ for the Zariski topology;
\item[(ii)] The Hodge locus $\mathrm{HL}(C, \V^\otimes)_\trans$ is dense in $C^\an$ for the euclidean topology.
\end{itemize}
\end{proposition}
\begin{proof}
By Lemma \ref{likelyc} there exists $(\M,D_M) \subsetneq (\pp, D_\pp)$ a $\V$-likely strict mixed Hodge subclass and for any such mixed Hodge subclass, there exists $g \in \mult_a^2(\Q)$ and $(a,b) \in \Z^2-\{(0,0), (1,2)\}$ such that $(\M,D_M) = (g \M_{a,b} g^{-1}, g \cdot D_{a,b})$. 

$(i)$First remark that for dimension reasons and by Hodge genericity of $C$ in $(\C^\times)^2$, we have 
\[
\mathrm{HL}(C, \V^\otimes, \M) = \mathrm{HL}(C, \V^\otimes, \M)_\trans,
\]
so we are tasked with studying the distribution of $\mathrm{HL}(C, \V^\otimes, \M)$ in full. Besides, we have the description:
\[
\mathrm{HL}(C, \V^\otimes, \M) = \bigcup_{(\alpha_1, \alpha_2) \in \Q^2} C \cap (e^{2\pi i \alpha_1}, e^{2 \pi i \alpha_2}) \cdot Y_{a,b}.
\]
Proceding as in the proof on Lemma \ref{genericc}, we find that for any $(x_1,x_2) \in \mathrm{HL}(C, \V^\otimes, \M)$, one has $|x_1| = e^{\frac{a}{2a-b}\log(3)}$ which readily proves that $\mathrm{HL}(C, \V^\otimes, \M)$ is not dense in $C$ for the euclidean topology.

Since $a$ and $b$ are coprime and $(a,b) \neq (1,2)$, one has $b \neq 2a$. To prove Zariski density of $\mathrm{HL}(C, \V^\otimes, \M)$ in $C$, it suffices to exhibit an infinite set of points thereof. For integers $q > p \geqslant 2$, set
\[
c_{p,q}(a,b) = (e^{\frac{2 \pi i p}{q}} e^{\frac{a}{2a - b}\log(3)},e^{\frac{4 \pi ip}{q}} e^{\frac{b}{2a - b}\log(3)}).
\]
These are clearly pairwise distinct points, and a direct computation shows that
\[
c_{p,q}(a,b) \in C \cap (e^{\frac{2 \pi ip}{q}}, e^{\frac{4 \pi ip}{q}})\cdot Y_{a,b}.
\]
This proves the desired Zariski-density.

$(ii)$ We are left with proving that the full Hodge locus $\mathrm{HL}(C, \V^\otimes)$ is dense in $C$ for the euclidean topology. Let $(x_1,x_2) = (r_1e^{2\pi i \theta_1}, r_2 e^{2 \pi i \theta_2}) \in C$ with $r_1,r_2 > 0$ and $\theta_1, \theta_2 \in [0, 1[$. The equation of $C$ forces $r_1^2 = 3r_2$ and $\theta_2 = [2 \theta_1]$, where for $\theta \in \R$ we denote by $[\theta]$ the unique integer translate of $\theta$ lying in $[0,1[$. One can chose integers $q > p \geqslant 2$ such that $\theta_1$ is arbitrarily close of $\frac{p}{q}$ which forces $e^{2 \pi i \theta_2}$ to be arbitrarily close of $e^{\frac{4 \pi i p}{q}}$. Similarly, chose $(a,b) \in \Z^2$, with $b \neq 2a$, such that $|r_2 - 3^{\frac{b}{2a-b}}| < \epsilon$ for arbitrary $\epsilon > 0$. Such a pair exists because $r \in \R \mapsto e^{r\log(3)}$ is a continuous surjection onto $\R_{>0}$ and rational numbers of the form $\frac{b}{2a-b}$ are dense in $\R$. Then
\[
|r_1^2 - 3^{\frac{2a}{2a-b}}| = 3|r_2 - 3^{\frac{2a}{2a-b}-1}| = 3|r_2 - 3^{\frac{b}{2a-b}}| < 3 \epsilon
\]
which forces $|r_1 - 3^{\frac{a}{2a-b}}| < \frac{3 \epsilon}{r_2}$. Therefore, one can approximate $(x_1,x_2)$ arbitrarily well using the points $c_{p,q}(a,b)$ for varying parameters. Since the latter are in $\mathrm{HL}(C, \V^\otimes)_\trans$, this proves the desired density for the euclidean topology.
\end{proof}
\section{The likeliness condition}\label{sect7}
In this section, we work in the following setting:
\begin{setting}\label{setting}
$\V$ is an effective integral variation of mixed Hodge structures on a smooth and irreducible quasi-projective complex algebraic variety $S$. Let $\pi_S : \widetilde{S^\an} \rightarrow S^\an$ is a universal covering map, fix a $\tilde{s}_0 \in \widetilde{S^\an}$ and $s_0 = \pi_S(\tilde{s}_0)$. We denote by
\begin{itemize}
\item $(\pp,D_\pp)$ the generic mixed Hodge class of $\V$ and identify $\pp$ to $\pp_{s_0}$ and $D_\pp$ to $D_{\pp_{s_0}}$;
\item $\Gamma \subset \pp(\Q)^+$ a neat arithmetic subgroup containing the image of the monodromy representation associated to $\V$ at $s_0$, and $\pi_\Gamma : D_\pp \rightarrow \period_{(\pp, D_\pp), \Gamma}$ the quotient map;
\item $\Phi : S^\an \rightarrow \period := \period_{(\pp,D_\pp), \Gamma}$ the associated period map, with lift $\tilde{\Phi} : \widetilde{S^\an} \rightarrow D_\pp$ at $\tilde{s}_0$;
\item $\mon \triangleleft \pp$ the algebraic monodromy group of $\V$, and $D_{\mon}$ the weak Mumford-Tate subdomain of $D_\pp$ containing $\tilde{\Phi}(\widetilde{S^\an})$;
\item $\F \subset D_\pp$ a semi-algebraic fundamental set for the action of $\Gamma$ on $D_\pp$, such that the period map $\Phi$ is definable in $\anexp$ for the structure of $\anexp$-definable complex-analytic space on $\period_{(\pp,D_\pp), \Gamma}$ induced by $\F$;
\item $\I$ the smooth locally closed complex-analytic (by effectivity of $\V$) subset $\tilde{\Phi}(\widetilde{S^\an}) \cap \F$ of $D_\pp$;
\item $q_\NN : D_\pp\rightarrow D_\pp/\NN$ and by $p_\NN : \period \rightarrow \period/\NN$ the quotient maps associated to a $\NN \in \Np$, and by $\Phi_{/\NN} : S^\an \rightarrow \period/\NN$ the composite $p_\NN \circ \Phi$.
\end{itemize}
\end{setting}
\subsection{Characterization of likeliness}
In this subsection, we prove an easier to handle and equivalent formulation of the likeliness condition, in the spirit of \cite{ku}. Start with the:
\begin{definition}
Let $(\M,D_\M) \subsetneq (\pp,D_\pp)$ be a strict mixed Hodge subclass and $\NN \in \Np$ be a strict normal subgroup whose radical is unipotent. A pair $(g,t) \in \pp(\R)^+\pp^u(\C) \times D_\pp/\NN$ is called a \textit{$(\I,D_\M,D_\pp/\NN)$-intersecting pair} if
\[
\I \cap (g \cdot D_\M) \cap q_\NN^{-1}(t) \neq \emptyset
\]
\end{definition}
Then we have:
\begin{proposition}\label{indep}
Let $(\M,D_\M) \subsetneq (\pp,D_\pp)$ be a strict mixed Hodge subclass. There exists a non-empty Zariski open subset $S_\gen$ of $S$ such that letting $\I_\gen = \F \cap \tilde{\Phi}(\pi_S^{-1}(S_\gen^\an))$ the following holds: for every $\NN\in \Np$ and any $(\I_\gen,D_\M,D_\pp/\NN)$-intersecting pair $(g,t)  \in \pp(\R)^+\pp^u(\C) \times D_\pp/\NN$, the intersections
\begin{itemize}
\item[(a)] $(g\cdot D_\M) \cap q_\NN^{-1}(t)$
\item[(b)] $\I_\gen \cap q_\NN^{-1}(t)$
\end{itemize}
are equidimensional and the dimension of each of them is independent of the choice of the $(\I_\gen,D_\M,D_\pp/\NN)$-intersecting pair $(g,t)$. In particular, the quantity
\[
d_\NN(\M,D_\M) := \dim\big((g\cdot D_\M) \cap q_\NN^{-1}(t)\big) + \dim\big(\I_\gen \cap q_\NN^{-1}(t)\big) - \dim q_\NN^{-1}(t)
\]
is well defined and only depends on $\NN$.
\end{proposition}
\begin{proof}
$(a)$ Let $(g,t)$ and $(g',t')$ be two possibly equal $(\I,D_\M,D_\pp/\NN)$-intersecting pairs. Let $x \in (g\cdot D_\M) \cap q_\NN^{-1}(t)$ and $x' \in (g'\cdot D_\M) \cap q_\NN^{-1}(t')$. We want to prove that
\[
\dim_x (g\cdot D_\M) \cap q_\NN^{-1}(t) = \dim_{x'} (g'\cdot D_\M) \cap q_\NN^{-1}(t').
\]
Since $g' \M(\R)^+\M^u(\C)g'^{-1}$ acts transitively on $g' \cdot D_\M$, and since $g'g^{-1}x \in g' \cdot D_\M$, there exists a $m \in g' \M(\R)^+\M^u(\C)g'^{-1}$ such that $x' = mg'g^{-1}x$. Then,
\[
x' \in mg'g^{-1}\cdot \Big((g \cdot D_\M) \cap q_\NN^{-1}(t)\Big) = (g' \cdot D_\M) \cap \Big(mg'g^{-1}\cdot (q_\NN^{-1}(t))\Big).
\]
We claim that $mg'g^{-1}\cdot (q_\NN^{-1}(t)) = q_\NN^{-1}(t')$. Indeed, the map $q_\NN$ is by construction equivariant under the quotient map $ \pi_\NN : \pp(\R)^+\pp^u(\C) \rightarrow  (\pp/\NN)(\R)^+(\pp^u/\NN)(\C)$, so that
\[mg'g^{-1}\cdot (q_\NN^{-1}(t)) = q_\NN^{-1}(\pi_\NN(mg'g^{-1}) \cdot t).\]
But, we saw that $x' \in mg'g^{-1}\cdot (q_\NN^{-1}(t))$ and by definition $x' \in q_\NN^{-1}(t')$ so $\pi_\NN(mg'g^{-1}) \cdot t = t'$. This proves the claim.

Conjugation by $mg'g^{-1}$ gives a biholomorphic automorphism of $D_\pp$ and maps biholomorphically any complex-analytic irreducible component of $(g\cdot D_\M) \cap q_\NN^{-1}(t)$ containing $x$ to some complex-analytic irreducible component of $(g'\cdot D_\M) \cap q_\NN^{-1}(t')$ containing $x'$. Choosing $(g,t) = (g',t')$ and $x$ (resp. $x'$) such that only one complex-analytic irreducible component of $(g\cdot D_\M) \cap q_\NN^{-1}(t)$ (resp. $(g\cdot D_\M) \cap q_\NN^{-1}(t)$) contains $x$ (resp. $x'$) shows the equidimensionality of $(g\cdot D_\M) \cap q_\NN^{-1}(t)$. Choosing $(g,t) \neq (g',t')$ and $x,x'$ arbitrarily shows that the (well defined by equidimensionality) complex-analytic dimensions of $(g\cdot D_\M) \cap q_\NN^{-1}(t)$ and $(g'\cdot D_\M) \cap q_\NN^{-1}(t')$ are equal. This proves $(a)$, without having to replace $S$ by an open subset.

$(b)$ Fix some $\NN \in \Np$. We start by proving that there is a non-empty Zariski open subset $S_{\NN,\gen}$ of $S$ such that $(b)$ holds for $\NN$ after replacing $S$ by $S_{\NN,\gen}$. Let $T_\NN = \{t \in D_\pp/\NN \hspace{0.1cm} : \hspace{0.1cm} \I \cap q_\NN^{-1}(t) \neq \emptyset \}$ the set of second-coordinates of $(\I,D_\M, D_\pp/\NN)$-intersecting pairs. Let $C_\NN = \min \{\dim_x \I \cap q_\NN^{-1}(t) \hspace{0.1cm} : \hspace{0.1cm} t \in T_\NN, x \in \I \cap q_\NN^{-1}(t)\}$ which is a well defined non-negative integer as the minimum of a set of non-negative integers. Define
\[
\I_{\NN,\ex} = \{ x \in \I \hspace{0.1cm} : \hspace{0.1cm} \dim_x \I \cap q_\NN^{-1}(q_\NN(x)) > C_\NN\},
\]
and $S_{\NN,\ex} = \pi_S(\tilde{\Phi}^{-1}(\I_{\NN,\ex}))$. We first claim that it suffices to prove that $S_{\NN,\ex}$ is a strict Zariski-closed subset of $S$ to prove $(b)$. Indeed, assume that $S_{\NN,\ex}$ is a strict Zariski-closed subset of $S$. Replace $S$ by its non-empty Zariski open subset $S_{\NN,\gen} := S-S_{\NN,\ex}$. After this operation, the subspace $\I \subset D_\pp$ is replaced by $\I_{\NN,\gen} := \I-\I_{\NN,\ex} \subset D_\pp$. Let $(g,t)$ and $(g',t')$ be two possibly equal $(\I_{\NN,\gen},D_\M,D_\pp/\NN)$-intersecting pairs, $x \in \I_{\NN,\gen} \cap q_\NN^{-1}(t)$ and $x' \in \I_{\NN,\gen} \cap q_\NN^{-1}(t')$. By definition of $\I_{\NN,\gen}$, one has that $\dim_x \I_{\NN,\gen} \cap q_\NN^{-1}(t) = \dim_{x'} \I_{\NN,\gen} \cap q_\NN^{-1}(t') = C_\NN$. One then concludes as in the proof of $(a)$ that   $\I_{\NN,\gen} \cap q_\NN^{-1}(t)$ is equidimensional with dimension independent of the choice of $(\I_{\NN,\gen}, D_\M, D_\pp/\NN)$-intersecting pair.

We are therefore left with proving that $S_{\NN,\ex}$ is a strict Zariski-closed subset of $S$. By definition, $S_{\NN,\ex}$ is a strict subset of $S$. Recall that $\Phi$ factors as $i \circ f$ where $f : S \rightarrow Y$ is the analytification of a dominant regular morphism to a quasi-projective complex variety and $i : Y^\an \hookrightarrow \period$ is the analytification of a closed definable immersion. This is summarized in the following commutative diagram:
\[
\xymatrix{\widetilde{S^\an} \ar[d]_{\pi_S} \ar[rr]^{\tilde{\Phi}} & & D_\pp \ar[d]^{\pi_\Gamma} \\ S^\an \ar[r]_{f}& Y^\an \ar[r]_{i} & \period}
\]
where $\pi_\Gamma : D_\pp \rightarrow \period$ is the quotient map. Therefore, if we let $Y_{\NN, \ex} := i^{-1}(\pi_\Gamma(\I_{\NN,\ex}))$, we have $S_{\NN,\ex} = f^{-1}(Y_{\NN, \ex})$. Hence we are reduced to proving that $Y_{\NN, \ex}$ is Zariski-closed in $Y$. By the definable Chow Theorem \ref{petstar}, it suffices to prove that it is a definable closed complex-analytic subset. It is definable because $\I_{\NN,\ex} \subset \F$ is a definable subset and $i$ and $\pi_\Gamma\vert_\F$ are definable maps. Denote by $i_\NN := p_\NN \circ i$ which is a complex analytic map. We claim that
\[
Y_{\NN, \ex} = \{ y \in Y^\an \hspace{0.1cm} : \hspace{0.1cm} \dim_y i_\NN^{-1}(i_\NN(y)) > C_\NN\}.
\]
Because $i_\NN$ is a complex analytic morphism, this would prove that $Y_{\NN, \ex}$ is a closed complex-analytic subset of $Y$. As the projection $\pi_\Gamma$ is a local isomorphism on the source and $i$ is a closed immersion, for $z \in \I$ one has {\small\[ \dim_z \I \cap q_\NN^{-1}(q_\NN(z)) = \dim_{\pi_\Gamma(z)} i(Y^\an) \cap p_\NN^{-1}(p_\NN(\pi_\Gamma(z))) = \dim_{i^{-1}(\pi_\Gamma(z))} i_\NN^{-1}(i_\NN(i^{-1}(\pi_\Gamma(z)))).\]} The above claim follows from the latter fact, the definition of $\I_{\NN,\ex}$ and the fact that $Y_{\NN, \ex} = i^{-1}(\pi_\Gamma(\I_{\NN,\ex}))$.

The Zariski open subset of $S$ whose existence is claimed in the statement shall not depend on $\NN$, so we have to get rid of the dependence in $\NN$ of $S_{\NN, \gen}$. This is a consequence of Baldi-Urbanik's geometric Zilber-Pink theorem. Let $\Np_\ex$ be the set of $\NN \in \Np$ such that there exists a $(\I, D_\M, D_\pp/\NN)$-intersecting pair $(g,t)$ such that $\I \cap q_\NN^{-1}(t)$ is either not equidimensional or equidimensional of dimension bigger than the generic dimension $C_\NN$ defined above. 
\begin{lemma}\label{finitezp}
The set $\bigcup_{\NN \in \Np_\ex} S_{\NN, \ex}$ is a finite union of strict Zariski-closed subsets of $S$.
\end{lemma} 
Assuming this statement for a moment, the intersection $S_\gen := \bigcap_{\NN \in \Np_\ex} S_{\NN, \gen}$ is then a non-empty Zariski open subset of $S$ as it is a finite intersection of non-empty Zariski open subsets of the irreducible complex algebraic variety $S$. Furthermore, by construction $S_\gen$ satisfies the properties required in the statement.
\end{proof}
\begin{proof}[Proof of Lemma \ref{finitezp}]
Using notations of Theorem \ref{geozp}, let $\Np_{\mathrm{ZP}}$ be the set of $\NN \in \Np$ such that $((\pp, D_\pp), \NN) \in \Sigma_{(S, \V)}$. By \textit{op. cit.} we know that $\Np_{\mathrm{ZP}}$ is finite, hence it suffices to argue that
\[
\bigcup_{\NN \in \Np_\ex} S_{\NN, \ex} \subset \bigcup_{\NN \in \Np_{\mathrm{ZP}}} S_{\NN, \ex}
\]
Let $\NN \in \Np_\ex$ and $s \in S_{\NN, \ex}$. By definition, there exists a $(\I, D_\M, D_\pp/\NN)$-intersecting pair $(g,t)$ and a complex-analytic irreducible component $U$ of $\I \cap q_\NN^{-1}(t)$ such that $\dim U > C_\NN$ and $s \in \Phi^{-1}(\pi_\Gamma(U))$. Since $\I$ and $q_\NN^{-1}(t)$ are smooth locally closed complex analytic subsets of $D_\mon$, a combination of Lemma \ref{intersection} and the definition of $C_\NN$ shows that $C_\NN \geqslant \dim \I + \dim q_\NN^{-1}(t) - \dim D_\mon$. Therefore, one finds that any complex-analytic irreducible component of $\Phi^{-1}(\pi_\Gamma(U))$ is a monodromically atypical weakly special subvariety of $S$ for $\V$ in the sense of Definition \ref{monodratyp}. Summing up, we have shown that $s$ is contained in a monodromically atypical weakly special subvariety $Z$ of $S$ for $\V$. Let $Z'$ be a maximal one containing $Z$. By definition of $\Np_\mathrm{ZP}$ one finds that $\mon_{Z'} \in \Np_\mathrm{ZP}$ and since $Z'$ is monodromically atypical, arguing as above shows that $Z' \subset S_{\mon_{Z'}, \ex}$. This shows that $s \in \bigcup_{\NN \in \Np_{\mathrm{ZP}}} S_{\NN, \ex}$, as desired.
\end{proof}
\begin{corollary}\label{reflik}
Let $(\M,D_\M) \subsetneq (\pp,D_\pp)$ a mixed Hodge subclass. Then $(\M,D_\M)$ is $\V$-likely if and only if for every $\NN \in \Np$, one has:
\[
d_\NN(\M,D_\M) \leqslant d_{\pp^\der}(\M,D_\M).
\]
\end{corollary}
\begin{proof}
Let $\NN \in \Np$. Let $S_{\gen}$ and $\I_{\gen}$ be as in the statement of Proposition \ref{indep}. Let $(g,t)$ be an $(\I_{\gen},D_\M,D_\pp/\NN)$-intersecting pair. Then, because $\I_{\gen} \cap q_\NN^{-1}(t)$ is equidimensional and coïncides with the fiber $(q_\NN \vert_{\I_{\gen}})^{-1}(t)$, we find:
\[
\dim \I_{\gen} = \dim \Big(\I_{\gen} \cap q_\NN^{-1}(t)\Big) + \dim q_\NN(\I_{\gen}).
\]
Similarly, $\dim g \cdot D_\M = \dim \Big((g \cdot D_\M) \cap q_\NN^{-1}(t)\Big) + \dim q_\NN(g \cdot D_\M)$. Therefore, one has:
{\small\begin{eqnarray*}
d_{\pp^\der}(\M,X_M,D_\M) - d_\NN(\M,X_M,D_\M) & = & \Big[ \dim (g \cdot D_\M) - \dim \Big((g \cdot D_\M) \cap q_\NN^{-1}(t)\Big) \Big] \\ & & +\Big[ \dim \I_{\gen} - \dim \Big(\I_{\gen} \cap q_\NN^{-1}(t)\Big)\Big] \\ & & - \Big[ \dim q_\NN^{-1}(t) - \dim D_\pp\Big] \\
& = & \dim q_\NN(g  \cdot D_\M) + \dim q_\NN(\I_{\gen}) - \dim D_\pp/\NN \\
& = & \dim p_\NN(\period_\M) + \dim \Phi_\NN(S^\an) - \dim \period/\NN.
\end{eqnarray*}}
This immediately implies the claimed equivalence.
\end{proof}
\subsection{Existence of transverse intersections}
In this subsection we establish two intermediate results which will be of constant use in the sequel, and at the same time clarify the link between $\V$-likeliness and transverse intersections. The first is:
\begin{proposition}\label{nec}
Let $(\M, D_\M)\subsetneq (\pp,D_\pp)$ a strict mixed Hodge subclass and assume that there exists an $(\M,D_\M)$-transverse special subvariety $Z \subset S$, such that $Z \cap S_\gen \neq \emptyset$, where $S_\gen$ is the non-empty Zariski open subset of $S$ defined in Proposition \ref{indep}. Then $(\M,D_\M)$ is $\V$-likely.
\end{proposition}
\begin{proof}
We use the notations of Proposition \ref{indep}. First remark that $(\M,D_\M)$ is $\V$-likely if and only if it is $\V \vert_{S_\gen}$-likely. As $Z \cap S_\gen$ is non-empty, it is an $(\M,D_\M)$-transverse special subvariety of $S_\gen$. Hence we will assume $S = S_\gen$. Let $\Zcal \subset \I \cap D_\M$ be a complex-analytic irreducible component such that $\pi_\Gamma(\Zcal) = \Phi(Z^\an)$. Since $Z$ is transverse, one has
\[
\dim \Zcal = \dim \I + \dim D_\M - \dim D_\pp = d_{\pp^\der}(\M, D_\M).
\]
Let $\NN \in \Np$. Since fibers of $q_\NN$ cover $D_\pp$, there exists $t \in D/\NN$ such that $q_\NN^{-1}(t) \cap \Zcal \neq \emptyset$. In particular, $(1, t)$ is an $(\I, D_\M, D_\pp/\NN)$-intersecting pair. Since we assumed $S = S_\gen$, one has:
\[
\dim q_\NN^{-1}(t) \cap \I = \dim q_\NN^{-1}(t) + \dim \I - \dim D
\]
and
\[
\dim q_\NN^{-1}(t) \cap D_\M = \dim q_\NN^{-1}(t) + \dim D_\M - \dim D.
\]
Since $q_\NN^{-1}(t)$, $D_\M$ and $\I$ are smooth equidimensional locally closed complex-analytic subsets of $D$, the above equalities ensure that $q_\NN^{-1}(t) \cap \I$ and $q_\NN^{-1}(t) \cap D_\M$ are (equidimensional) locally closed complex-analytic subsets of $q_\NN^{-1}(t)$. Now remark that each irreducible component $\Zcal \cap q_\NN^{-1}(t)$ is a complex-analytic irreducible component of the intersection $(q_\NN^{-1}(t) \cap \I) \cap (q_\NN^{-1}(t) \cap D_\M)$ in $q_\NN^{-1}(t)$. Hence, by Lemma \ref{intersection}(ii),
\[
\dim \Zcal \cap q_\NN^{-1}(t) \geqslant \dim q_\NN^{-1}(t) \cap \I + \dim q_\NN^{-1}(t) \cap D_\M - \dim q_\NN^{-1}(t) = d_\NN(\M, D_\M)
\]
where the last equality follows from the definition, using that $(1,t)$ is an $(\I, D_\M, D_\pp/\NN)$-intersecting pair. Finally, as $\dim \Zcal \geqslant \dim \Zcal \cap q_\NN^{-1}(t)$, we obtain
\[
d_{\pp^\der}(\M,D_\M) = \dim \Zcal \geqslant \dim \Zcal \cap q_\NN^{-1}(t) = d_\NN(\M,D_\M).
\]
By Corollary \ref{reflik}, this proves that $(\M,D_\M)$ is $\V$-likely.
\end{proof}
The second result is:
\begin{proposition}\label{suf}
Let $(\M, D_\M)\subsetneq (\pp,D_\pp)$ a strict $\V$-likely mixed Hodge subclass. Let $s \in S_\gen^\an$ be a Hodge generic point, let $x = \Phi(s)$ and let $z \in \I$ such that $\pi_\Gamma(z) = x$. Let $g \in \pp(\R)^+\pp^u(\C)$ such that $z \in (g \cdot D_\M) \cap \I$ and let $U$ be a complex-analytic irreducible component of $(g \cdot D_\M) \cap \I$ containing $z$. Then,
\[
\dim U = \dim D_\M + \dim \I - \dim D_\pp.
\]
\end{proposition}
\begin{proof}
Arguing as at the beginning of the proof of Proposition \ref{nec} using that $s \in S_\gen^\an$, we can assume that $S = S_\gen$. We proceed in three steps.

\noindent \textit{Step 1 : $\dim U \geqslant \dim D_\M + \dim \I - \dim D_\pp$.}

Since $\V$ is assumed to be effective one has that $\I$ is a smooth locally closed complex-analytic subset of $D_\pp$. Besides $D_\M$ is a smooth closed complex-analytic subset of $D_\pp$. By Lemma \ref{intersection}(ii), we have the claimed inequality of dimensions. For the remaining part of the proof, we assume for the sake of contradiction that this inequality is strict, that is:
\[
\dim U > \dim D_\M + \dim \I - \dim D_\pp.
\]

\noindent \textit{Step 2 : There exists a $\NN \in \Np$ such that $U \subset q_\NN^{-1}(t)$.}

Firstly, remark that since $\tilde{\Phi}(\widetilde{S^\an}) \subset D_{\mon}$, one has $\I \subset D_{\mon}$ hence $U \subset D_{\mon}$. In particular $\dim \I = \dim \I \cap D_{\mon}$ and $\dim U = \dim U \cap D_{\mon}$. By the reformulation of $\V$-likeliness in Corollary \ref{reflik} with $\NN = \mon$, one finds that
\[
\dim (D_\M \cap D_{\mon}) + \dim \I - \dim D_{\mon} \leqslant \dim D_\M + \dim \I - \dim D_\pp,
\]
which implies
\[
\dim U > \dim (D_\M \cap D_{\mon}) + \dim \I - \dim D_{\mon}.
\]
Let $Y_0$ and $\Phi_0$ be as in Lemma \ref{auxper}. Since $\Phi_0$ is the period map of an admissible graded-polarized integral variation of mixed Hodge structures on $Y_0$ by Lemma \ref{auxper}, the Ax-Schanuel Theorem \ref{axschan} applies to $\Phi_0$. Set $W = Y_0 \times (g \cdot D_M) \cap D_{\mon}$ which is an algebraic subset of $Y_0 \times D_{\mon}$, let $\Delta = Y_0 \times_{\period_{(\pp, D_\pp), \Gamma}} D_{\mon}$. Denote by $p_{Y_0}$ (resp. $p_{D_{\mon}}$) the projection of $Y_0 \times D_{\mon}$ to $Y_0$ (resp. $D_{\mon}$). Set theoretically,
\[
p_{D_{\mon}}(W \cap \Delta) = \tilde{\Phi}(\widetilde{Y_0^\an}) \cap (g \cdot D_\M),
\]
which contains $U$. Let $Z \subset \Delta \cap W$ a complex-analytic irreducible component such that $U \subset p_{D_{\mon}}(Z)$. In particular $\dim Z \geqslant \dim U$. We claim that $\dim Z > \dim W + \dim \Delta - \dim Y_0 \times D_{\mon}$. Indeed,
\[
\dim W + \dim \Delta - \dim Y_0 \times D_{\mon} = \dim (D_\M \cap D_{\mon}) + \dim \Delta - \dim D_{\mon}.
\]
Hence, one has to compute $\dim \Delta$. As $\pi_\Gamma$ has discrete fibers, the definition of $\Delta$ implies that $p_{Y_0}\vert_{\Delta}$ has discrete fibers. It follows that $\dim \Delta = \dim Y_0 = \dim \I$. Therefore,
\[
\dim Z \geqslant \dim U > \dim (D_\M \cap D_{\mon}) + \dim \I - \dim D_{\mon} = \dim W + \dim \Delta - \dim Y_0 \times D_{\mon}.
\]
By the Ax-Schanuel Theorem \ref{axschan}, there exists a strict weakly special subvariety $X_0$ of $Y_0$ for $\V_0$ containing $p_{Y_0}(Z)$. Unwrapping definitions, this implies that there exists a $\NN \in \Np$ and a $t \in D/\NN$ such that $\Phi_0(p_{Y_0}(Z)) \subset \pi_\Gamma(q_\NN^{-1}(t))$. This implies that $\pi_\Gamma^{-1}(\Phi_0(p_{Y_0}(Z))) \subset \pi_\Gamma^{-1}(\pi_\Gamma(q_\NN^{-1}(t)))$. Since $Z \subset \Delta$, one has that $\pi_\Gamma^{-1}(\Phi_0(p_{Y_0}(Z))) = p_{D_{\mon}}(Z)$, hence the above containement rewrites $p_{D_{\mon}}(Z) \subset \pi_\Gamma^{-1}(\pi_\Gamma(q_\NN^{-1}(t)))$. On the other hand, since $\NN$ is normal in $\pp$, for any $\gamma \in \Gamma$, the translate $\gamma \cdot q_\NN^{-1}(t)$ is the $\NN(\R)^+\NN^u(\C)$-orbit of $\gamma \cdot h$ for some $h \in q_\NN^{-1}(t)$. Letting $t_\gamma = q_\NN(\gamma \cdot h)$, this rewrites $\gamma \cdot q_\NN^{-1}(t) = q_\NN^{-1}(t_\gamma)$. In particular, for any $\gamma \in \Gamma$ the intersection $\gamma \cdot q_\NN^{-1}(t) \cap q_\NN^{-1}(t)$ is either empty or equal to $q_\NN^{-1}(t)$. Therefore the $\gamma \cdot q_\NN^{-1}(t)$ are the connected components of
\[
\pi_\Gamma^{-1}(\pi_\Gamma(q_\NN^{-1}(t))) = \bigcup_{\gamma \in \Gamma} \gamma \cdot q_\NN^{-1}(t).
\]
The fact that $p_{D_{\mon}}(Z)$ is connected (because $Z$ is) then implies that there exists some $\gamma \in \Gamma$ such that $p_{D_{\mon}}(Z) \subset \gamma \cdot q_\NN^{-1}(t) = q_\NN^{-1}(t_\gamma)$. By definition of $Z$, one has $U \subset p_{D_{\mon}}(Z)$, hence $U \subset q_\NN^{-1}(t_\gamma)$. This finishes Step $2$.

\noindent \textit{Step 3 : Contradiction}

Choose $\NN \in \Np$ of minimal dimension among those satisfying the property that there exists some (necessarily unique as the fibers of $q_\NN$ are disjoint) $t \in D_\pp/\NN$ such that $U \subset q_\NN^{-1}(t)$. Let $X_0 \subset Y_0$ be the strict weakly special subvariety $\Phi_0^{-1}(\pi_\Gamma(q_\NN^{-1}(t)))$ of $Y_0$ for $\V_0$. We claim that
\[
\dim U = \dim q_\NN^{-1}(t) \cap \I + \dim q_\NN^{-1}(t) \cap D_\M - \dim q_\NN^{-1}(t).
\]
Indeed, arguing as in the proof of Proposition \ref{nec}, using that $S = S_\gen$, one finds that $U = U \cap q_\NN^{-1}(t)$ is an irreducible complex-analytic component of the intersection between the smooth locally closed complex analytic subsets $q_\NN^{-1}(t) \cap \I$ and $q_\NN^{-1}(t) \cap D_\M$ of $q_\NN^{-1}(t)$ in $q_\NN^{-1}(t)$, so that
\[
\dim U \geqslant \dim q_\NN^{-1}(t) \cap \I + \dim q_\NN^{-1}(t) \cap D_\M - \dim q_\NN^{-1}(t).
\]
If this inequality was to be strict, the use of the Ax-Schanuel Theorem \ref{axschan} exactly as in Step $2$ (replacing $Y_0$ by $X_0$ and $\Phi_0$ by $\Phi_0\vert_{X_0}$) would produce some $\NN' \in \Np$ with smaller dimension than $\NN$ such that $U \subset q_{\NN'}^{-1}(t')$ for some $t' \in D_\pp/\NN'$. This would contradict the minimality of $\NN$. This proves the claimed equality. We conclude the proof by remarking that the chain of inequalities
\begin{eqnarray*}
0 & = & \dim U - \dim U \\
& > & \big[\dim \I + \dim D_\M - \dim D_\pp \big] - \big[\dim \I \cap q_\NN^{-1}(t) + \dim D_\M \cap q_\NN^{-1} - \dim q_\NN^{-1}(t)\big] \\
& = & d_{\pp^\der}(\M, D_\M) - d_\NN(\M, D_\M)
\end{eqnarray*}
contradicts the $\V$-likely property for $(\M,D_\M)$ formulated as in Proposition \ref{reflik}.
\end{proof}
\subsection{Consequences of likeliness and horizontality}
In this subsection, we work in the following:
\begin{setting}\label{settingext}
Fix $\V$ an effective variation of mixed Hodge structures on a smooth and irreducible quasi-projective complex variety $S$, with notations as in Setting \ref{setting}. We fix a mixed Hodge subclass $(\M,D_\M)$ of $(\pp, D_\pp)$ which is $\V$-likely and a point $z \in D_\M$ along with some associated bigraduation $x : \s_\C \rightarrow \M_\C$. We denote by $\mt$ (resp. $\para$) the Lie algebra of $\M$ (resp. $\pp$) and
\[
\mt_\C = \bigoplus_{p+q \leqslant 0} \mt^{p,q} \subset \para_\C = \bigoplus_{p+q \leqslant 0} \para^{p,q}
\]
the bigraduations associated with $x$ and the adjoint representations of $\pp$ and $\M$. Under the above inclusion, one has $\mt^{p,q} \subset \para^{p,q}$ for every $p+q \leqslant 0$. We denote by $W_\bul \para$ and $W_\bul \mt$ the weight filtrations on $\para$ and $\mt$. Finally, we denote by $h_{\mt}^{p,q}$ (resp. $h_\para^{p,q}$) the complex dimension of $\mt^{p,q}$ (resp. $\para^{p,q}$).
\end{setting}
In this setting, the mixed Hodge class $(\M,D_\M)$ is heavily constrained:
\begin{proposition}\label{likconstr}
The following assertions hold:
\begin{itemize}
\item[(i)] the inclusion $W_{-3} \mt \subset W_{-3} \para$ is an equality;
\item[(ii)] the inclusion $\M/ \mon \cap \M \subset \pp/\mon$ is an equality;
\item[(iii)] the inclusion $\mt^{p,q} \subset \para^{p,q}$ is an equality for every $(p,q) \in \Z^2$ such that $p+q \leqslant 0$ and $(p,q) \notin \{(-1,1), (1,-1), (0,0), (0,-1), (-1, 0), (-1, -1)\}$.
\end{itemize}
\end{proposition}
\begin{proof}
Let $h \in \I$ be a lift of the image under $\Phi$ of some Hodge generic point of $S^\an$. As $\pp(\R)^+ \pp^u(\C)$ acts transitively on $D_\pp$, there exists an element $g \in \pp(\R)^+ \pp^u(\C)$ such that $h = g \cdot z$ hence $h \in g \cdot D_\M \cap \I$. Since $(\M,D_\M)$ is $\V$-likely, Proposition \ref{suf} ensures that there exists a complex irreducible component of $(g \cdot D_\M) \cap \I$ containing $x$ along which $g \cdot D_\M$ and $\I$ intersect transversally. By Proposition \ref{typtotrans}, this implies that up to changing $(g,h)$ by a neighboring couple which we denote identically and which might be chosen so that $h$ remains Hodge generic, we have:
\[
T_h (g \cdot D_\M) + T_h \I = T_h D_\pp.
\]
Let $\Ad : \pp_\C \rightarrow \gl(\para_\C)$ be the adjoint representation. The bigraduation $\intt_{g} \circ x$ satisfies Pink's axioms and defines a bigraduation of $h \in D_\pp$. Therefore, the tangent space $T_h D_\pp$ naturally identifies to \[\bigoplus_{p+q \leqslant 0, p<0} \Ad(g)(\para^{p,q}).\] Under this identification $T_h (g \cdot D_\M)$ identifies to \[\bigoplus_{p+q \leqslant 0, p<0} \Ad(g)(\mt^{p,q}),\] and one has \[T_h \I \subset \bigoplus_{q \leqslant 1}\Ad(g)(\para^{-1,q})\] by Griffiths transversality. In particular, this implies that for $p+q \leqslant -1$ and $p < -1$ the inclusion $\Ad(g)(\mt^{p,q}) \subset \Ad(g)(\para^{p,q})$ is an equality hence $h_\mt^{p,q} = h_\para^{p,q}$ as $\Ad(g)$ preserves complex dimensions. In turn, one finds that for $p+q \leqslant -1$ and $p<-1$, the inclusion $\mt^{p,q} \subset \para^{p,q}$ is an equality. This implies both assertions $(i)$ and $(iii)$ as we shall now explain.

For assertion $(i)$, let $p,q \in \Z$ such that $p+q \leqslant -3$. If $p < -1$, there is nothing to prove. Assume that $p \geqslant -1$, which implies $q \leqslant -2$. By Lemma \ref{hsymnum}, Hodge symmetry holds at the level of Hodge numbers, hence the equality $h_{\mt}^{q,p} = h_{\para}^{q,p}$ (which holds because $q < -1$) implies that $h_\mt^{p,q} = h_{\para}^{p,q}$ which in turn implies $\mt^{p,q} = \para^{p,q}$ as desired. By Equation \ref{weight}, this implies that $W_{-3} \lie(\pp) = W_{-3} \lie(\M)$ as desired.

For assertion $(iii)$, let $p,q \in \Z$ such that $p+q \leqslant 0$. If $p+q \leqslant -3$, the statement is already a consequence of $(i)$. Otherwise, by Lemma \ref{hsymnum} one has $h_\mt^{p,q} = h_{\para}^{p,q}$ as soon as $q<-1$ or $p<-1$.  Assertion $(iii)$ follows.

We are left with proving assertion $(ii)$. Applying the $\V$-likeliness condition relative to the algebraic monodromy group $\mon \in \Np$, one finds that 
\[
\dim \Phi_{/\mon}(S^\an) + \dim q_\mon(D_\M) \geqslant \dim D_\pp/\mon.
\]
By Lemma \ref{monorb}, one has $\dim \Phi_{/\mon}(S^\an) = 0$ and the above inequality shows that $q_\mon(D_\M) = D_\pp/\mon$. Let $h \in D_\pp$ such that $\pp(h) = \pp$, for instance one can take the lift in $\I$ of the image under $\Phi$ of a Hodge generic point in $S^\an$. Then $\pp(q_\mon(h)) = \pp/\mon$. Since $q_\mon(D_\M) = D_\pp/\mon$ it follows that $q_\mon(h) \in q_\mon(D_\M)$ hence $\pp(q_\mon(h)) \subset \M/(\M \cap \mon)$.  The other inclusion holds trivially hence $\pp/\mon = \M/\M \cap \mon$.
\end{proof}
\section{The transverse mixed Hodge locus of prescribed type}\label{sect8}
\subsection{All-or-nothing}
The aim of this subsection is to prove the following:
\begin{theorem}\label{allornothing}
Let $\V$ be an admissible integral and graded-polarized variation of mixed Hodge structures on a smooth irreducible quasi-projective complex algebraic variety $S$. Let $(\M,D_\M)$ be a strict mixed Hodge subclass of the generic mixed Hodge class $(\pp,D)$ of $\V$. Assume that $\HLM_\trans$ is non-empty. Then $\HLM_\trans$ is dense in $S$ for the Zariski topology.
\end{theorem}
\begin{proof}
Zariski-density of $\HLM_\trans$ remains unchanged if one pulls-back $(S,\V)$ along finite etale covers or open immersions. Besides, in the proof of Proposition \ref{pbeff}, the only step at which one needs replacing $S$ by a non-empty Zariski open subset of $S$ is to restrict to the smooth locus of $\Phi(S^\an)$ which preserves the non-emptiness of $\HLM_\trans$, since the image of a transverse special subvariety under the period map intersects the smooth locus of $\Phi(S^\an)$ by definition. Hence, we can assume in the proof that $\V$ is effective, which we do starting from now, and work in Setting \ref{setting}

Assume that $\HLM_\trans \neq \emptyset$. By Proposition \ref{typtotrans}, up to replacing $(\M,D_\M)$ by a rational conjugate (which doesn't change $\HLM_\trans$), we can assume that there exists an $(\M,D_\M)$-transverse special subvariety $Z$ of $S$ for $\V$ such that furthermore there exists a smooth point $x \in \Phi(Z^\an)$ and a $z \in \I$ such that $\pi_\Gamma(z) = x$ and $T_z D_\M + T_z \I = T_z D_\pp$ . Let $S_\M$ be the Zariski-closure of $\HLM_\trans$ in $S$. We want to prove that $S_\M = S$. 

We first claim that it suffices to prove that $\Phi(S_\M^\an) = \Phi(S^\an)$. Indeed assume for the sake of contradiction that the latter holds but $S_\M \subsetneq S$. Recall that $\Phi$ factors as $i^\an \circ f^\an$ with $f$ a dominant algebraic morphism from $S$ to a quasi-projective complex algebraic variety $Y$, and $i$ a definable immersion. Since $f(S_\M) = f(S)$ we know that $S_\M$ meets every fiber of $f$ over $f(S)$, and since $S_\M \subsetneq S$, generic flatness applied to $f \vert_{S_\M}$ ensures the existence of a dense Zariski-open subset $U \subset f(S_\M)$ such that $S_\M$ meets the fiber of $f$ over any point of $U$ properly. By definition of $S_\M$, the preimages by $i$ of transverse irreducible components of intersections of $\Phi(S^\an)$ with $\phi_{i, \Gamma}(\period_{(g\M g^{-1},g \cdot D_\M), \Gamma_{g\M g^{-1}}})$ (for $g$ varying in $\pp(\Q)^+$) are Zariski dense in $f(S_\M)$, hence we can pick one, say $X \subset f(S_\M)$, meeting $U$. But now by definition, the fiber over any point of $X \cap U$ is fully contained in $S_\M$ (because special subvarieties of $S$ are preimages under $\Phi$ of intersections of the period image with Hodge subvarieties) which gives the desired contradiction.

Let $x \in \Phi(Z^\an)$ be a smooth point, and $z \in \I$ such that $\pi_\Gamma(z) = x$. Since $\Phi(S^\an)$ is irreducible it suffices to prove that $T_x \Phi(S_\M^\an) = T_x \Phi(S^\an)$. Let $\I_\M = \pi_\Gamma^{-1}(\Phi(S_\M^\an)) \cap \I$, and $\Zcal \subset \I_\M$ a complex-analytic irreducible component of the intersection of $\I$ and $D_\M$ containing $z$, such that $\pi_\Gamma$ identifies the germ of $\Zcal$ at $z$ with the germ of $\Phi(Z^\an)$ at $x$. Since $\pi_\Gamma$ is a local biholomorphism, the equality $T_x \Phi(S_\M^\an) = T_x \Phi(S^\an)$ is equivalent to $T_z \I_\M = T_z \I$ in $T_z D_\pp$. Recall that $T_z D_\pp$ can be identified to $\lie(\pp_\C)/ F_z^0 \lie(\pp_\C)$. Denote by $(T_z D_\pp)_\R$ the image of $\lie(\pp_\R)$ under the quotient map $\lie(\pp_\C) \rightarrow \lie(\pp_\C) / F_z^0 \lie(\pp_\C)$. We will use the following, to be shown below.
\begin{lemma}\label{realsub}
$(T_z D_\pp)_\R \subset T_z \I_\M + T_z D_\M \subset T_z D_\pp$
\end{lemma}
Then, combining the facts that $(T_z D_\pp)_\R \oplus i (T_z D_\pp)_\R = T_z D_\pp$ and that $T_z \I_\M + T_z D_\M$ is a complex vector subspace of $T_z D_\pp$, one finds that $T_z \I_\M + T_z D_\M = T_z D_\pp$. On the other hand, by definition $\I$ and $D_\M$ intersect transversally along $\Zcal$, so that $T_z \I + T_z D_\M = T_z D_\pp$ and $T_z \I \cap T_z D_\M = T_z \Zcal$. But $\Zcal \subset \I_\M$ hence $T_z \Zcal \subset T_z \I_\M$, so that $T_z \I_\M \cap T_z D_\M = T_z \Zcal$.  Therefore, using Grassmann formula twice, one finds
\[
\dim T_z \I_\M  = \dim T_z D_\pp - \dim T_z D_\M + \dim T_z \Zcal = \dim T_z \I.
\]
This shows that $T_z \I_\M = T_z \I$ as desired.
\end{proof}
\begin{proof}[Proof of Lemma \ref{realsub}]
Let $\mu : \pp(\R)^+ \times D_\pp \rightarrow D_\pp, (p,x) \mapsto p \cdot x$ be the action map, and for any $x \in D_\pp$ denote by $\mu_x : \pp(\R)^+ \rightarrow D_\pp$ be the composition of $\mu$ with the real-analytic immersion $\pp(\R)^+ \hookrightarrow \pp(\R)^+ \times D_\pp, p \mapsto (p,x)$. By definition, one has that $(T_x D_\pp)_\R = \im(d_{1}\mu_x)$, where $1$ denotes the identity of $\pp(\R)^+$.

By assumption, the inclusion $i_1 : D_\M \hookrightarrow D_\pp$ is transverse to $\I$ at $z$. The map $i:= \mu\vert_{\pp(\R)^+ \times D_\M}: \pp(\R)^+ \times D_\M \rightarrow D_\pp$ is a smooth map such that for any $x \in D_\M$, one has $i(1,x) = i_1(x)$. Therefore, Theorem \ref{stabtrans} ensures the existence of an open neighborhood $\Omega$ of $1$ in $\pp(\R)^+$ and a smooth map $\sigma : \pp(\R)^+ \rightarrow D_\M$ such that $\sigma(1) = z$ and for any $g \in \Omega$, the map $i_g := i(g, \cdot) : D_\M \hookrightarrow D_\pp$ is transverse to $\I$ at $\sigma(g)$.

Let $\mathscr{X} = i^{-1}(\I) \cap \pi^{-1}(\Omega)$ where $\pi : \pp(\R)^+ \times D_\M \rightarrow \pp(\R)^+$ is the projection map. We still denote by $\pi$ the restriction of the projection to $\mathscr{X}$ and we denote by $f : \mathscr{X} \rightarrow \I \subset D_\pp$ the restriction to $\mathscr{X}$ of $i$. These objects fit in the following diagram of smooth manifolds.
\[
\xymatrix{\mathscr{X} \ar[r]^f \ar[d]_\pi & \I \subset D_\pp \\ \pp(\R)^+ & }.
\]
We claim that the following two properties are satisfied:
\begin{itemize}
\item[$(1)$] the smooth map $\sigma_\M : \Omega \rightarrow D_\pp$ which assigns to a $g \in \Omega$ the point $\mu(g^{-1}, \sigma(g))$ has image in $D_\M$;
\item[$(2)$] the smooth map $\sigma$ has image in $\I_\M \subset \I$.
\end{itemize}
Property $(1)$ is simply a consequence of the definition of $\mathscr{X}$. Indeed for any $(g,x) \in \mathscr{X}$, one has $x \in g \cdot D_\M \cap \I$ hence $\mu(g^{-1},x) \in D_\M$. Furthermore, $\sigma_\M$ is a smooth map as $\mu$, $\sigma$ and the inversion in $\pp(\R)^+$ are. To show property $(2)$, note that by construction $g \cdot D_\M$ and $\I$ intersect transversally at $\sigma(g)$ for any $g \in \Omega$. Hence, by Lemma \ref{equivdifftrans}, for any $g \in \Omega \cap \pp(\Q)^+$ one has $\pi_\Gamma(\sigma(g)) \in \Phi(S_\M^\an)$ and therefore $\sigma(g) \in \I_\M$. Since $\I_\M$ is locally closed and $\Omega \cap \pp(\Q)^+$ is dense in $\Omega$, it follows that $\sigma(g) \in \I_\M$ for any $g \in \Omega$ as claimed in $(2)$.

We are now in position to conclude the proof. By definition of $\sigma_\M$, we have $\sigma(g) = \mu(g, \sigma_\M(g))$ for any $g \in \Omega$. Applying the chain rule, we therefore have:
\[
d_1 \sigma = (d_{(1,z)}\mu) \circ (\id \oplus d_1 \sigma_\M) = d_1 \mu_z + d_1 \sigma_\M.
\]
Therefore, 
\[
(T_z D_\pp)_\R = \im(d_1 \mu_z) \subset \im(d_1 \sigma - d_1 \sigma_\M) \subset T_z \I_\M + T_z D_\M \subset T_z D_\pp,
\]
as desired.
\end{proof}
\begin{corollary}\label{necplus}
Assume that $\HLM_\trans$ is non-empty. Then $(\M,D_\M)$ is $\V$-likely.
\end{corollary}
\begin{proof}
Since $\HLM_\trans$ is non-empty, Theorem \ref{allornothing} ensures that $\HLM_\trans$ is Zariski-dense in $S$. Therefore, there exists a $g \in \pp(\Q)^+$ and a $(g \M g^{-1}, g \cdot D_\M)$-transverse special subvariety $Z$ of $S$ for $\V$ such that $Z \cap S_\gen \neq \emptyset$. By Proposition \ref{nec}, this forces $(g \M g^{-1}, g \cdot D_\M)$ to be $\V$-likely, which in turn implies that $(\M,D_\M)$ is $\V$-likely.
\end{proof}
Finally, we can readily prove Corollary \ref{aontypos}:
\begin{proof}[Proof of Corollary \ref{aontypos}]
Assume that $\HL_{\typ, \pos}$ is non-empty. In particular there exists a strict mixed Hodge subclass $(\M, D_\M)$ of $(\pp, D_\pp)$ such that $\HLM_\trans$ is non-empty and:
\[
\dim D_\M + \dim \Phi(S^\an) - \dim D_\pp > 0.
\]
By Theorem \ref{allornothing}, the non-emptiness of $\HLM_\trans$ implies that $\HLM_\trans$ is dense in $S$ for the Zariski-topology. Besides the dimension inequality above implies that \[\HLM_\trans \subset \HL_\pos.\] Hence we have shown that $\HL_\pos$ is dense in $S$ for the euclidean topology. On the other hand, Corollary \ref{zpcor} implies that $\HL_{\atyp, \mathrm{pos}}$ is not dense in $S$ for the Zariski topology since $\HL_{\atyp, \mathrm{pos}}$ is contained in $\HL_{\atyp, \mathrm{f-pos}}$ by definition. Combining the above two facts, the disjunction $\HL_{\mathrm{pos}} = \HL_{\typ,\mathrm{pos}} \cup \HL_{\atyp,\mathrm{pos}}$ forces $\HL_{\typ, \mathrm{pos}}$ to be dense in $S$ for the Zariski topology, as desired.
\end{proof}
\subsection{Analytic density}
Let $(\pp, D_\pp)$ be a mixed Hodge class and denote by $W_\bul \lie(\pp)$ the induced weight filtration on $\lie(\pp)$.
\begin{definition}
Let $(\M,D_\M)$ be a strict mixed Hodge subclass of $(\pp, D_\pp)$. It is called \textit{sufficiently unipotent} if $\exp(W_{-2}\lie(\pp)) \subset \M$.
\end{definition}
The usefulness of this a priori obscure property lies in the following:
\begin{lemma}\label{sufuncrit}
Let $(\M,D_\M)$ be a strict mixed Hodge subclass of $(\pp,D_\pp)$ which is sufficiently unipotent. Then $\pp(\R)^+\cdot D_\M = D_\pp$.
\end{lemma}
\begin{proof}
Choose $(\pp, X_\pp, D_\pp)$ and $(\M, X_\M, D_\M)$ representatives of $(\pp, D_\pp)$ and $(\M,D_\M)$ such that the inclusion of $\M$ in $\pp$ makes $(\M,X_\M, D_\M)$ a mixed Hodge subdatum of $(\pp, X_\pp, D_\pp)$. Let $\uu' := \exp(W_{-2}\lie(\pp))$. Let $X_\pp'$ be the set of $x \in X_\pp$ such that the post-composite of $x$ with the complexification of the quotient map $\pp \rightarrow \pp/\uu'$ is defined over $\R$. By \cite[Prop. 1.16]{pinkthese}, we have that $X_\pp'$ is a non-empty $\pp(\R)\uu'(\C)$-orbit in $X_\pp$ such that $\psi_{X_\pp}(X_\pp') = \mathcal{D}_\pp$. Since $\psi_{X_\pp}$ is $\pp(\R)\pp^u(\C)$-equivariant, this proves that \[\pp(\R)^+\uu'(\C) = (\pp(\R)^+\pp^u(\C)) \cap (\pp(\R)\uu'(\C)) \subset \pp(\R)^+\pp^u(\C)\] acts transitively on $D_\pp$.

Assume that $(\M,D_\M)$ is sufficiently unipotent, that is $\uu' \subset \M$. By definition of $D_\pp$, one has the containment $\pp(\R)^+ D_\M \subset D_\pp$, so we are tasked with proving the inverse inclusion. Let $h \in D_\pp$. Since $\pp(\R)^+ \uu'(\C)$ acts transitively on $D_\pp$, there exists $h_\M \in D_\M$, $p \in \pp(\R)^+$ and $u \in \uu'(\C)$ such that $pu \cdot h_\M = h$. From weight considerations one has that $\uu' \subset \M$ is a unipotent normal subgroup of $\M$. Therefore it is contained in $\M^u$ and we find $u \cdot h_\M \in D_\M$. Therefore $h \in \pp(\R)^+ \cdot D_\M$ as desired.
\end{proof}
Then the main result of this subsection can be stated as follows:
\begin{theorem}\label{mainprescribed}
Let $\V$ be an admissible integral and graded-polarized variation of mixed Hodge structures on a smooth irreducible and quasi-projective complex algebraic variety $S$. Let $(\pp, D_\pp)$ be the generic Hodge class of $\V$.  Let $(\M, D_\M)$ be a strict mixed Hodge subclass of $(\pp, D_\pp)$ which is sufficiently unipotent. Then the following are equivalent:
\begin{enumerate}
\item[(i)] $\HLM_\trans$ is non-empty;
\item[(ii)] $\HLM_\trans$ is dense in $S^\an$ for the euclidean topology;
\item[(iii)] $(\M,D_\M)$ is $\V$-likely.
\end{enumerate}
\end{theorem}
\begin{proof}
Of course $(ii)$ implies $(i)$, and Corollary \ref{necplus} shows that $(i)$ implies $(iii)$, both without assuming that $(\M,D_\M)$ is sufficiently unipotent. Assume that $(\M,D_\M)$ is $\V$-likely and sufficiently unipotent. We are tasked with proving that  $\HLM_\trans$ is dense in $S^\an$ for the euclidean topology. As the assumptions and desired conclusion are unchanged after replacing $S$ by a non-empty Zariski-open subset or by a finite etale cover, we can assume that $\V$ is effective and work in Setting \ref{setting}, and assume that $S = S_\gen$ in the notation of Proposition \ref{indep}.  

Let $s \in S^\an$ be a Hodge generic point, and $z \in \I$ a lift of $\Phi(s)$ in $\I$. let $B_0 \subset S^\an$ be an open ball containing $s$ and $B \subset \I$ be an open ball containing $z$ such that $\Phi^{-1}(\pi_\Gamma(B))$ (which is an open neighborhood of $s$ in $S^\an$ for the euclidean topology) is contained in $B_0$. Let $i : \pp(\R)^+ \times D_\M \rightarrow D_\pp$ the restriction of the action map, which is a smooth map between smooth manifolds. For $p \in \pp(\R)^+$, we denote by $i_p : D_\M \rightarrow D_\pp$ mapping $x \in D_\M$ to $i(p,x)$. We are tasked with finding an element $p \in \pp(\Q)^+$ such that $i_p$ is transverse to $\I$ at some point $x \in D_\M$ such that $i_p(x) \in B$. 

By Lemma \ref{sufuncrit}, there exists a $p_0 \in \pp(\R)^+$ such that $z \in p_0 \cdot D_\M$. Let $U$ be a complex-analytic irreducible component of $\I \cap (p_0 \cdot D_\M)$ containing $z$. By Proposition \ref{suf}, one has
\begin{equation}\label{eqdim}
\dim U = \dim (p_0 \cdot D_\M) + \dim \I - \dim D_\pp.
\end{equation}
By Proposition \ref{typtotrans}, this implies that up to replacing $(z, p_0)$ by a neighboring couple in $\I \times \pp(\R)^+$ still satisfying $z \in p_0 \cdot D_\M \cap \I$, the inclusion $i_{p_0}$ is transverse to $\I$ at $z_\M := p_0^{-1} \cdot z$. Theorem \ref{stabtrans} then ensures the existence of an open neighborhood $\Omega$ of $p_0$ in $\pp(\R)^+$ and a smooth map $\sigma : \Omega \rightarrow D_\M$ such that $\sigma(p_0) = z_\M$ and for any $p \in \Omega$, the map $i_p$ is transverse to $\I$ at $\sigma(p)$. Appealing to Lemma \ref{equivdifftrans} one finds that for any $p \in \Omega$ the translate $p \cdot D_\M$ and $\I$ intersect transversally in $D_\pp$ along some complex-analytic irreducible component $U_p$ containing $i_p(\sigma(p))$.

Combining the facts that $\sigma$ and $i$ are smooth, that $i(p_0,\sigma(p_0)) = z$ and that $\pp(\Q)^+$ is analytically dense in $\pp(\R)^+$, we conclude to the existence of a $p \in \pp(\Q)^+ \cap \Omega$ such that $i(p, \sigma(p)) \in B$. By construction, we then have a $(p \M p^{-1}, p \cdot D_\M)$-transverse special subvariety of $S$ for $\V$ meeting $B_0$. As Hodge generic points are analytically dense in $S^\an$, this proves that $\HLM_\trans$ is analytically dense in $S^\an$, as desired.
\end{proof}
In particular this yields:
\begin{proof}[Proof of Theorem \ref{thm1}]
Assume that $\gr^W_{-2} \lie(\pp) = 0$ which is equivalent to $W_{-2} \lie(\pp) = W_{-3} \lie(\pp)$. By Proposition \ref{likconstr}$(i)$ one knows that any $\V$-likely strict mixed Hodge subclass $(\M, D_\M) \subsetneq (\pp, D_\pp)$ satisfies $W_{-3} \lie(\M) = W_{-3} \lie(\pp)$, which under the above assumption ensures that $\exp(W_{-2} \lie(\pp)) \subset \M$. Therefore, any $\V$-likely strict mixed Hodge subclass of $(\pp, D_\pp)$ is sufficiently unipotent and Theorem \ref{mainprescribed} finishes the proof.
\end{proof}
\section{The full transverse mixed Hodge locus}\label{sect9}
As demonstrated in Propositions \ref{contrexc} and \ref{contrexc1}, the $\V$-likeliness of a mixed Hodge subclass $(\M, D_\M)$ is not sufficient in general to guarantee the density of $\HLM_\trans$ or $\HL_\trans$ in $S$ for the euclidean topology. However, we can still show:
\begin{theorem}\label{mainfull}
Let $\V$ be an admissible and graded-polarized integral variation of mixed Hodge structures on a smooth and irreducible quasi-projective complex algebraic variety $S$. Let $(\pp, D_\pp)$ be the generic mixed Hodge class of $\V$. Then the following assertions are equivalent:
\begin{itemize}
\item[(i)] $\HL_\trans$ is non-empty;
\item[(ii)] $\HL_\trans$ is dense in $S$ for the Zariski topology;
\item[(iii)] There exists a strict $\V$-likely mixed Hodge subclass $(\M, D_\M) \subsetneq (\pp, D_\pp)$.
\end{itemize}
\end{theorem}
The following notation will be convenient:
\begin{notation}
Let $(\pp, D_\pp)$ be a mixed Hodge class and $(\M, D_\M)$ be a strict mixed Hodge subclass, and let $W_\bul \lie(\M)$ be the adjoint weight filtration of $\lie(\M)$. We denote by $k_\pp(\M)$ the largest non-positive integer $k$ such that $\Gr^{W}_{k} \lie(\M) \subsetneq \Gr^{W}_{k} \lie(\pp)$. This is well defined because $(\M,D_\M)$ is a \textit{strict} mixed Hodge subclass of $(\pp, D_\pp)$.
\end{notation}
\begin{lemma}\label{minusone}
Let $(\M,D_\M) \subsetneq (\pp, D_\pp)$ be a strict mixed Hodge subclass which is $\V$-likely and satisfies $k_\pp(\M) \geqslant -1$. Then, there exists a strict mixed Hodge subclass $(\LL, D_\LL) \subsetneq (\pp, D_\pp)$ which is $\V$-likely and sufficiently unipotent.
\end{lemma}
\begin{proof}
Assume that $k_\pp(\M) = 0$, which means that the inclusion $\M/\M^u \hookrightarrow \pp/\pp^u$ is strict. Choose a Levi subgroup $\G_\M$ of $\M$. Then $\G_\M$ is a closed reductive subgroup of $\pp$, which normalizes $\pp^u$. Let $h \in D_\M$. Since the connected $\Q$-subgroup $\LL := \pp^u \rtimes \G_\M$ of $\pp$ contains $\M$, there is, by Corollary \ref{mtconn} applied to $X = \{h\}$, an intermediate strict mixed Hodge subclass $(\M, D_\M) \subset (\LL,D_\LL) \subsetneq (\pp, D_\pp)$ with underlying inclusion the inclusion of $\LL$ in $\pp$. Since $\LL^u = \pp^u$, one finds that $(\LL, D_\LL)$ is sufficiently unipotent. Besides $(\LL, D_\LL)$ is also $\V$-likely as it contains $(\M, D_\M)$ which is $\V$-likely. 

Assume that $k_\pp(\M) = -1$. In particular $\M/\M^u = \pp/\pp^u$ hence there is a connected $\Q$-subgroup $\G \subset \M$ which is a Levi subgroup of both $\M$ and $\pp$. Denote by $\lieu = \lie(\pp^u)$ and $\lieu_\M = \lie(\M^u)$. Let $\lieu_\M^{-1}$ be a vector space supplement of $W_{-2}(\lieu_\M)$ in $\lieu_\M$, and let $\lieu_\LL = W_{-2}(\lieu) \oplus \lieu_\M^{-1}$. It immediately follows from weight considerations that $\lieu_\LL$ is a Lie subalgebra of $\lieu$, and from the fact that $\lieu_\M$ and $W_{-2} \lieu$ are stable under $\Ad(\G)$ (seen as a subgroup of $\gl(\para)$) that $\lieu_\LL$ is stable under $\Ad(\G)$. Hence $\uu_\LL := \exp(\lieu_\LL)$ is a subgroup of $\pp^u$ normalized by $\G$ and containing $\M^u$. Hence, the connected $\Q$-subgroup $\LL := \uu_\LL \rtimes \G$ of $\pp$ is well defined, and contains $\M$ so that there is, by Corollary \ref{mtconn} applied to $X = \{h\}$, an intermediate strict mixed Hodge subclass $(\M, D_\M) \subset (\LL,D_\LL) \subsetneq (\pp, D_\pp)$ with underlying inclusion the inclusion of $\LL$ in $\pp$. By construction $(\LL,D_\LL)$ is sufficiently unipotent. Besides $(\LL, D_\LL)$ is also $\V$-likely as it contains $(\M, D_\M)$ which is $\V$-likely. 
\end{proof}
We can now turn to the:
\begin{proof}[Proof of Theorem \ref{mainfull}]
The implication $(ii) \Rightarrow (i)$ is obvious and the implication $(i) \Rightarrow (iii)$ has been settled in Corollary \ref{necplus}. It remains to prove that $(iii) \Rightarrow (ii)$. Let $(\M, D_\M) \subsetneq (\pp, D_\pp)$ be a strict mixed Hodge subclass which is $\V$-likely. We use the notations in Setting \ref{settingext}.

We first claim that $k_\pp(\M) \geqslant -2$. Indeed, by Proposition \ref{likconstr}(i), we have $W_{-3} \lie(\pp) = W_{-3} \lie(\M)$. In particular, for any $k < -2$ the inclusion $\gr_W^k \lie(\M) \subset \gr_W^k \lie(\pp)$ is an equality. By definition, this means that $k_\pp(\M) \geqslant -2$.

If $k_\pp(\M) \geqslant -1$ there exists by Lemma \ref{minusone} a strict mixed Hodge subclass $(\LL,D_\LL) \subsetneq (\pp, D_\pp)$ which is $\V$-likely and sufficiently unipotent, and Theorem \ref{mainprescribed} shows that $\HL_\trans$ is dense in $S(\C)$ for the euclidean topology hence in particular for the Zariski topology. Therefore, we can assume that $k_\pp(\M) = -2$.


By construction $(\M,D_\M)$ is a strict mixed Hodge subclass of $(\pp, D_\pp)$, is $\V$-likely and satisfies $k_\pp(\M) = -2$. By the assumptions that $k_\pp(\M) = -2$ it immediately follows that for $p+q \geqslant -1$ the inclusion $\mt^{p,q} \subset \para^{p,q}$ is an equality. Proposition \ref{likconstr}$(iii)$ then implies that for all $(p,q) \neq (-1,-1)$, the inclusion $\mt^{p,q} \subset \para^{p,q}$ is an equality. Let $\NN = \M^u \rtimes \G^\der$. It is clearly an element of $\Np$. The $\V$-likeliness condition for $\NN$ then shoes that
\[
\dim \Phi_{/\NN}(S^\an) + \dim q_\NN(D_\M) \geqslant \dim D_\pp/\NN.
\]
By construction $D_\M$ is a fiber of $q_\NN$ hence $\dim q_\NN(D_\M) = 0$ and since $\dim \Phi_{/\NN}(S^\an) \leqslant \dim D_\pp/\NN$, one finds that $\Phi_{/\NN}$ is dominant, and has positive image because $\mt^{-1,-1} \subsetneq \para^{-1,-1}$ (otherwise we would have $\M = \pp$).

To conclude the proof, it suffices to show that the smooth locus of $\Phi_{/\NN}(S^\an)$ contains a point of the form $\phi_{i, \Gamma_{/\NN}}(\period_{(\mathbf{T}, D_{\mathbf{T}}), (\Gamma_{/\NN})_\mathbf{T}})$ for some strict torsion mixed Hodge subclass $(\mathbf{T}, D_{\mathbf{T}})$ of $(\pp/\NN, D_\pp/\NN)$. Indeed $\Phi_{\NN}$ being dominant, such a point would then be a transverse intersection (not contained in the singular locus) of $\Phi_{/\NN}(S^\an)$ with some Hodge subvariety of $\period_{(\pp/\NN, D_\pp/\NN), \Gamma_{/\NN}}$. It is then a matter of unwrapping definitions to see that an irreducible component of the preimage of such a point by $\Phi_{/\NN}$ produces a transverse special subvariety of $S$ for $\V$, hence showing that $\HL_\trans \neq \emptyset$ and Theorem \ref{allornothing} would imply the sought for density in the Zariski topology.

By Proposition \ref{torsionex}, there exists a strict torsion mixed Hodge subclass $(\mathbf{T}, D_{\mathbf{T}})$ of $(\pp/\NN, D_\pp/\NN)$. Let $\pi_{\Gamma_{/\NN}} : D_\pp/\NN \rightarrow \period_{(\pp/\NN, D_\pp/\NN), \Gamma_{/\NN}}$ be the quotient by $\Gamma_{/\NN}$ and $h_{\mathbf{T}} \in D_{\mathbf{T}}$. We need to show that no complex-analytic subvariety $Y$ of $\period_{(\pp/\NN, D_\pp/\NN), \Gamma_{/\NN}}$ of strictly smaller dimension contains $\pi_{\Gamma_{/\NN}}((\pp/\NN)(\Q)^+ \cdot h_{\mathbf{T}})$. This can be shown exactly as in the proof of Theorem \ref{allornothing}, which only uses the complex-analytic topology: if such a $Y$ existed, its tangent space at a smooth point $y$ would be forced to contain a real form of $T_y \period_{(\pp/\NN, D_\pp/\NN), \Gamma_{/\NN}}$ hence to fill the tangent space as it is a complex vector subspace, and this would contradict the assumption that $Y$ has strictly smaller dimension.
\end{proof}
\section{Emptiness of the transverse Hodge locus}\label{sect10}
Let $\V$ be an admissible integral and graded-polarized variation of mixed Hodge structures on a smooth and irreducible quasi-projective complex algebraic variety $S$. Fix a base point $s \in S^\an$ and let $(\pp, X_\pp, D_\pp)$ be the generic mixed Hodge datum of $\V$ at $s$, let $\mon = \mon_s$ be the algebraic monodromy group of $\V$ at $s$. Let $V = \V_s$ and $W_\bul = (\W_\bul)_s$. We see both $\pp$ and $\mon$ as subgroups of the subgroup $\pp_{W_\bul}$ of $\gl(V)$ preserving the filtration $W_\bul$. Denote by $\rho : \pi_1(S,s) \rightarrow \pp_{W_\bul}(\Q)$ the monodromy representation associated to the local system $\V$ at $s$. Denote by $\pp_0$ and $\mon_0$ the quotients of $\pp$ and $\mon$ by their unipotent radicals, and denote by $\pi_0 : \pp \rightarrow \pp_0$ the quotient map.

By Theorem \ref{andre}, the group $\mon$ is a normal subgroup of $\pp$ hence $\ort = \lie(\mon)$ is a mixed $\Q$-Hodge substructure of $\para = \lie(\pp)$. By Pink's axioms $\Ad(\pp)$ respects the weight filtration and $\Ad(\pp^u)$ makes the weight decrease of at least $1$ hence for every integer $k \leqslant 0$ we get a representation $R_k : \pp_0 \rightarrow \gl(W_k\ort / W_{k-1} \ort)$ from the adjoint representation of $\pp$.
\begin{definition}
The \textit{pure level of $\V$}, denoted $\mathrm{pl}(\V)$ is the level of $\gr(\V)$ in the sense of \cite[Def. 4.15]{bku}.
\end{definition}
\begin{theorem}\label{alternativempty}
Assume that $\mathrm{pl}(\V) \geqslant 3$ and $\pp_0^\der = \mon_0$. If $\HL_\trans \neq \emptyset$, there exists, for any transverse special subvariety $Z$ of $S$ for $\V$, a strict normal subgroup $\NN \in \Np$ of $\pp$ such that:
\begin{itemize}
\item $\NN/\NN^u = \pp_0^\der$;
\item for any $h \in D_\pp$, the natural mixed Hodge structure induced on $\lie(\pp/\NN)$ is of type $\{(-1,-1), (0,0)\}$;
\item the quotient period map $\Phi_{/\NN} : S^\an \rightarrow \period_{(\pp/\NN,D_\pp/\NN), \Gamma_{/\N}}$ is dominant with positive dimensional image;
\item $Z$ is an irreducible component of the fiber of $\Phi_{/\NN}$ over a torsion point in $\period_{(\pp/\NN,D_\pp/\NN), \Gamma_{/\N}}$.
\end{itemize}
\end{theorem}
In particular, one obtains a criterion for emptiness of $\V$:
\begin{theorem}\label{emptymain}
Assume that $\mathrm{pl}(\V) \geqslant 3$ and $\pp_0^\der = \mon_0$. If furthermore one of the two following conditions hold
\begin{itemize}
\item[(A)] The trivial representation of $\pp_0^\der$ does not occur in $R_{-2}\vert_{\pp_0^\der}$;
\item[(B)] There is no dominant regular morphism $f : S \rightarrow \mult_m$.
\end{itemize}
Then $\HL_\trans$ (hence $\HL_\typ$) is empty.
\end{theorem}
Let us comment on the assumptions $(A)$ and $(B)$. They are both intended to exclude the possibility of adding direct factors which are extensions in the category of integral variations of mixed Hodge structure of $\underline{\Z}_S(0)$ by $\underline{\Z}_S(1)^n$ for some integer $n \geqslant 1$. For instance, assume that $\V$ satisfies $\mathrm{pl}(\V) \geqslant 3$ and $\pp_0^\der = \mon_0$ but $(B)$ fails. Assume that there exists a dominant regular morphism $f : S \rightarrow \mult_m$. Let $\mathbb{E}$ be the extension of $\Z_S(0)$ by $\Z_S(1)$ with period map $f$. Set $\V' := \V \oplus \mathbb{E}$ and use all notations defined at the beginning of the section for $\V$ with prime superscrips for $\V'$. Then $\mathrm{pl}(\V') \geqslant 3$ and $(\pp'_0)^\der = \mon_0$. However $\mathrm{HL}(S, (\V')^\otimes)_\trans$ is non-empty due to the direct factor $\mathbb{E}$ which has Zariski-dense transverse Hodge locus. Because of this factor, the trivial representation does occur as $\pp_0^\der = (\pp'_0)^\der$ acts trivially on the unipotent part of the monodromy of $\mathbb{E}$, i.e. $(A)$ fails. We do not know if such a situation can actually occur. In any case, the above result holds in the following interesting situations:
\begin{corollary}\label{critnum}
Assume that $\mathrm{pl}(\V) \geqslant 3$, that $\pp_0^\der = \mon_0$ and that $h^{-1,-1}_{\lie(\pp)} = 0$. Then $\HL_\trans$ is empty.
\end{corollary}
\begin{proof}
If $h^{-1,-1}_{\lie(\pp)} = 0$ condition $(A)$ of Theorem \ref{emptymain} is automatically satisfied because as shown in Lemma \ref{gradnontriv}, under $R_{-2}$ the group $\pp_0^\der$ acts non-trivially on $(\gr_{-2}^W \lie(\pp))^{p,-2 - p}$ for $p \in \Z -\{-1\}$. Hence the result.
\end{proof}
\begin{corollary}\label{critalg}
Assume that $\mathrm{pl}(\V) \geqslant 3$, that $\pp_0^\der = \mon_0$ and that $S$ has a smooth compactification $\overline{S}$ such that $\overline{S}-S$ is either empty, of codimension at least $2$ or a smooth irreducible divisor in $\overline{S}$. Then $\HL_\trans$ is empty.
\end{corollary}
\begin{proof}
Under one of these assumptions condition $(B)$ of Theorem \ref{emptymain} is automatically satisfied. Hence the result.
\end{proof}

\subsection{A grading element}\label{grading}
In this subsection, we assume that $\V$ is effective. Denote by $\gr(\V) = \bigoplus_{k \in \Z} \W_k/\W_{k-1}$ the graded of $\V$ for $\W_\bul$, which carries by definition a pure polarized integral variation of Hodge structures. Let $\gr(V) = \gr(\V)_s$. By Lemma \ref{subgroup}, the groups $\pp_0$ and $\mon_0$ are naturally subgroups of $\gl(\gr(V)) = \pp_{W_\bul}/(\pp_{W_\bul})^u$.
\begin{lemma}\label{monpure}
The (connected) $\Q$-algebraic subgroup $\mon_0$ of $\gl(\gr(V))$ is the algebraic monodromy group of $\gr(\V)$ at $s$.
\end{lemma}
\begin{proof}
Since we assumed $\V$ to be effective, the image of $\rho$ is contained in $\mon(\Q)$ hence $\mon$ is the smallest $\Q$-algebraic subgroup of $\pp_{W_\bul}$ whose $\Q$-points contain $\im(\rho)$. By definition, the monodromy representation of $\gr(\V)$ at $s$ is the composite $\rho_{\gr(\V)} = \pi \circ \rho : \pi_1(S,s) \rightarrow \gl(\gr(V))$ where $\pi : \pp_{W_\bul} \rightarrow \gl(\gr(V))$ is the quotient by $(\pp_{W_\bul})^u$. Since $\mon$ is a normal subgroup in $\pp$ by Theorem \ref{andre} its unipotent radical $\mon^u$ is $\pp^u \cap \mon$ hence $\mon_0 = \pi(\mon)$. This readily implies that $\im(\rho_{\gr(\V)}) \subset \mon_0(\Q)$ hence $\mon_0$ contains the algebraic monodromy group of $\gr(\V)$. If the latter inclusion was not an equality, the image of $\rho_{\gr(\V)}$ would be contained in the $\Q$-points of a strict $\Q$-algebraic subgroup $\mon'_0$ of $\mon_0$, hence $\im(\rho)$ would be contained in the $\Q$-points of $\pi^{-1}(\mon'_0)$, contradicting the minimality of $\mon$ for this property since $\dim \pi^{-1}(\mon'_0)  > \dim \mon$.
\end{proof}
Let $x \in X_\pp$ be a bigraduation of the mixed Hodge structure on $\V_s$. Post-composition of $x$ with the adjoint representation defines a mixed $\Q$-Hodge structure on $\para = \lie(\pp)$ which by Pink's axioms has only non-positive weights. Since $\mon$ is a normal subgroup of $\pp$ by Theorem \ref{andre}, its Lie algebra $\ort = \lie(\mon)$ is a mixed sub-$\Q$-Hodge structure of $\para$, a bigraduation of which is given by $R_\C \circ x$ where $R_\C : \pp \rightarrow \gl(\ort)$ is the sub-representation of the adjoint representation associated with $\mon$. For $k \leqslant 0$ the $k$-th graded-piece for the weight filtration $\ort_k = W_k \ort/ W_{k-1} \ort$ carries a weight $k$ pure polarized $\Q$-Hodge structure whose Hodge decomposition we denote by
\[
\ort_{k, \C} = \bigoplus_{p \in \Z} \ort_k^{p,k-p}.
\]
Combining the facts that  $W_{-1} \ort = W_{-1} \para$, that $W_{-1} \para = \lie(\pp^u)$ and that $\mon^u = \mon \cap \pp^u$, one finds that $\ort_0 = \lie(\mon_0)$. This is an ideal in the derived Lie algebra $\para_0^\der$ of $\para_0$. 

By Pink's axiom $(i)$, the morphism $ (\pi_0)_\C \circ x: \s_\C \rightarrow \pp_{0,\C}$ is defined over $\R$ and we let $\bar{x} : \s \rightarrow \pp_{0, \R}$ be the corresponding morphism of real algebraic groups. Let $k \leqslant 0$ be an integer. The image of $\pp$ by $R$ preserves the weight filtration on $\ort$, hence $R$ descends to a representation $R_k : \pp_0  \rightarrow \gl(\ort_k)$, and the Hodge cocharacter associated with the pure Hodge structure on $\ort_k$ is given by $(R_k)_\R \circ \bar{x}$. Finally we let $r_k : \para_0 \rightarrow \End(\ort_k)$ be the derivative $d R_k : \para_0 \rightarrow \End(\ort_k)$. Explicitly $r_k$ can be described as the descent of the adjoint representation $W_0 \ort \rightarrow \End(W_k \ort)$ using that the bracket is a morphism of Hodge structures so that $[W_{-1} \ort , W_k \ort] \subset W_{k-1} \ort$. Note that by construction $r_k$ is a morphism of $\Q$-Hodge structures of weight $0$ and a morphism of Lie algebras, with $\End(\ort_k)$ endowed with the bracket induced by composition of endomorphisms.

Following a construction of Robles in \cite[Sect. 3]{robles}, we now define a semi-simple element of $\para_0^\der$ whose eigenspaces under the representations $r_k$ will be exactly the pieces of the Hodge decomposition of $\ort_k$. Let $S^1$ be the subgroup of $\s$ defined on real points as the unit circle in $\C^\times$, parametrized as $\{e^{\frac{it}{2}} \hspace{0.1cm} : \hspace{0.1cm} t \in \R\}$. We identify $\lie(S^1)$ with $\R$ using this parametrization. Let $\mathtt{T}_\V$ be the image of $1 \in \lie(S^1)$ under the derivative $d_0 \bar{x} \vert_{\lie(S^1)} : \R \rightarrow \para_{0, \R}$.
\begin{lemma}
The element $\mathtt{T}_\V$ belongs to the derived subalgebra $(\para_0^\der)_\R$ of $\para_{0,\R}$.
\end{lemma}
\begin{proof}
Since $\para_0^\der = \lie(\pp_0^\der)$, it suffices to show that \[\bar{x}(S^1) \subset \pp_{0, \R}^\der = \bigcap_{\chi : \pp_0 \rightarrow \mult_m} \ker{\chi_\R}\] where the intersection is over characters of $\pp_0$. Let $\chi : \pp_0 \rightarrow \gl(L)$ with $L$ a dimension $1$ rational vector space. Since $x$ satisfies Pink's axioms, the morphism $\chi_\R \circ \bar{x}$ defines a pure Hodge structure on $L$. As a pure Hodge structure of rank $1$ it must be of Hodge type $(l,l)$ for some $l \in \Z$. Hence, for $e^{\frac{it}{2}}$ in $S^1$ the automorphism $\chi_\R \circ \bar{x}(e^{\frac{it}{2}})$ of $L_\R$ must be $e^{\frac{itl}{2}} e^{-\frac{itl}{2}} \id_{L_\R} = \id_{L_\R}$. Therefore $\bar{x}(S^1) \subset \ker{\chi_\R}$ as desired.
\end{proof}
\begin{lemma}\label{gradnontriv}
Let $k,p \in \Z$ be integers with $k \leqslant 0$. The endomorphism $(r_k)_\C(\mathtt{T}_\V)$ leaves stable $\ort_k^{p,k-p}$ and \[ (r_k)_\C(\mathtt{T}_\V) \vert_{\ort_k^{p,k-p}} = \frac{k-2p}{2} \cdot \id_{\ort_k^{p,k-p}}.\]
\end{lemma}
\begin{proof}
Let $k,p \in \Z$ be integers with $k \leqslant 0$. For $t \in \R$, the complexification of the automorphism $(R_k)_\R \circ \bar{x}(e^{\frac{it}{2}})$ of $\ort_{k,\C}$ preserves $\ort_k^{p,k-p}$ and acts on it as $e^{\frac{(k-p)it}{2}}e^{-\frac{pit}{2}}\id_{\ort_k^{k,p-k}}$. Taking the differential at $t = 0$ and unwrapping definitions gives the result.
\end{proof}
\subsection{Emptiness in pure level at least $3$}
\begin{proof}[Proof of Theorems \ref{alternativempty} and \ref{emptymain}]
As explained in the beginning of the Proof of Theorem \ref{allornothing}, one can assume that $\V$ is effective without changing wether or not $\HL_\trans$ is empty. Therefore we assume in this proof that $\V$ is effective. Assume for the sake of contradiction that $\HL_\trans$ is non-empty and let $(\M, D_\M) \subsetneq (\pp, D_\pp)$ be a strict mixed Hodge subclass such that there exists an $(\M, D_\M)$-transverse special subvariety $Z$ of $S$ for $\V$. In particular $(\M, D_\M)$ is $\V$-likely. By Lemma \ref{typtotrans}, since $\pp(\Q)^+$ is dense in $\pp(\R)^+$ we may assume, up to replacing $(\M,D_\M)$ by a rational translate (which does not change $\V$-likeliness of $(\M,D_\M)$ nor non-emptiness of $\HLM_\trans$), that the inclusion $i : D_\M \hookrightarrow D_\pp$ is transverse to $\I$ at any lift to $\I$ of a smooth point of $\Phi(Z^\an)$. We will use freely the notations of Setting \ref{settingext} and established in Section \ref{grading}. 

\textit{Step 1: First reductions.} By Proposition \ref{likconstr}$(ii)$, the inclusion $\M/ \M \cap \mon \subset \pp/\mon$ is an equality. Let $\mt_\mon = \lie(\M \cap \mon)$. It is a mixed $\Q$-Hodge sub-structure of $\mon$. In particular, for $k \leqslant 0$ the graded-piece $\mt_{\mon,k} := W_k \mt_\mon/ W_{k-1} \mt_\mon$ is a pure $\Q$-Hodge sub-structure of $\ort_k$. By Proposition \ref{likconstr}$(i)$ the inclusion $\mt_{\mon,k} \subset \ort_k$ is an equality for every $k \leqslant -3$.

\textit{Step 2: $\mt_{\mon,0} = \ort_0$.} The argument is similar to \cite[Proof of Theorem 7.1]{bku} and we refer to \textit{op. cit.} for more details. By Lemma \ref{monpure} the Lie algebra $\ort_0$ is the Lie algebra of the algebraic monodromy group $\mon_0$ of $\gr(\V)$ which is a polarized variation of pure Hodge structures on the smooth and irreducible quasi-projective complex algebraic variety $S$. Therefore, by \cite[Prop. 7.4]{bku} each non-compact factor of $\ort_{0,\R}$ is generated in level $1$. By assumption, the $\Q$-Hodge Lie algebra $\ort_0$ has level at least $3$. Arguing as follows for each $\Q$-simple factor of $\ort_0$, one can assume that $\ort_0$ is $\Q$-simple. Let $\ortl_\R$ be a non-compact $\R$-simple factor of $\ort_{0,\R}$ which has level at least $3$. By Proposition \ref{likconstr}$(iii)$ one has that $h_\mt^{p,-p} = h_\para^{p,-p}$ for each $p \in \Z$ such that $|p| \geqslant -2$. Therefore, for $p \in \Z$ such that $|p| \geqslant -2$, one has that the inclusion $\mt_{\mon,0}^{p,-p} \cap \ortl_\C \subset \ortl^{p,-p}$ is an equality. Then \cite[Prop. 7.5]{bku} implies that $\ortl_\R \subset (\mt_{\mon,0})_\R$, which in turn forces $\mt_{\mon, 0} = \ort_0$.

\textit{Step 3: $\mt_{\mon,-1} = \ort_{-1}$.} Let $\lieu \subset \ort_{-1}$ be a (possibly zero) $\ort_0$-invariant (under $r_{-1}$) complement of $\mt_{\mon,-1}$ in $\ort_{-1}$ and $r_\lieu : \ort_0 \rightarrow \End(\lieu)$ the corresponding sub-representation. By Proposition \ref{likconstr}$(iii)$, one has $\lieu_\C \subset \ort_{-1}^{-1,0} \oplus \ort_{-1}^{0,-1}$. In particular, for every $p \in \Z$ such that $|p| \geqslant -2$ one has that $(r_{\lieu})_\C(\ort_0^{p,-p}) = \{0\}$. In other words,
\begin{equation}\label{incl}
\bigoplus_{p \in \Z, |p| \geqslant -2} \ort_0^{p,-p} \subset \ker(r_{\lieu}).
\end{equation}
Since $r_{\lieu}$ is a morphism of both Lie algebras and pure $\Q$-Hodge structures, one has that $\ker(r_{\lieu})$ is a $\Q$-Hodge-Lie subalgebra of $\ort_0$. Arguing as in the end of the last paragraph using \cite[Prop. 7.5]{bku} and the containment (\ref{incl}), one finds that $\ker(r_{\lieu}) = \ort_0$. In particular one finds that $(r_{-1})_\C(\mathtt{T}_\V)\vert_\lieu = 0$. By Lemma \ref{gradnontriv}, this forces $\lieu = 0$ as desired: the inclusion $\mt_{\mon,-1} \subset \ort_{-1}$ is an equality.

\textit{Step 4: End of proof.} Summing up, we have shown that for every $k \neq -2$, the inclusion $\mt_{\mon,k} \subset \ort_k$ is an equality. Using furthermore Proposition \ref{likconstr}$(iii)$, we find that for all $(p,q) \neq (-1,-1)$ the inclusion $\mt^{p,q} \subset \para^{p,q}$ is an equality. We now show that we are indeed in the situation described in Theorem \ref{alternativempty} and explain how this contradicts assumptions $(A)$ and $(B)$.

Repeating the proof of \textit{Step 3} one finds that $\para_0^\der = \ort_0$ acts trivially on a subrepresentation of $r_{-2} \vert_{\para_0^\der}$ defined as an $\ort_0$-invariant complement of $\mt_{\mon, -2}$ in $\ort_{-2}$. Since this complement must be non-zero (because otherwise one would have $\mt = \para$, hence $\M = \pp$ which was excluded) this shows that the trivial representation occurs in $r_{-2} \vert_{\para_0^\der}$ hence in $R_{-2} \vert_{\pp_0^\der}$. This contradicts assumption $(A)$.

On the other hand, since $\pp_0 = \M/\M^u$ there exists a common Levi subgroup $\G \subset \M$ to $\M$ and $\pp$. Let $\NN = \M^u \rtimes \G$ which is by construction a normal subgroup of $\pp$ whose radical is unipotent. The $\V$-likeliness condition of for $(\M, D_\M)$ and $\NN$ then shoes that
\[
\dim \Phi_{/\NN}(S^\an) + \dim q_\NN(D_\M) \geqslant \dim D_\pp/\NN.
\]
By construction $D_\M$ is a fiber of $q_\NN$ hence $\dim q_\NN(D_\M) = 0$ and since $\dim \Phi_{/\NN}(S^\an) \leqslant \dim D_\pp/\NN$, one finds that $\Phi_{/\NN}$ is dominant (with positive dimensional image by construction of $\NN$). By Theorem \ref{facto} it follows that the Hodge variety is the analytification of a quasi-projective algebraic variety and that $\Phi_{/\NN}$ is the analytification of a regular morphism between quasi-projective varieties. Since the Hodge structure on $\lie(\pp/\NN)$ is of type $\{(0,0), (-1,-1)\}$ by construction, the Hodge variety $\period_{(\pp/\NN, D_\pp/\NN), \Gamma_{/\NN}}$ is isomorphic as an algebraic variety to $\mult_m^n$ for some $n \geqslant 1$. Therefore, there exists a dominant regular morphism $f : S \rightarrow \mult_m^n$ such that $\Phi_{/\NN} = f^\an$. Postcomposition of $f$ with any projection $\mult_m^n \rightarrow \mult_m$ gives a dominant regular morphism $S \rightarrow \mult_m$, contradicting assumption $(B)$.

Finally, we have just shown that the group $\NN$, constructed under the assumption that $\HL_\trans$ is non-empty, indeed satisfies the properties in Theorem \ref{alternativempty} hence the claimed alternative. Besides, by construction $Z$ is the preimage under $\Phi_{/\NN}$ of some torsion point in $\period_{(\pp/\NN, D_\pp/\NN), \Gamma_{/\NN}} \cong \mult_m^n$.
\end{proof}
\subsection{Algebraicity of the Hodge locus}
From the emptiness criterion proved above, we can conclude, using Corollary \ref{zpcor} of the Geometric Zilber-Pink theorem, to a finiteness statement for the part of the Hodge locus that "does not come from special points". Let $\HL_{\mathrm{f-pos}}$ be the union of special subvarieties of $S$ for $\V$ which have $\V$-factorwise positive dimension. The following is complementary to \cite[Thm. 3.2]{ketay}.
\begin{theorem}\label{alglvl3}
Assume that $\pp_0^\der = \mon_0$ and $\mathrm{pl}(\V) \geqslant 3$ and at least one of the following holds
\begin{itemize}
\item[(A)] The trivial representation does not occur in $R_{-2}\vert_{\pp_0^\der}$;
\item[(B)] There is no dominant regular morphism $f : S \rightarrow \mult_m$.
\end{itemize} 
Then $\HL_{\mathrm{f-pos}}$ is a Zariski-closed subset of $S$ for $\V$.
\end{theorem}
\begin{proof}
Let $\{Z_1, \cdots, Z_l\}$ be the set of special subvarieties of $S$ for $\V$ as in the statement of Corollary \ref{zpcor}. Let $\mathscr{S}$ be the set subvarieties of $S$ which are special and have $\V$-factorwise positive dimension and are maximal for these properties with respect to the inclusion. We wish to prove that $\mathscr{S}$ is a finite set. 

Let $Z \in \mathscr{S}$. By Theorem \ref{emptymain}, we have $\HL_\typ = \emptyset$ so that $Z$ is an atypical special subvariety of $S$ with $\V$-factorwise positive dimension. By Corollary \ref{zpcor}, there exists a $k \in \{1, \cdots, l\}$ such that $Z \subset Z_k$. This forces $Z_k$ to have $\V$-factorwise positive dimension. By maximality of $Z$, it follows that $Z = Z_k$. We have shown that $\mathscr{S} \subset \{Z_1, \cdots, Z_l\}$ which implies the desired finiteness.
\end{proof}
\printbibliography
\bigskip
\noindent
\small{\textsc{Mathematics Department, Duke University, Box 90320, Durham, NC, 27708}\\
\textit{E-mail :} \texttt{nazim.khelifa@duke.edu}}
\end{document}